\documentclass{amsart}
\usepackage{amsmath}
\usepackage{paralist}
\usepackage[english]{babel}
\usepackage{amsfonts}
\usepackage{amscd}
\usepackage{amssymb}
\usepackage{amsmath}
\usepackage{amstext}
\usepackage{latexsym}
\usepackage{xspace}
\usepackage[all]{xy}
\usepackage[clean]{optional}
\usepackage{scalerel,stackengine}
\usepackage[colorlinks=true]{hyperref}

\usepackage{hyperref}

\newcommand{\baseRing}[1]{\ensuremath{\mathbb{#1}}}
\newcommand{\R}{\baseRing{R}}
\newcommand{\F}{\baseRing{F}}
\newcommand{\I}{\baseRing{I}}
\newcommand{\J}{\baseRing{J}}
\newcommand{\jgU}{\ensuremath{\mathfrak{U}}\xspace}
\newcommand{\jgg}{\ensuremath{\mathfrak{g}}\xspace}
\newcommand{\jgd}{\ensuremath{\mathfrak{d}}\xspace}
\newcommand{\jgso}{\ensuremath{\mathfrak{so}}\xspace}
\newcommand{\jdef}[1]{\index{#1}\emph{#1}}
\newcommand{\stext}[1]{\ensuremath{\quad\text{#1}\quad}}

\renewcommand{\iff}{\ensuremath{\Leftrightarrow}\xspace}

\newcommand{\DC}{\ensuremath{{\mathcal{A}_d}}}
\newcommand{\DCp}[1]{\ensuremath{{\mathcal{A}_d^{#1}}}}
\newcommand{\HL}{\ensuremath{h_d}}
\newcommand{\HLds}[1]{\ensuremath{\overline{\HL^{{#1}}}}}
\newcommand{\HLd}[1]{\ensuremath{{\HL^{{#1}}}}}
\newcommand{\lie}[1]{\ensuremath{Lie(#1)}}

\numberwithin{equation}{section}

\DeclareMathOperator{\rank}{rank}
\DeclareMathOperator{\Ob}{ob}
\DeclareMathOperator{\Mor}{mor}
\DeclareMathOperator{\Ad}{Ad}
\DeclareMathOperator{\tr}{Tr}
\DeclareMathOperator{\cay}{Cay}

\DeclareMathOperator{\DLPSC}{{\mathfrak{LP}_\textit{d}}}
\DeclareMathOperator{\DLDPSC}{{\mathfrak{LDP}_\textit{d}}}

\newcommand{\ti}[1]{\widetilde{#1}}
\newcommand{\conj}{\overline}
\newcommand{\imp}{\ensuremath{\Rightarrow}\xspace}

\newtheorem{theorem}{Theorem}[section]
\newtheorem{corollary}[theorem]{Corollary}
\newtheorem{lemma}[theorem]{Lemma}
\newtheorem{proposition}[theorem]{Proposition}
\theoremstyle{definition}
\newtheorem{definition}[theorem]{Definition}
\newtheorem{remark}[theorem]{Remark}
\newtheorem{example}[theorem]{Example}

\stackMath
\newcommand\widecheck[1]{%
\savestack{\tmpbox}{\stretchto{%
  \scaleto{%
    \scalerel*[\widthof{\ensuremath{#1}}]{\kern-.6pt\bigwedge\kern-.6pt}%
    {\rule[-\textheight/2]{1ex}{\textheight}}
  }{\textheight}%
}{0.5ex}}%
\stackon[1pt]{#1}{\scalebox{-1}{\tmpbox}}%
}
\begin{document}

\title[Nonholonomic discrete lagrangian reduction by stages]{Lagrangian reduction of nonholonomic discrete mechanical systems by stages}


\author[J. Fern\'andez]{Javier Fern\'andez}

\address{Instituto Balseiro, Universidad Nacional de Cuyo --
  C.N.E.A.\\ Av. Bustillo 9500, San Carlos de Bariloche, R8402AGP,
  Rep\'ublica Argentina}

\email{jfernand@ib.edu.ar}


\author[C. Tori]{Cora Tori}

\address{Depto. de Ciencias B\'asicas, Facultad de Ingenier\'ia,
  Universidad Nacional de La Plata\\Calle 116 entre 47 y 48,
  2$^{\underline{o}}$ piso, La Plata, Buenos Aires, 1900, Rep\'ublica
  Argentina}

\address{Centro de Matem\'atica de La Plata (CMaLP)}

\email{cora.tori@ing.unlp.edu.ar}


\author[M. Zuccalli]{Marcela Zuccalli}

\address{Depto. de Matem\'atica, Facultad de Ciencias Exactas,
  Universidad Nacional de La Plata Calles 50 y 115, La Plata, Buenos
  Aires, 1900, Rep\'ublica Argentina}

\address{Centro de Matem\'atica de La Plata (CMaLP)}

\email{marce@mate.unlp.edu.ar}




\thanks{This research was partially supported by grants from
  Universidad Nacional de Cuyo (\#06/C567 and \#06/C574) and
  Universidad Nacional de La Plata.}


\bibliographystyle{amsplain}

\maketitle

\begin{abstract}
  In this work we introduce a category $\DLDPSC$ of discrete-time
  dynamical systems, that we call discrete
  Lagrange--D'Alembert--Poincar\'e systems, and study some of its
  elementary properties. Examples of objects of $\DLDPSC$ are
  nonholonomic discrete mechanical systems as well as their lagrangian
  reductions and, also, discrete Lagrange-Poincar\'e systems. We also
  introduce a notion of symmetry group for objects of $\DLDPSC$ and a
  process of reduction when symmetries are present. This reduction
  process extends the reduction process of discrete
  Lagrange--Poincar\'e systems as well as the one defined for
  nonholonomic discrete mechanical systems. In addition, we prove
  that, under some conditions, the two-stage reduction process (first
  by a closed and normal subgroup of the symmetry group and, then, by
  the residual symmetry group) produces a system that is isomorphic in
  $\DLDPSC$ to the system obtained by a one-stage reduction by the
  full symmetry group.
\end{abstract}



\section{Introduction}
\label{sec:introduction}


Mechanical systems are dynamical systems that are used to model a wide
variety of aspects of the real world, from the falling apple to the
movement of astronomical objects, including machinery and billiards
(see, for instance,~\cite{bo:goldstein-classical_mechanics}
and~\cite{bo:AM-mechanics}). One of the flavors of Mechanics
---Lagrangian or Variational Mechanics--- describes the evolution of a
mechanical system using a variational principle defined in terms of a
function, the Lagrangian, $L:TQ\rightarrow \R$, where $Q$ is the
configuration manifold of the system. Nonholonomic mechanical systems,
which describe systems containing rolling or sliding contact (such as
wheels or skates), add constraints ---in the form of a non-integrable
subbundle ${\mathcal D}\subset TQ$--- to the variational principle
(see, for instance,~\cite
{bo:bloch-nonholonomic_mechanics_and_control}
and~\cite{ar:cendra_marsden_ratiu-geometric_mechanics_lagrangian_reduction_and_nonholonomic_systems}).

\smallskip

\noindent \textit{Numerical integrators and discrete mechanical
  systems.} As in many applications it is essential to predict the
evolution of a mechanical system, the equations of motion that can be
derived from the corresponding variational principle must be
solved. Solving these ordinary differential equations can be quite
difficult in practice, so numerical integrators are used to find
approximate solutions to those equations. The standard methods for
numerically approximating solutions of ODEs do not necessarily
preserve the structural characteristics of the solutions of the
equations of motion of mechanical systems
(see~\cite{bo:hairer_lubich_wanner-geometric_numerical_integration}). Discrete
mechanical systems were introduced as a way of modeling discrete-time
analogues of mechanical systems; the evolution of a discrete
mechanical system is also defined in terms of a variational principle
for the discrete Lagrangian $L_d:Q\times Q\rightarrow \R$; this
formalism is extended to deal with more general systems, including
forced discrete systems as well as discrete nonholonomic ones
(see~\cite{ar:marsden_west-discrete_mechanics_and_variational_integrators}
and~\cite{bo:cortes-non_holonomic}). The equations of motion of
discrete mechanical systems are algebraic equations whose solutions
are numerical integrators for the corresponding continuous system. In
many cases, these integrators have very good structural
characteristics (especially when considering long-time evolution),
that resemble those of the continuous system
(\cite{ar:patrick_cuell-error_analysis_of_variational_integrators_of_unconstrained_lagrangian_systems}
and~\cite{bo:hairer_lubich_wanner-geometric_numerical_integration}).

\smallskip

\noindent \textit{Symmetries and symmetry reduction.} It is a natural
idea to think that when a mechanical or, more generally, a dynamical
system has some degree of symmetry, it should be possible to gain some
insight into its dynamics by studying some other ``simplified system''
obtained by eliminating or locking the symmetry. This process is
usually known as the \jdef{reduction} of the given system and the
resulting system is known as the \jdef{reduced system}. In the case of
Classical Mechanics, this idea seems to go back as far as the work of
Lagrange. Over time, it has become a technique that has been applied
in both the Lagrangian and Hamiltonian formalisms, for unconstrained
systems as well as for holonomically and nonholonomically constrained
ones (see among \emph{many} other
references,~\cite{bo:bloch-nonholonomic_mechanics_and_control}~\cite{ar:cendra_marsden_ratiu-geometric_mechanics_lagrangian_reduction_and_nonholonomic_systems},~\cite{ar:arnlod-sur_la_geometrie_differentielle_des_groupes_de_lie_de_dimension_infinie_et_ses_applications_a_l'hydrodynamique_des_fluides_parfaits},~\cite{ar:smale-topology_and_mechanics_1,ar:smale-topology_and_mechanics_2}~\cite{ar:meyer-symmetries_and_integrals_in_mechanics},~\cite{ar:marsden_weinstein-reduction_of_symplectic_manifolds_with_symmetry},~\cite{bo:marsden_misiolek_ortega_perlmutter_ratiu-hamiltonian_reduction_by_stages},~\cite{ar:castrillonlopez_ratiu-reduction_of_principal_bundles_covariant_lagrange_poincare_equations}
and~\cite{ar:marrero_martin_martinez-discrete_lagrangian_and_hamiltonian_mechanics_on_lie_groupoids}). The
reduction process has also been applied to discrete-time mechanical
systems with and without constraints (see, for
instance,~\cite{ar:jalnapurkar_leok_marsden_west-discrete_routh_reducion},~\cite{ar:mclachlan_perlmutter-integrators_for_nonholonomic_mechanical_systems},~\cite{ar:iglesias_marrero_martin_martinez_padron-reduction_of_symplectic_lie_algebroids_by_a_lie_subalgebroid_and_a_symmetry_lie_group}
and~\cite{ar:fernandez_tori_zuccalli-lagrangian_reduction_of_discrete_mechanical_systems}).

It is well known that, in most instances, the reduction of a
mechanical system is not a mechanical system but, rather, a more
general dynamical system: that is, while the dynamics of a mechanical
system on $Q$ is defined using a variational principle for the
Lagrangian, defined on $TQ$ in the continuous case or on $Q\times Q$
in the discrete case (and, maybe, other additional data), the dynamics
of the reduced system is determined by a function that is usually not
defined on a tangent bundle or a Cartesian product (of a manifold with
itself). This can be problematic if one expects to analyze the reduced
system with the same techniques as the original one. That issue has
usually been solved by passing from the family of mechanical systems
to a larger class of dynamical systems, where there is a reduction
process that is closed within this larger class. Such is the case, for
example, of mechanical systems on Lie algebroids and Lie groupoids
(see~\cite{ar:iglesias_marrero_deDiego_martinez-discrete_nonholonomic_lagrangian_systems_on_lie_groupoids}
and~\cite{ar:iglesias_marrero_martin_martinez_padron-reduction_of_symplectic_lie_algebroids_by_a_lie_subalgebroid_and_a_symmetry_lie_group}). In
this paper we follow this guiding principle, but choose the larger
class following ideas adapted
from~\cite{bo:cendra_marsden_ratiu-lagrangian_reduction_by_stages}
and~\cite{ar:fernandez_tori_zuccalli-lagrangian_reduction_of_discrete_mechanical_systems_by_stages}.

\smallskip

\noindent \textit{Reduction by stages.} Sometimes, it may be convenient to
eliminate part of the symmetric behavior of a mechanical system, while
keeping some residual symmetry that could be analyzed at a later
point, if so desired.  In this case a second reduction step to
eliminate the residual symmetry is possible. A natural question is, in
that case, whether the result of this two-stage reduction is
equivalent to the full reduction of all the symmetries at one
time. The equivalence of the two-stage and one-stage reduction
processes has been established in several cases. For instance, for
Lagrangian systems without constraints by H. Cendra, J. Marsden and
T. Ratiu
in~\cite{bo:cendra_marsden_ratiu-lagrangian_reduction_by_stages}, for
Lagrangian systems with nonholonomic constraints by H. Cendra and
V. D\'{i}az
in~\cite{ar:cendra_diaz-lagrange_dalembert_poincare_equations_by_several_stages},
for Hamiltonian systems by J. Marsden et
al. in~\cite{bo:marsden_misiolek_ortega_perlmutter_ratiu-hamiltonian_reduction_by_stages}
and, for unconstrained discrete mechanical systems by the authors
in~\cite{ar:fernandez_tori_zuccalli-lagrangian_reduction_of_discrete_mechanical_systems_by_stages}.

\smallskip

\noindent \textit{Aims.} The main purpose of the present work is to
establish the same equivalence described in the previous paragraph for
nonholonomically constrained discrete-time mechanical systems. In this
respect, the paper is an extension
of~\cite{ar:fernandez_tori_zuccalli-lagrangian_reduction_of_discrete_mechanical_systems_by_stages}
to the constrained setting that
parallels~\cite{ar:cendra_diaz-lagrange_dalembert_poincare_equations_by_several_stages}
in the discrete-time context.  We stress that, for the present paper,
whether a discrete-time system is related to a (continuous-time)
mechanical system or not is irrelevant; in this respect, our analysis
and results are completely independent of the discretization process
chosen to produce the discrete dynamical system in question, if one
was used at all. As an aside, we also mention that, at the moment, the
very interesting subject of geometric discretization of nonholonomic
mechanical systems should be regarded as ``work in progress'', without
conclusive hard results on the quality of the numerical integrators
obtained.

\smallskip

\noindent \textit{Constructions and results.} Except for a few special
cases, the reduced system obtained from a nonholonomic discrete
mechanical system via the general reduction process defined
in~\cite{ar:fernandez_tori_zuccalli-lagrangian_reduction_of_discrete_mechanical_systems}
is not a discrete mechanical system, constrained or not. So, as we
mentioned above, the first step is to construct a family of dynamical
systems that contains all the systems of interest ---nonholonomic
discrete mechanical systems as well as their reductions. The discrete
Lagrange--D'Alembert--Poincar\'e systems (DLDPSs) form such a family:
one of these systems is determined by a fiber bundle
$\phi:E\rightarrow M$, a function $L_d:E\times M\rightarrow\R$, the
\jdef{discrete Lagrangian}, a \jdef{nonholonomic infinitesimal
  variation chaining map} $\mathcal{P}$ (see
Definition~\ref{def:chaining_map}) as well as a regular submanifold
$\mathcal{D}_d\subset E\times M$, the \jdef{kinematic constraints},
and a subbundle $\mathcal{D}\subset p_1^*TE$ (where
$p_1:E\times M\rightarrow E$ is the projection), the \jdef{variational
  constraints}.  All such systems are discrete-time dynamical systems
whose trajectories are determined by a variational principle. Examples
of DLDPSs include the nonholonomic discrete mechanical systems
(Example~\ref{ex:DNHMS_as_DLDPS}) and, when they are symmetric, their
reductions by the procedure defined
in~\cite{ar:fernandez_tori_zuccalli-lagrangian_reduction_of_discrete_mechanical_systems}
(Section~\ref{sec:nonholonomic_discrete_mechanical_systems_with_symmetry});
also, the discrete Lagrange--Poincar\'e systems considered
in~\cite{ar:fernandez_tori_zuccalli-lagrangian_reduction_of_discrete_mechanical_systems_by_stages}
are DLDPSs. A convenient notion of morphism between DLDPSs is
introduced and a category $\DLDPSC$ is so defined. The category
$\DLPSC$ of discrete Lagrange--Poincar\'e systems defined
in~\cite{ar:fernandez_tori_zuccalli-lagrangian_reduction_of_discrete_mechanical_systems_by_stages}
is a full subcategory of $\DLDPSC$.

Roughly speaking, a Lie group $G$ is a symmetry group of a DLDPS if it
acts on the underlying fiber bundle in such a way that it preserves
the different structures. When $G$ is a symmetry group of a DLDPS
$\mathcal{M}$, we construct a new DLDPS $\mathcal{M}/G$ that we call
the \jdef{reduced system}. In fact, the construction requires an
additional piece of data: an affine discrete connection on a certain
principal $G$-bundle; interestingly, we prove that the reduced systems
obtained using different affine discrete connections are always
isomorphic in $\DLDPSC$
(Proposition~\ref{prop:reduction_with_different_DC_are_isomorphic}). Also,
the reduction mapping $\mathcal{M}\rightarrow \mathcal{M}/G$ is a
morphism in $\DLDPSC$; Corollary~\ref{cor:4_pts} and
Theorem~\ref{thm:reconstruction} prove that the reduction mapping
determines a bijective correspondence between the trajectories of
${\mathcal M}$ and those of ${\mathcal M}/G$. It is important to
notice that both the notion of symmetry group and the reduction
process extend the ones already in use for nonholonomic discrete
mechanical systems as well as for discrete Lagrange--Poincar\'e
systems.

When $G$ is a symmetry group of the DLDPS $\mathcal{M}$ and
$H\subset G$ is a closed and normal subgroup, $H$ is a symmetry group
of $\mathcal{M}$, so we can consider the reduced system
$\mathcal{M}^H:=\mathcal{M}/H$ using a discrete affine connection
$\DCp{H}$. Then, under a condition on $\DCp{H}$, we prove that $G/H$
is a symmetry group of $\mathcal{M}^H$, so that we can consider a new
reduced system $\mathcal{M}^{G/H}:=\mathcal{M}^H/(G/H)$. One of the
main results of the paper, Theorem~\ref{thm:F_diffeomorphism}, is that
$\mathcal{M}^{G/H}$ is isomorphic in $\DLDPSC$ to $\mathcal{M}/G$.

\smallskip

\noindent \textit{Plan for the paper.}
Section~\ref{sec:revision_of_some_discrete_tools} reviews the notion
of affine discrete connection as well as some basic results on
principal bundles. Section~\ref{sec:discrete_LDP_systems} introduces
the DLDPSs and their dynamics.
Section~\ref{sec:nonholonomic_discrete_mechanical_systems_with_symmetry}
shows that both nonholonomic discrete mechanical systems as well as
their reduction (in the sense
of~\cite{ar:fernandez_tori_zuccalli-lagrangian_reduction_of_discrete_mechanical_systems})
are examples of DLDPS and, also, that their dynamics as DLDPSs is the
same as the ``classical
one''. Section~\ref{sec:categorical_formulation} introduces the
category $\DLDPSC$ whose objects are DLDPSs. Symmetries and a
reduction process in $\DLDPSC$ are analyzed in
Section~\ref{sec:Reduction_discrete_LDP_systems}; in particular, in
Section~\ref{sec:discrete_LL_systems_on_lie_groups}, we illustrate how
these ideas can be applied by studying the discrete $LL$ systems on a
Lie group $G$. Finally, Section~\ref{sec:reduction_by_two_stages}
establishes the equivalence between the two-stage and the single-stage
reduction process, under appropriate conditions.

\smallskip

\noindent\textit{Future work.} It would be very interesting to connect the
analysis of this paper with a discretization process for continuous
mechanical systems. This would allow, for instance, the estimation of
the error made when using a DLDPS as an approximation of a
(continuous) mechanical system. Indeed, a first step would be to
tackle this same problem with no constraints, that is, for discrete
Lagrange--Poincar\'e systems
(\cite{ar:fernandez_tori_zuccalli-lagrangian_reduction_of_discrete_mechanical_systems_by_stages}). It
should be noted that this error analysis is only known for
unconstrained systems
(see~\cite{ar:patrick_cuell-error_analysis_of_variational_integrators_of_unconstrained_lagrangian_systems})
and forced mechanical systems
(see~\cite{ar:deDiego_deAlmagro-variational_order_for_forced_lagrangian_systems}
and~\cite{ar:fernandez_graiffZurita_grillo-error_analysis_of_forced_discrete_mechanical_systems}). Another
avenue for exploration would be the study of possible Poisson
structures in DLDPSs: even though DLDPSs do not have a canonical
Poisson structure, some of them do (those coming from discrete
mechanical systems, for instance) and it would be interesting to see
how those structures behave under the reduction process.

\smallskip

\noindent\textit{Notation.} Throughout the paper many spaces are Cartesian
products. In general we denote the corresponding projections by
$p_k:\prod_{j=1}^N X_j\rightarrow X_k$ and the obvious
adaptations. Also, $l^X$ and $r^X$ will denote left and right smooth
actions of a Lie group on the manifold $X$. If $G$ acts on the left on
$X$ we denote the corresponding quotient map by
$\pi^{X,G}:X\rightarrow X/G$.


\section{Revision of some discrete tools}
\label{sec:revision_of_some_discrete_tools}

In this section we review some basic notions and results about affine
discrete connections and smooth fiber bundles.


\subsection{Affine discrete connections}
\label{subsec:affine_discrete_connections}

Let $l^Q:G\times Q\rightarrow Q$ be a smooth left action of the Lie
group $G$ on the manifold $Q$. We consider several other actions of
$G$; for example, we have the $G$ actions $l^{Q\times Q}$ and
$l^{Q\times Q_2}$ on $Q\times Q$ defined by
$l^{Q\times Q}_g(q_0,q_1):=(l^Q_g(q_0),l^Q_g(q_1))$ and
$l^{Q\times Q_2}_g(q_0,q_1):=(q_0,l^Q_g(q_1))$. We also consider the
left $G$-action on itself given by $l^G_g(g'):=g g' g^{-1}$.

\begin{definition}\label{def:affine_discrete_connection}
  Let $\gamma:Q \rightarrow G $ be a smooth $G$-equivariant map with
  respect to $l^{Q}$ and $l^{G}$,
  $\Gamma :=\{(q,l^{Q}_{\gamma(q)}(q)):q\in Q\}$ and
  $Hor\subset Q\times Q$ be an $l^{Q\times Q}$-invariant submanifold
  containing $\Gamma$. We say that $Hor$ defines an \jdef{affine
    discrete connection} $\DC$ on the principal $G$-bundle
  $\pi^{Q,G}:Q \rightarrow Q/G$ if
  $(id_Q \times \pi^{Q,G})|_{Hor}: Hor \rightarrow Q\times (Q/G)$ is
  an injective local diffeomorphism. We denote $Hor$ by $Hor_{\DC}$
  and we call $\gamma$ the level of $\DC$. \emph{As in this paper the
    only type of discrete connection that we consider is the affine,
    we will simply call them discrete connections}.
\end{definition}

Given a discrete connection $\DC$ on $\pi^{Q,G}:Q \rightarrow Q/G$,
the space
\begin{equation*}
  \jgU := l_{G}^{Q \times Q_2}(Hor_{\DC}) =
  \{(q_0,l_{g}^{Q}(q_1)) \in Q \times Q : (q_0,q_1) \in
  Hor_{\DC}, g \in G \},
\end{equation*}
is called the \jdef{domain} of $\DC$.
 
\begin{proposition}
  The space $\jgU$ is an open set in $Q \times Q$.
\end{proposition}

\begin{proof}
  See point 1 of Proposition 2.4
  in~\cite{ar:fernandez_zuccalli-a_geometric_approach_to_discrete_connections_on_principal_bundles}.
\end{proof}

\begin{proposition}\label{prop:existence_of_g_moving_to_hor}
  Let $\DC$ be a discrete connection with level $\gamma$ and domain
  $\jgU$ on the principal $G$-bundle $\pi^{Q,G} : Q \rightarrow
  Q/G$. For each $(q_0,q_1) \in \jgU$, there is a unique $g \in G$
  such that $(q_0,l_{g^{-1}}^{Q}(q_1)) \in Hor_{\DC}$.
\end{proposition}

\begin{proof}
  See Proposition 2.5
  in~\cite{ar:fernandez_zuccalli-a_geometric_approach_to_discrete_connections_on_principal_bundles}.
\end{proof}

\begin{definition}\label{def:affine_discrete_connection_form}
  Given a discrete connection $\DC$ with domain $\jgU$ on the
  principal $G$-bundle $\pi^{Q,G} : Q \rightarrow Q/G$, we define its
  \jdef{discrete connection form}
  \begin{equation*}
    \DC : \jgU \subset Q \times Q \rightarrow G \text{ by } \DC(q_0,q_1):= g,
  \end{equation*}
  where $g \in G$ is the element that appears in
  Proposition~\ref{prop:existence_of_g_moving_to_hor}.
\end{definition}

In what follows we consider the open set
$\jgU':= (id \times \pi^{Q,G})(Hor_{\DC}) \subset Q \times (Q/G)$.

\begin{definition}\label{def:discrete_horizontal_lift}
  Let $\DC$ be a discrete connection on the principal $G$-bundle
  $\pi^{Q,G} : Q \rightarrow Q/G$. The \jdef{discrete horizontal lift}
  $h_d: \jgU' \rightarrow Hor_{\DC}$ is the inverse map of the
  injective local diffeomorphism
  $(id_Q \times \pi^{Q,G})|_{Hor_{\DC}} : Hor_{\DC} \rightarrow
  \jgU'$. That is 
  \begin{equation*}
    \HLd{q_0}(r_1) = \HL(q_0,r_1) :=(q_0,q_1) \quad \Leftrightarrow \quad 
    (q_0,q_1) \in Hor_{\DC} \stext{ and } \pi^{Q,G}(q_1)=r_1.
  \end{equation*}
  In addition we define $\HLds{q_0} := p_{2} \circ \HLd{q_0}$.
\end{definition}	

\begin{proposition}\label{prop:properties_affine_discrete_connection}
  Let $\DC$ be a discrete connection on the principal
  $G$-bundle $\pi^{Q,G} : Q \rightarrow Q/G$. Then,
  \begin{enumerate}
  \item \label{it::properties_affine_discrete_connection-smooth} the
    discrete connection form $\DC$ and the discrete horizontal lift
    $\HL$ are smooth maps and,
  \item \label{it::properties_affine_discrete_connection-G_equiv} if we
    consider the left $G$-actions on $G$ and on $Q \times (Q/G)$ given
    by $l^G$ and
    \begin{equation*}
      l_{g}^{Q \times (Q/G)}(q_0,r_1):=(l_{g}^{Q}(q_0),r_1),
    \end{equation*}
    and the diagonal action $l^{Q\times Q}$ on $Q \times Q$ then $\DC$
    and $\HL$ are $G$-equivariant.
  \item \label{it::properties_affine_discrete_connection-GxG_equiv} In
    general, for any $g_0,g_1 \in G$,
    \begin{equation}
      \label{eq:equivariance_of_connection_form}
      \DC(l_{g_0}^{Q}(q_0),l_{g_1}^{Q}(q_1)) =
      g_{1} \DC(q_0,q_1) g_{0}^{-1} \stext{ for all } (q_0,q_1) \in \jgU.
    \end{equation}
  \end{enumerate}
\end{proposition}

\begin{proof}
  All of the following references are
  from~\cite{ar:fernandez_zuccalli-a_geometric_approach_to_discrete_connections_on_principal_bundles}
  and must be adapted to affine discrete connections.
  Point~\ref{it::properties_affine_discrete_connection-smooth} is
  Lemma 3.2 (smoothness of $\DC$) and Point 2 in Theorem 4.4
  (smoothness of
  $\HL$). Point~\ref{it::properties_affine_discrete_connection-GxG_equiv}
  is part of Theorem 3.4 while
  Point~\ref{it::properties_affine_discrete_connection-G_equiv}
  follows from
  point~\ref{it::properties_affine_discrete_connection-GxG_equiv} just
  proved and point 2 in Theorem 4.4.
\end{proof}

\begin{proposition}
  Given a smooth function $\mathcal{A} :Q \times Q \rightarrow G$ such
  that~\eqref{eq:equivariance_of_connection_form} holds (with
  $\mathcal{A}$ instead of $\DC$), then
  $Hor:=\{(q_0,q_1)\in Q \times Q: \mathcal{A}(q_0,q_1)=e\}$ defines
  an affine discrete connection with level set
  $\gamma(q):=\mathcal{A}(q,q)^{-1}$ and whose discrete connection
  $1$-form is $\mathcal{A}$.
\end{proposition}	

\begin{proof}
  This proof is analogue to the proof of Proposition 4.12
  in~\cite{ar:fernandez_tori_zuccalli-lagrangian_reduction_of_discrete_mechanical_systems}.
\end{proof}


\subsection{Principal bundles}
\label{subsec:principal_bundles}

Here we review a few basic notions and results on principal
bundles. We refer to Section 9
of~\cite{ar:fernandez_tori_zuccalli-lagrangian_reduction_of_discrete_mechanical_systems_by_stages}
and its references for additional details.

\begin{definition}\label{def:G_action_on_fiber_bundle}
  Let $G$ be a Lie group and $(E,M,\phi,F)$ a fiber bundle. We say
  that $G$ \jdef{acts on the fiber bundle} $E$ if there are free left
  $G$-actions $l^{E}$ and $l^{M}$ on $E$ and $M$ respectively and a
  right $G$-action $r^{F}$ on $F$ such that
  \begin{enumerate}
  \item $l^{M}$ induces a principal $G$-bundle structure
    $\pi^{M,G}: M \rightarrow M/G$,
  \item $\phi$ is a $G$-equivariant map for the given actions,
  \item for every $m \in M$ there is a trivializing chart
    $(U,\Phi_{U})$ of $E$ such that $U \subset M$ is $G$-invariant,
    $m \in U$ and, when considering the left $G$-action
    $l^{U \times F}$ on $U \times F$ given by
    $l_{g}^{U \times F}(m,f) := (l_{g}^{M}(m),r_{g^{-1}}^{F}(f))$, the
    map $\Phi_{U}$ is $G$-equivariant.
  \end{enumerate}
\end{definition}

\begin{remark}
  When a Lie group $G$ acts on the fiber bundle $(E,M,\phi,F)$ and on
  the manifold $F'$ by a right action, it is possible to construct an
  \jdef{associated bundle} on $M/G$ with total space $(E \times F')/G$
  and fiber $F \times F'$. The special case when $F'=G$ acting on
  itself by $r_{g}(h) := g^{-1}hg$ is known as the \jdef{conjugate
    bundle} and is denoted by $\ti{G}_E$.
\end{remark}

\begin{proposition}\label{prop:equivariant_diffeomorphisms}	
  Let $G$ be a Lie group that acts on the fiber bundle $(E,M,\phi,F)$
  and $\DC$ be a discrete connection on the principal $G$-bundle
  $\pi^{M,G} : M \rightarrow M/G$. We define
  $\ti{\Phi}_{\DC} : E \times M \rightarrow E \times G \times (M/G)$
  and
  $\ti{\Psi}_{\DC} : E \times G \times (M/G) \rightarrow E \times M$
  by
  \begin{gather*}
    \ti{\Phi}_{\DC}(\epsilon ,m) := (\epsilon,\DC(\phi(\epsilon),m),
    \pi^{M,G}(m)) \stext{ and } \ti{\Psi}_{\DC}(\epsilon,w,r) :=
    (\epsilon, l_{w}^{M}(\HLds{\phi(\epsilon)}(r))).
  \end{gather*}
  Then, $\ti{\Phi}_{\DC}$ and $\ti{\Psi}_{\DC}$ are smooth functions,
  inverses of each other. If we view $E\times M$ and
  $E\times G\times (M/G)$ as fiber bundles over $M$ via
  $\phi\circ p_1$, then $\ti{\Phi}_{\DC}$ and $\ti{\Psi}_{\DC}$ are
  bundle maps (over the identity). In addition, if we consider the
  left $G$-actions $l^{E \times M}$ and $l^{E \times M \times (M/G)}$
  defined by
  \begin{equation*}
    l_{g}^{E \times M}(\epsilon,m) := (l_{g}^{E}(\epsilon), l_{g}^{M}(m))
    \stext{ and } l_{g}^{E \times M \times (M/G)}(\epsilon,w,r) :=
    (l_{g}^{E}(\epsilon),l_{g}^{G}(w),r),
  \end{equation*}
  then $\ti{\Phi}_{\DC}$ and $\ti{\Psi}_{\DC}$ are $G$-equivariant and
  they induce diffeomorphisms
  $\Phi_{\DC}:(E \times M)/G \rightarrow \ti{G}_{E} \times (M/G)$ and
  $\Psi_{\DC}:\ti{G}_{E} \times (M/G) \rightarrow (E \times M)/G$.
\end{proposition}

\begin{proof}
  This is Proposition 2.6
  in~\cite{ar:fernandez_tori_zuccalli-lagrangian_reduction_of_discrete_mechanical_systems_by_stages}
  adapted to affine discrete connections.
\end{proof}

\begin{remark}
  The discrete connection $\DC$ need not be defined on $Q \times Q$
  but, rather, on the open subset $\jgU$. This restricts the domain of
  $\ti{\Psi}_{\DC}$ and $\ti{\Phi}_{\DC}$ to appropriate open sets,
  where the results of the Proposition
  \ref{prop:equivariant_diffeomorphisms} hold. We will ignore this
  point and keep working as if $\DC$ were globally defined in order to
  avoid a more involved notation.
\end{remark}

We have the commutative diagram
\begin{equation}\label{eq:diagram_ExM_to_reduced}
  \xymatrix{ {E\times M} \ar[r]^(.4){\ti{\Phi}_{\DC}}_(.4){\sim}
    \ar[d]_{\pi^{E\times M,G}} \ar[dr]^(.55){\Upsilon_{\DC}} & 
    {(E\times G) \times (M/G)} \ar[d]^{\pi^{E\times G,G}\times id_{M/G}} \\
    {(E\times M)/G} \ar[r]_(.45){\Phi_{\DC}}^(.45){\sim} &
    {\ti{G}_E\times (M/G)}}
\end{equation}
where $\Upsilon_{\DC}:E \times M \rightarrow \ti{G}_{E} \times (M/G)$
is defined as
\begin{equation}\label{eq:Upsilon_DC-def}
  \Upsilon_{\DC} := \Phi_{\DC} \circ \pi^{E\times M,G}
  = (\pi^{E\times G,G}\times id_{M/G})\circ \ti{\Phi}_{\DC}.
\end{equation}

\begin{lemma}\label{le:Upsilon_fiber_bundle}
  Let $G$ be a Lie group that acts on the fiber bundle $(E,M,\phi,F)$
  and $\DC$ be a discrete connection on the principal $G$-bundle
  $\pi^{M,G} : M \rightarrow M/G$. Then,
  $\Upsilon_{\DC} : E \times M \rightarrow \ti{G}_E \times (M/G)$
  defined by~\eqref{eq:Upsilon_DC-def} is a principal $G$-bundle.
\end{lemma}
\begin{proof}
  This is Lemma 2.8
  in~\cite{ar:fernandez_tori_zuccalli-lagrangian_reduction_of_discrete_mechanical_systems_by_stages}
  adapted to affine discrete connections.
\end{proof}

All together, we have the following commutative diagram
\begin{equation}\label{eq:ExM_and_tiGxM/G}
  \xymatrix{
    {E} \ar[d]_{\pi^{M,G}\circ \phi} & 
    {E\times M} \ar[l]_{p_1} \ar[dr]^{\Upsilon_{\DC}} \ar[r]^(.35){\ti{\Phi}_{\DC}}
    & {(E\times G)\times (M/G)} \ar[d]^{\pi^{E\times G,G}\times id_{M/G}}\\
    {M/G} & {} & {\ti{G}_E \times (M/G)} \ar[ll]^{p^{M/G}\circ p_1}\\
  }
\end{equation}

\begin{proposition}\label{prop:bundle_image_submanifold}
  Let $\rho:X\rightarrow Y$ be a principal G-bundle, $Z \subset X$ a
  $G$-invariant regular submanifold, and $S:=\rho(Z)$. Then $S$ is a
  regular submanifold of Y.
\end{proposition}

\begin{proof}
  The statement can be proved locally, that is, it suffices to show
  that for each $s\in S$ there is an open subset $U\subset Y$ such
  that $s\in U$ and $(S\cap U) \subset U$ is a regular submanifold.

  As $\rho$ is a principal $G$-bundle, for each $s\in S$, there are an
  open subset $U\subset Y$ such that $s\in U$ and a diffeomorphism
  $\Phi_U:\rho^{-1}(U) \rightarrow U \times G$ that is $G$-equivariant
  (for $l^{U\times G}_g(u,g'):=(u,gg')$) and that
  $p_1\circ \Phi_U = \rho|_{\rho^{-1}(U)}$.

  As $Z\subset X$ is a regular submanifold and $\rho^{-1}(U)$ is an
  open subset of $X$, $Z\cap \rho^{-1}(U)$ is a regular submanifold of
  $\rho^{-1}(U)$. Then, as $\Phi_U$ is a diffeomorphism,
  $\ti{Z}:=\Phi_U(Z\cap \rho^{-1}(U))$ is a regular submanifold of
  $U\times G$. Furthermore, as $Z$ is $G$-invariant and $\Phi_U$ is
  $G$-equivariant, $\ti{Z}$ is $G$-invariant.

  Let $i:U\rightarrow U\times G$ be given by $i(u):=(u,e)$, where $e$
  is the identity of $G$; it is easy to check that
  $S\cap U=i^{-1}(\ti{Z})$. Then $i$ is smooth and, furthermore, for
  each $s'\in S\cap U$,
  $di(s')(T_{s'}U) = T_{s'}U\oplus \{0\} \subset T_{(s',e)} (U\times
  G)$. On the other hand, as $(s',e)\in \ti{Z}$ and $\ti{Z}$ is
  $G$-invariant, we have that
  $\{0\}\oplus T_e G \subset T_{(s',e)} \ti{Z} \subset T_{(s',e)}
  (U\times G)$. Then,
  $T_{i(s')}\ti{Z} \oplus di(s')(T_{s'}U) = T_{i(s')}(U\times G)$ and
  $i$ is transversal to $\ti{Z}$, so that $S\cap U=i^{-1}(\ti{Z})$ is
  a regular submanifold of $U$ (see Theorem 6.30
  in~\cite{bo:lee-introduction_to_smooth_manifolds}).
\end{proof}


\section{Discrete Lagrange--D'Alembert--Poincar\'e systems}
\label{sec:discrete_LDP_systems}

In this section we introduce a type of discrete-time dynamical system
that contains, among other examples, all nonholonomic discrete
mechanical systems as well as their reductions, as defined
in~\cite{ar:fernandez_tori_zuccalli-lagrangian_reduction_of_discrete_mechanical_systems}.


\subsection{Some definitions}
\label{sec:some_definitions}

Given a fiber bundle $\phi:E \rightarrow M$ we denote
$C'(E) := E \times M$, seen as a fiber bundle over $M$ by
$\phi \circ p_1$. We define the \jdef{discrete second order manifold}
$C''(E):= (E \times M) \times_{p_2,\phi \circ p_1} (E \times M)$
considered as a fiber bundle over $M$ by $\ti{p_2}:=p_2|_{C''(E)}$ for
the projection $p_2:E\times M\times E\times M\rightarrow M$.

\begin{remark}
  Given a fiber bundle $\phi:E \rightarrow M$, the second order
  manifold $\ti{p_2}:C''(E) \rightarrow M$ is isomorphic as a fiber
  bundle to the fiber bundle
  $\phi \circ p_2 : E \times E \times M \rightarrow M$ with
  $F_E((\epsilon_0,m_1),(\epsilon_1,m_2)) :=
  (\epsilon_0,\epsilon_1,m_1)$.
\end{remark}

\begin{definition}
  Given a fiber bundle $\phi:E \rightarrow M$, a \jdef{discrete path}
  in $C'(E)$ is a set
  $(\epsilon_\cdot,m_\cdot) =
  ((\epsilon_0,m_1),\ldots,(\epsilon_{N-1},m_N))$ where
  $((\epsilon_k,m_{k+1}),(\epsilon_{k+1},m_{k+2}))\in C''(E)$ for
  $k=0,\ldots,N-2$.
\end{definition}

\begin{definition}\label{def:chaining_map}
  Let $\phi:E \rightarrow M$ be a fiber bundle and $\mathcal{D}$ be a
  subbundle of the pullback bundle $p_{1}^{*}(TE)\subset T(C'(E))$. A
  \jdef{nonholonomic infinitesimal variation chaining map} (NIVCM)
  $\mathcal{P}$ on $(E,\mathcal{D})$ is a homomorphism of vector
  bundles over $\ti{p}_1$, according to the following commutative
  diagram
  \begin{equation*}
    \xymatrix{
      {\mathcal{D}} \ar[d] & \ti{p_{34}}^{*}(\mathcal{D}) \ar[l] \ar[d]
      \ar[r]^{\mathcal{P}} & {\ker(d\phi)} \ar[d]  \ar@{^{(}->}[r] &
      {TE} \ar[dl]\\
      {E\times M} & {C''(E)} 
      \ar[l]^{\ti{p_{34}}} \ar[r]_{\ti{p_1}} & {E} & {}
    }
  \end{equation*}
  where $\ti{p_1} ((\epsilon_0,m_1),(\epsilon_1,m_2)) := \epsilon_0$
  and
  $\ti{p_{34}} ((\epsilon_0,m_1),(\epsilon_1,m_2)) :=
  (\epsilon_1,m_2)$.
\end{definition}

\begin{remark}
  The fiber of the bundle $p_{1}^{*}(TE)$ on $(\epsilon,m)$ consists
  of vectors of the form
  $(\delta \epsilon,0) \in T_{(\epsilon,m)}(C'(E))$.
\end{remark}

\begin{definition}\label{def:ininitesimal_variation}
  Let $\phi:E \rightarrow M$ be a fiber bundle,
  $\mathcal{D}\subset p_1^*TE$ be a subbundle, $\mathcal{P}$ be a
  NIVCM on $(E,\mathcal{D})$ and
  $(\epsilon_\cdot,m_\cdot) = ((\epsilon_0,m_1), \ldots,
  (\epsilon_{N-1},m_N))$ be a discrete path in $C'(E)$. An
  \jdef{infinitesimal variation} over $(\epsilon_\cdot,m_\cdot)$ is a
  tangent vector
  $(\delta\epsilon_\cdot,\delta m_\cdot) = ((\delta\epsilon_0, \delta
  m_1), \ldots, (\delta\epsilon_{N-1}, \delta m_N))\in
  T_{(\epsilon_\cdot,m_\cdot)}(C'(E)^{N})$ such that
  \begin{equation}\label{eq:infinitesimal_varation_1}
    \delta m_k=d\phi(\epsilon_k)(\delta\epsilon_k) \stext{ with } k=1,\ldots,N-1.
  \end{equation}
  A \jdef{nonholonomic infinitesimal variation over
    $(\epsilon_\cdot,m_\cdot)$ with fixed endpoints} is an
  infinitesimal variation $(\delta\epsilon_\cdot,\delta m_\cdot)$ over
  $(\epsilon_\cdot,m_\cdot)$ such that
  \begin{equation}\label{eq:infinitesimal_varation_2}
    \begin{split}
      \delta m_N =& 0,\\
      \delta \epsilon_{N-1} =& \ti{\delta \epsilon_{N-1}},\\
      \delta \epsilon_k =& \ti{\delta \epsilon_k} +
      \mathcal{P}((\epsilon_{k},m_{k+1}),(\epsilon_{k+1},m_{k+2}))(\ti{\delta
        \epsilon_{k+1}},0), \text{ if } k=1,\ldots,N-2,\\
      \delta \epsilon_0 =&
      \mathcal{P}((\epsilon_{0},m_{1}),(\epsilon_{1},m_{2}))(\ti{\delta
        \epsilon_{1}},0),
    \end{split}
  \end{equation}
  where
  $(\ti{\delta \epsilon_k},0)\in \mathcal{D}_{(\epsilon_k,m_{k+1})}$
  is arbitrary for $k=1,\ldots,N-1$.
\end{definition}

\begin{definition}
  Let $\phi:E \rightarrow M$ be a fiber bundle. A \jdef{discrete
    Lagrange--D'Alembert--Poincar\'e system} (DLDPS) over $E$ is a
  collection
  $\mathcal{M} := (E,L_d,\mathcal{D}_d,\mathcal{D},\mathcal{P})$ where
  $L_d:C'(E) \rightarrow \R$ is a smooth function, the \jdef{discrete
    Lagrangian}, $\mathcal{D}_d\subset C'(E)$ is a regular
  submanifold, the \jdef{kinematic constraints}, $\mathcal{D}$ is a
  subbundle of $p_{1}^{*}(TE)$, the \jdef{variational constraints},
  and $\mathcal{P}$ is a NIVCM over $(E,\mathcal{D})$.
\end{definition}

\begin{definition}\label{def:trajectory_DLDPS}
  Let $\mathcal{M} = (E,L_d,\mathcal{D}_d,\mathcal{D},\mathcal{P})$ be
  a DLDPS. The \jdef{discrete action} of $\mathcal{M}$ is a function
  $S_d :C'(E)^{N} \rightarrow \R$ defined by
  $S_d(\epsilon_\cdot,m_\cdot) := \sum_{k=0}^{N-1}
  L_{d}(\epsilon_{k},m_{k+1}) $. A \jdef{trajectory of} $\mathcal{M}$
  is a discrete path $(\epsilon_\cdot,m_\cdot) \in C'(E)^{N}$ such
  that $(\epsilon_{k},m_{k+1})\in \mathcal{D}_d$ for all
  $k=0,\ldots,N-1$ and
  \begin{equation*}
    dS_{d}(\epsilon_\cdot,m_\cdot)(\delta\epsilon_\cdot,\delta m_\cdot)=0
  \end{equation*}
  for all nonholonomic infinitesimal variations
  $(\delta\epsilon_\cdot,\delta m_\cdot)$ on
  $(\epsilon_\cdot,m_\cdot)$ with fixed endpoints.
\end{definition}

Given a DLDPS
$\mathcal{M} = (E,L_d,\mathcal{D}_d,\mathcal{D},\mathcal{P})$ we have
the vector bundle $(\ti{p_{34}}^*(\mathcal{D}))^*\rightarrow
C''(E)$. Let $\nu_d$ be the smooth section of this bundle defined by
\begin{equation}\label{eq:section_of_motion-def}
  \begin{split}
    \nu_d((\epsilon_0,m_1),(\epsilon_1,m_2)) :=&
    D_{1}L_{d}(\epsilon_1,m_2)\circ dp_{1}(\epsilon_1,m_2) \\& +
    D_{2}L_{d}(\epsilon_0,m_1) \circ d(\phi\circ p_1)(\epsilon_1,m_2)
    \\&+ D_{1}L_{d}(\epsilon_0,m_1) \circ \mathcal{P}(
    (\epsilon_{0},m_{1}),(\epsilon_{1},m_{2})).
  \end{split}
\end{equation}
The next result characterizes the trajectories of a DLDPS in terms of
its equations of motion.

\begin{proposition}\label{prop:dynamics_DLDPS}
  Let $\mathcal{M} = (E,L_d,\mathcal{D}_d,\mathcal{D},\mathcal{P})$ be
  a DLDPS and $(\epsilon_\cdot,m_\cdot)$ be a discrete path in
  $C'(E)$. Then $(\epsilon_\cdot,m_\cdot)$ is a trajectory of
  $\mathcal{M}$ if and only if
  \begin{equation}
    \label{eq:evolution}
    \begin{split}
      (\epsilon_{k},m_{k+1}) \in \ & \mathcal{D}_d \stext{ for all }
      k=0,\ldots,N-1
      \stext{ and }\\
      \nu_d((\epsilon_{k-1},m_k),(\epsilon_k,m_{k+1})) =\ & 0 \stext{
        for all } k=1,\ldots,N-1,
    \end{split}
  \end{equation}
  where $\nu_d$ is the section defined
  by~\eqref{eq:section_of_motion-def}.
 \end{proposition}

\begin{proof}
  Let $(\delta \epsilon_\cdot, \delta m_\cdot)$ be a nonholonomic
  infinitesimal variation over $(\epsilon_\cdot, m_\cdot)$ with fixed
  endpoints.  A straightforward but lengthy computation using
  Definition~\ref{def:ininitesimal_variation} shows that
  \begin{equation*}
    \begin{split}
      dS_{d}( \epsilon_{\cdot },m_{\cdot }) (\delta \epsilon_{\cdot},
      \delta m_{\cdot }) =& \sum_{k=1}^{N-1}\big(D_{1}L_{d}(\epsilon
      _{k}, m_{k+1})\circ dp_{1}(\epsilon_k, m_{k+1})
      \\&\phantom{\sum\big(} +D_{1}L_{d}(\epsilon_{k-1},
      m_{k}) \circ \mathcal{P}((\epsilon_{k-1}, m_{k}),(\epsilon_{k},
      m_{k+1})) \\&\phantom{\sum\big(}+
      D_{2}L_{d}(\epsilon_{k-1}, m_{k}) \circ d(\phi\circ
      p_{1})(\epsilon_{k}, m_{k+1})\big)(\ti{\delta \epsilon_{k}},0).
    \end{split}
  \end{equation*}
  As the
  $\ti{\delta \epsilon_k}\in \mathcal{D}_{(\epsilon_k,m_{k+1})}$ are
  arbitrary, the result then follows by
  Definition~\ref{def:trajectory_DLDPS}.
\end{proof}

We refer to condition~\eqref{eq:evolution} as the \jdef{equations of
  motion} of the system.

\begin{example}\label{ex:DNHMS_as_DLDPS}
  We recall
  from~\cite{ar:fernandez_tori_zuccalli-lagrangian_reduction_of_discrete_mechanical_systems}
  (Definition 3.1) that a discrete nonholonomic mechanical system is a
  collection $(Q,L_d,\mathcal{D}_d,\mathcal{D}^{nh})$ where $Q$ is a
  differentiable manifold, $L_d:Q \times Q \rightarrow \R$ is a smooth
  function, $\mathcal{D}^{nh}$ is a subbundle of $TQ$ and
  $\mathcal{D}_d$ is a regular submanifold of $Q \times Q$.  In an
  analogous way to what happens with discrete mechanical systems and
  the discrete Lagrange--Poincar\'e systems
  in~\cite{ar:fernandez_tori_zuccalli-lagrangian_reduction_of_discrete_mechanical_systems_by_stages}
  (Example 3.12), a discrete nonholonomic mechanical system can be
  seen as a discrete Lagrange--D'Alembert--Poincar\'e system with
  $\phi=id_{Q}$ (so that $C'(E)=Q\times Q$), $\mathcal{P}=0$, the same
  $\mathcal{D}_d$ as kinematic constraints and
  $\mathcal{D}:=p_{1}^{*}(\mathcal{D}^{nh})\subset p_{1}^{*}(TQ)$. In
  this case, a discrete path in $C'(E)$ can be identified with path
  $q_\cdot=(q_0,\ldots,q_N)\in Q^{N+1}$ and the equations of
  motion~\eqref{eq:evolution} become
  \begin{gather*}
    (q_k,q_{k+1})\in \mathcal{D}_d \stext{ for all } k=0,\ldots, N-1
    \stext{ and }\\
    D_1L_d(q_k,q_{k+1})+D_2L_d(q_{k-1},q_k) \in (\mathcal{D}^{nh}_{q_k})^\circ
    \stext{ for all } k=1,\ldots, N-1
  \end{gather*}
  for $k=1,\ldots N-1$, which are the same equations of motion of the
  discrete nonholonomic mechanical system
  $(Q,L_d,\mathcal{D}_d,\mathcal{D}^{nh})$ given by (6)
  in~\cite{ar:cortes_martinez-non_holonomic_integrators} or (3)
  in~\cite{ar:fernandez_tori_zuccalli-lagrangian_reduction_of_discrete_mechanical_systems}.
\end{example}

\begin{remark}
  Under appropriate regularity conditions on the discrete lagrangian
  $L_d$ and dimensional relation on the constraints spaces, the
  existence of trajectories of a DLDPS is guaranteed in a neighborhood
  of a given trajectory.
\end{remark}


\subsection{Nonholonomic discrete mechanical systems with symmetry}
\label{sec:nonholonomic_discrete_mechanical_systems_with_symmetry}

Symmetries of a nonholonomic discrete mechanical system
(Example~\ref{ex:DNHMS_as_DLDPS}) were considered
in~\cite{ar:fernandez_tori_zuccalli-lagrangian_reduction_of_discrete_mechanical_systems}. Even
more, a reduction process was developed there so that a new
discrete-time dynamical system ---called the \jdef{reduced system}---
was constructed starting from a symmetric nonholonomic discrete
mechanical system and whose dynamics captured the essential features
of that of the original system. Unfortunately, that reduced system is
not usually a nonholonomic discrete mechanical system. The goal of
this section is to recall those constructions and results
from~\cite{ar:fernandez_tori_zuccalli-lagrangian_reduction_of_discrete_mechanical_systems}
and prove that, indeed, the reduced system can be interpreted as a
DLDPS whose trajectories in the sense of
Definition~\ref{def:trajectory_DLDPS} are the same as those of the
reduced system (in the sense
of~\cite{ar:fernandez_tori_zuccalli-lagrangian_reduction_of_discrete_mechanical_systems}).

Let $l^Q$ be a left $G$-action on $Q$ such that
$\pi^{Q,G} : Q \rightarrow Q/G$ is a principal $G$-bundle and fix a
discrete connection $\DC$ on this bundle. In this case, the
commutative diagram~\eqref{eq:diagram_ExM_to_reduced} turns into
\begin{equation}
  \xymatrix{ {Q\times Q} \ar[r]^(.4){\ti{\Phi}_{\DC}}_(.4){\sim}
    \ar[d]_{\pi^{Q\times Q,G}} \ar[dr]^(.55){\Upsilon_{\DC}} & 
    {(Q\times G) \times (Q/G)} \ar[d]^{\pi^{Q\times G,G}\times id_{Q/G}} \\
    {(Q\times Q)/G} \ar[r]_(.45){\Phi_{\DC}}^(.45){\sim} &
    {\ti{G} \times (Q/G)}}
\end{equation}
where $\ti{G} := (Q\times G)/G$ with $G$ acting on $Q$ by $l^{Q}$ and
on $G$ by conjugation and
\begin{equation} \label{eq:Upsilon-special_case-def}
  \Upsilon_{\DC}(q_0,q_1):=(\pi^{Q\times
    G,G}(q_0,\DC(q_0,q_1)),\pi^{Q,G}(q_1)).
\end{equation}

A Lie group $G$ is a symmetry group of the discrete nonholonomic
mechanical system $(Q,L_d,\mathcal{D}_d,\mathcal{D}^{nh})$ if
$\pi^{Q,G}:Q \rightarrow Q/G$ is a principal bundle, $L_d$ and
$\mathcal{D}_d$ are invariant by the diagonal action $l^{Q \times Q}$
and $\mathcal{D}^{nh}$ is invariant by the lifted action $l^{TQ}$.  By
the $G$-invariance of $L_d$, there is a well defined map
$\check{L}_d:\ti{G}\times (Q/G)\rightarrow \R$ such that
$\check{L}_d(v_0,r_1)=L_d(q_0,q_1)$ for any $(q_0,q_1)\in Q\times Q$
such that $\Upsilon_{\DC}(q_0,q_1)=(v_0,r_1)$. The actions associated
to $L_d$ and $\check{L}_d$ are
$S_d(q_\cdot) := \sum_k L_d(q_k,q_{k+1})$ and
$\check{S}_d(v_\cdot,r_\cdot):= \sum_k \check{L}_d(v_k,r_{r+1})$. Also
by $G$-invariance, being $\mathcal{D}_d\subset Q\times Q$ a regular
submanifold, $\mathcal{D}_d/G\subset (Q\times Q)/G$ is a regular
submanifold by Proposition~\ref{prop:bundle_image_submanifold}; as
$\Phi_\DC$ is a diffeomorphism,
$\check{\mathcal{D}}_d:=\Phi_\DC(\mathcal{D}_d/G)$ is a regular
submanifold of $\ti{G}\times (Q/G)$.

The next result
of~\cite{ar:fernandez_tori_zuccalli-lagrangian_reduction_of_discrete_mechanical_systems}
relates the variational principle that describes the dynamics of
$(Q,L_d,\mathcal{D}_d,\mathcal{D}^{nh})$ with a variational principle
for its reduced system defined on $\ti{G} \times (Q/G)$.

\begin{theorem}\label{thm:reduction_NH_discrete_mech_system}
  Let $G$ be a symmetry group of the discrete nonholonomic mechanical
  system $(Q,L_d,\mathcal{D}_d,\mathcal{D}^{nh})$. Let $\DC$ be a
  discrete connection on the principal $G$-bundle
  $\pi^{Q,G} : Q \rightarrow Q/G$. Let $q_{\cdot}$ be a discrete path
  in $Q$, $r_k:= \pi^{Q,G} (q_k)$,
  $w_k:= \DC(q_{k},q_{k+1})$ and
  $v_k:= \pi^{Q \times G,G}(q_k,w_k)$ be the corresponding discrete
  paths in $Q/G$, $G$ and $\ti{G}$. Then, the following statements
  are equivalent.
  \begin{enumerate}
  \item \label{it:reduction_NH_discrete_mech_system-not_reduced}
    $(q_{k},q_{k+1}) \in \mathcal{D}_d$ for all $k$ and $q_{\cdot}$
    satisfies the criticality condition
    $dS_d(q_{\cdot})(\delta q_{\cdot})=0$ for all fixed-endpoint
    variations $\delta q_{\cdot}$ such that
    $\delta q_k \in \mathcal{D}^{nh}_{q_k}$ for all $k$.
  \item \label{it:reduction_NH_discrete_mech_system-reduced}
    $(v_{k},r_{k+1}) \in \check{\mathcal{D}}_d$ for all $k$ and
    $d \check{S}_d(v_{\cdot},r_{\cdot})(\delta v_{\cdot},\delta
    r_{\cdot})=0$ for all $(\delta v_{\cdot},\delta r_{\cdot})$ such
    that
    \begin{equation}
      \label{eq:Def_variation}
      (\delta v_{k},\delta r_{k+1}):=
      d \Upsilon_{\DC}(q_{k},q_{k+1})(\delta q_{k}, \delta q_{k+1})
    \end{equation}
    for $k=0,\ldots,N-1$ and where $\delta q_{\cdot}$ is a fixed-endpoint
    variation on $q_{\cdot}$ such that
    $\delta q_{k} \in \mathcal{D}^{nh}_{q_k}$ for all $k$.
  \end{enumerate}
\end{theorem}

\begin{remark}
  Theorem~\ref{thm:reduction_NH_discrete_mech_system} is part of
  Theorem 5.11
  in~\cite{ar:fernandez_tori_zuccalli-lagrangian_reduction_of_discrete_mechanical_systems};
  this last result requires the additional data of a connection on the
  principal bundle $\pi^{Q,G} : Q \rightarrow Q/G$ to decompose the
  variations $\delta q_{\cdot}$ in horizontal and vertical parts. We
  have omitted this requirement and adapted the result accordingly.
\end{remark}

The reduced system associated to
$(Q,L_d,\mathcal{D}_d,\mathcal{D}^{nh})$ in Section 5
of~\cite{ar:fernandez_tori_zuccalli-lagrangian_reduction_of_discrete_mechanical_systems}
is the discrete-time dynamical system on $\ti{G} \times (Q/G)$ whose
trajectories are discrete paths that satisfy the variational principle
of point~\ref{it:reduction_NH_discrete_mech_system-reduced} in
Theorem~\ref{thm:reduction_NH_discrete_mech_system}. Next we construct
a DLDPS that will, eventually, be equivalent to this reduced system.
We define the fiber bundle $\phi:E\rightarrow M$ as
$p^{Q/G}:\ti{G}\rightarrow Q/G$, where
$p^{Q/G}(\pi^{Q \times G,G}(q,w)) := \pi^{Q,G}(q)$. The reduced
lagrangian $\check{L}_d:\ti{G} \times (Q/G)\rightarrow \R$ is a smooth
map on $C'(\ti{G})=\ti{G} \times (Q/G)$. We already saw that
$\mathcal{\check{D}}_{d}$ is a regular submanifold of
$C'(\ti{G})=\ti{G}\times (Q/G)$. The space
$\check{\mathcal{D}}:=d\Upsilon_{\DC}(p_{1}^{*}(\mathcal{D}^{nh}))$
defines a subbundle of $T(C'(\ti{G}))$ (this is a special case of
Lemma~\ref{le:check_D_subbundle} in
Section~\ref{subsec:Reduced_DLP_system}).

In order to define the NIVCM
$\check{\mathcal{P}}\in
\hom(\ti{p_{34}}^*(\mathcal{\check{D}}),\ker(dp^{Q/G}))$ we consider the
application $\Upsilon_{\DC}:Q\times Q\rightarrow \ti{G}\times (Q/G)$
given by~\eqref{eq:Upsilon-special_case-def} and define
\begin{equation}\label{eq:P_def}
  \check{\mathcal{P}}((v_{0},r_{1}),(v_{1},r_{2}))(\ti{\delta v_{1}},0) :=
  D_{2}(p_{1}\circ \Upsilon_{\DC})(q_{0},q_{1})(\delta q_{1})
  \in T_{v_{0}}\ti{G}  
\end{equation}
where $(q_{0},q_{1},q_{2})$ are such that
$(v_{0},r_{1})=\Upsilon_{\DC}(q_{0},q_{1})$ and
$(v_{1},r_{2})=\Upsilon_{\DC}(q_{1},q_{2})$ and
$\delta q_{1}\in \mathcal{D}^{nh}_{q_{1}}$ satisfies that
$\delta v_{1} = D_{1}(p_{1}\circ \Upsilon_{\DC})(q_{0},q_{1})(\delta
q_{1})$. That $\check{\mathcal{P}}$ is well defined follows from
Lemma~\ref{le:prop_chaining_map}.

In this way, we associate a DLDPS
$\mathcal{M} := (E, \check{L}_d, \check{\mathcal{D}}_d,
\check{\mathcal{D}}, \check{\mathcal{P}})$ to the reduced system and
we will prove that the trajectories of both systems coincide.

\begin{lemma}\label{le:prop_chaining_map}
  Let $Q$, $\mathcal{D}$, $\DC$ and $\Upsilon_{\DC}$ be as
  before. Then, the following statements are true.
  \begin{enumerate}
  \item \label{it:prop_chaining_map-iso} For
    $(q_{0},q_{1})\in Q\times Q$,
    $d \Upsilon_{\DC}|_{\mathcal{D}^{nh}_{q_0}\times
      \{0\}}(q_{0},q_{1}):(\mathcal{D}^{nh}_{q_0}\times \{0\})\subset
    T_{(q_{0},q_{1})}(Q \times Q)\rightarrow
    \check{\mathcal{D}}_{\Upsilon_{\DC}(q_{0},q_{1})}\subset
    p_{1}^{*}(T\ti{G})_{\Upsilon_{\DC}(q_{0},q_{1})}$ is an
    isomorphism of vector spaces.
		
  \item \label{it:prop_chaining_map-well_def} For
    $((v_{0},r_{1}),(v_{1},r_{2}))\in C''(E)$ and
    $(\delta v_{1},0)\in \check{\mathcal{D}}$ the map
    $\check{\mathcal{P}}$ given by \eqref{eq:P_def} is well defined
    and it is linear in $\delta v_{1}$.
		
  \item \label{it:prop_chaining_map-ker} For
    $((v_{0},r_{1}),(v_{1},r_{2})) \in C^{\prime \prime}(E)$ and
    $(\delta v_{1},0)\in \check{\mathcal{D}}$ we have
    \begin{equation*}
      dp^{Q/G}(v_{0})(\check{\mathcal{P}}((v_{0},r_{1}),(v_{1},r_{2}))(\delta
      v_{1},0)) =0.
    \end{equation*}
  \end{enumerate}
\end{lemma}

\begin{proof}
  See point~\ref{it:prop_Upsilon-iso} in Lemma~\ref{le:prop_Upsilon}
  for point~\ref{it:prop_chaining_map-iso} and
  Lemma~\ref{le:P_reduced_is_well_defined} for
  points~\ref{it:prop_chaining_map-well_def}
  and~\ref{it:prop_chaining_map-ker}.
\end{proof}

The following result
of~\cite{ar:fernandez_tori_zuccalli-lagrangian_reduction_of_discrete_mechanical_systems_by_stages}
proves that all discrete paths in $C'(E)=C'(\ti{G})$ arise from
discrete paths in $C'(id_Q)$.

\begin{lemma}
  Let $(v_{\cdot},r_{\cdot})$ be a discrete path in $C'(E)$ and
  $q_{0}\in Q$ such that $p^{Q/G}(v_{0})=\pi^{Q,G}(q_{0})$. Then,
  there exists a unique discrete path in $C'(id_Q:Q\rightarrow Q)$
  such that $\Upsilon_{\DC}(q_{k},q_{k+1})=(v_{k},r_{k+1})$ for all
  $k=0,\ldots,N-1$.
\end{lemma}

The following result compares the nonholonomic infinitesimal
variations with fixed endpoints on the discrete path
$(v_{\cdot},r_{\cdot})$ in $C'(E)=C'(\ti{G})$ for the system
$\mathcal{M}$, with the variations defined in
point~\ref{it:reduction_NH_discrete_mech_system-reduced} of
Theorem~\ref{thm:reduction_NH_discrete_mech_system} on the same
discrete path.

\begin{proposition}\label{prop:variations_fixed_endpoints}
  Let $(v_{\cdot},r_{\cdot})$ be a be discrete path in $C'(E)$ such
  that $(v_{k},r_{k+1})=\Upsilon_{\DC}(q_{k},q_{k+1})$ for all $k$,
  where $q_\cdot$ is a discrete path in $C'(id_Q)$. Then, the
  following statements are true.
	
  \begin{enumerate}
  \item \label{it::variations_fixed_endpoints-NH_variations} Given a
    fixed-endpoint variation $\delta q_{\cdot}$ in $\mathcal{D}^{nh}$
    on $q_{\cdot}$, the infinitesimal variation
    $(\delta v_{\cdot},\delta r_{\cdot})$ defined in
    point~\ref{it:reduction_NH_discrete_mech_system-reduced} of
    Theorem~\ref{thm:reduction_NH_discrete_mech_system} by
    \eqref{eq:Def_variation} is a nonholonomic infinitesimal variation
    on $(v_{\cdot},r_{\cdot})$ with fixed endpoints (in the sense of
    Definition~\ref{def:ininitesimal_variation}) for $\mathcal{M}$.
		
  \item \label{it::variations_fixed_endpoints-old_variations} Given a
    nonholonomic infinitesimal variation
    $(\delta v_\cdot, \delta r_\cdot)$ on $(v_{\cdot},r_{\cdot})$ with
    fixed endpoints (for $\mathcal{M}$), there exists a fixed-endpoint
    variation $\delta q_{\cdot}$ in $\mathcal{D}^{nh}$ on $q_{\cdot}$
    such that~\eqref{eq:Def_variation} is satisfied for all $k$.
  \end{enumerate}
\end{proposition}

\begin{proof}
  \begin{enumerate}
  \item Let $\delta q_{\cdot}$ be a fixed-endpoint variation on
    $q_{\cdot}$ in $Q$ such that
    $\delta q_k \in \mathcal{D}^{nh}_{q_k}$ and let
    $(\delta v_{\cdot},\delta r_{\cdot})$ be the variation defined by
    \eqref{eq:Def_variation} in terms of $\delta q_{\cdot}$. Let
    $(\ti{\delta v_{k}},0):=d\Upsilon_{\DC}(q_{k},q_{k+1})(\delta
    q_{k},0) \in d\Upsilon_{\DC}(p_{1}^{*}(\mathcal{D}^{nh}_{q_k})) =
    \check{\mathcal{D}}_{(v_{k},r_{k+1})}$ for $k=0,\ldots,N-1$.
		
    We want to see that $(\delta v_{\cdot},\delta r_{\cdot})$ is a
    nonholonomic infinitesimal variation on $(v_{\cdot},r_{\cdot})$
    with fixed endpoints. Recall that
    \begin{equation*}
      \begin{split}
        \Upsilon_{\DC}(q_{k},q_{k+1}) =& (\pi^{Q\times
          G,G}(q_{k},\DC(q_{k},q_{k+1})),\pi ^{Q,G}(q_{k+1})) \\=&
        ((p_{1}\circ
        \Upsilon_{\DC})(q_{k},q_{k+1}),\pi^{Q,G}(q_{k+1})),
      \end{split}
    \end{equation*}
    and given that $\delta q_{\cdot}$ is a fixed-endpoint variation,
    we notice that
    \begin{equation*}
      \begin{split}
        (\delta v_{N-1},\delta r_N) =& d
        \Upsilon_{\DC}(q_{N-1},q_{N})(\delta q_{N-1},\delta q_{N})
        \\=& d\Upsilon_{\DC}(q_{N-1},q_{N})(\delta q_{N-1},0) =
        (\ti{\delta v_{N-1}},0),
      \end{split}
    \end{equation*}
    and
    \begin{equation*}
      \begin{split}
        \delta v_{0} =& d(p_{1}\circ \Upsilon_{\DC})(q_{0},
        q_{1})(\delta q_{0}, \delta q_{1}) = d(p_{1}\circ
        \Upsilon_{\DC})(q_{0}, q_{1})(0, \delta q_{1}) \\=&
        D_{2}(p_{1}\circ \Upsilon_{\DC})(q_{0},q_{1})(\delta q_{1})
        =\check{\mathcal{P}}((v_{0},r_{1}),(v_{1},r_{2}))(\ti{\delta
          v_{1}},0).
      \end{split}
    \end{equation*}
    
    Also, taking into account that
    $\ti{\delta v_{k+1}} = D_{1}(p_{1}\circ
    \Upsilon_{\DC})(q_{k},q_{k+1})(\delta q_{k+1})$, for $k=1,\ldots,N-2$
    we have that
    \begin{equation*}
      \begin{split}
        (\delta v_{k},\delta r_{k+1}) =& d\Upsilon_{\DC}(q_{k},
        q_{k+1})(\delta q_{k},\delta q_{k+1}) \\=&
        d\Upsilon_{\DC}(q_{k}, q_{k+1})(\delta q_{k},0) +
        d\Upsilon_{\DC}(q_{k},q_{k+1})(0,\delta q_{k+1}) \\=&
        (\ti{\delta v_{k}},0) +
        D_{2}\Upsilon_{\DC}(q_{k},q_{k+1})(\delta q_{k+1})\\=&
        (\ti{\delta v_{k}},0) + (D_{2}(p_1\circ
        \Upsilon_{\DC})(q_{k},q_{k+1})(\delta q_{k+1}),\\&
        \phantom{(\ti{\delta v_{k}},0) + (} D_{2}(p_2\circ
        \Upsilon_{\DC})(q_{k},q_{k+1})(\delta q_{k+1})) \\=&
        (\ti{\delta v_{k}},0) +
        (\check{\mathcal{P}}((v_{k},r_{k+1}),(v_{k+1},r_{k+2}))(\ti{\delta
          v_{k+1}},0), \delta r_{k+1}).
      \end{split}
    \end{equation*}
    Thus, $(\delta v_\cdot,\delta r_\cdot)$ satisfies
    conditions~\eqref{eq:infinitesimal_varation_2}. By construcion of
    the discrete path $(v_{\cdot},r_{\cdot})$, $r_{k}=p^{Q/G}(v_{k})$
    for all $k$ and since
    $p^{Q/G}\circ p_1\circ \Upsilon_{\DC} = \pi^{Q,G}\circ p_1$, then
    \begin{equation}
      \label{eq:delta_r_k}
      \begin{split}
        dp^{Q/G}_{(p_{1}\circ\Upsilon_{\DC})(q_{k},q_{k+1})}
        (d(p_{1}\circ \Upsilon_{\DC})_{(q_{k},q_{k+1})}(\delta
        q_{k},\delta q_{k+1})) =& d\pi^{Q,G}(q_k)(\delta q_k) \\
        dp^{Q/G}(v_k)(\delta v_k) =& \delta r_k
      \end{split}
    \end{equation}
    so that $\delta r_{k}=dp^{Q/G}(v_{k})(\delta v_{k})$, where
    $\delta v_{k}$ is given by the condition
    \eqref{eq:Def_variation}. Hence $(\delta v_\cdot,\delta r_\cdot)$
    satisfies condition~\eqref{eq:infinitesimal_varation_1}, hence
    part~\ref{it::variations_fixed_endpoints-NH_variations} is true.
		
  \item We consider $(\delta v_{\cdot},\delta r_{\cdot})$ that
    satisfies \eqref{eq:infinitesimal_varation_1} and
    \eqref{eq:infinitesimal_varation_2} for some vectors
    $(\ti{\delta v_{k}},0)\in \check{\mathcal{D}}_{(v_k,r_{k+1})}$
    with $k=1,\ldots,N-1$.  Let
    $\delta q_{0} :=0 \in \mathcal{D}^{nh}_{q_{0}}$ and
    $\delta q_{N} := 0\in \mathcal{D}^{nh}_{q_{N}}$ and, for each
    $k=1,\ldots,N-1$, using point~\ref{it:prop_chaining_map-iso} of
    Lemma~\ref{le:prop_chaining_map}, let
    $\delta q_{k}\in \mathcal{D}^{nh}_{q_{k}}$ such that
    $d\Upsilon_{\DC}(q_{k},q_{k+1})(\delta q_k,0) = (\ti{\delta
      v_k},0)$.

    We have that $r_k=\phi(v_k)=p^{Q/G}(v_k)$ for all $k$. Also, using
    \eqref{eq:delta_r_k}, we have
    \begin{equation*}
      \begin{split}
        \delta r_{k} =& d \phi(v_k)(\delta v_k)=d p^{Q/G}(v_k) (\delta
        v_k)\\=& d p^{Q/G}(v_k) (\ti{\delta v_{k}} +
        \check{\mathcal{P}}((v_{k},r_{k+1})
        ,(v_{k+1},r_{k+2}))(\ti{\delta v_{k+1}},0)) \\=& d
        p^{Q/G}(v_k)(D_{1}(p_{1}\circ
        \Upsilon_{\DC})(q_{k},q_{k+1})(\delta q_{k}) +
        D_{2}(p_{1}\circ \Upsilon_{\DC})(q_{k},q_{k+1})(\delta
        q_{k+1})) \\=& d p^{Q/G}(v_k)(d(p_{1}\circ
        \Upsilon_{\DC})(q_{k},q_{k+1})(\delta q_{k},\delta q_{k+1}))
        \\=& d (p^{Q/G} \circ p_{1} \circ
        \Upsilon_{\DC})(q_{k},q_{k+1})(\delta q_{k},\delta q_{k+1}) =
        d \pi^{Q,G}(q_{k})(\delta q_{k}),
      \end{split}
    \end{equation*}
    and
    \begin{equation*}
      \begin{split}
        \delta v_{k} =& \ti{\delta v_{k}} +
        \check{\mathcal{P}}((v_{k},r_{k+1})
        ,(v_{k+1},r_{k+2}))(\ti{\delta v_{k+1}},0) \\=& d(p_{1}\circ
        \Upsilon_{\DC})(q_{k},q_{k+1})(\delta q_k,0) +
        D_{2}(p_{1}\circ \Upsilon_{\DC})(q_{k},q_{k+1})(\delta
        q_{k+1}) \\=& d(p_{1}\circ
        \Upsilon_{\DC})(q_{k},q_{k+1})(\delta q_k,0) + d(p_{1}\circ
        \Upsilon_{\DC})(q_{k},q_{k+1})(0,\delta q_{k+1}) \\=& d(
        p_{1}\circ \Upsilon_{\DC})(q_{k},q_{k+1})(\delta q_{k},\delta
        q_{k+1}).
      \end{split}
    \end{equation*}
    Finally, putting all together
    \begin{equation*}
      \begin{split}
        (\delta v_{k},\delta r_{k+1}) =& (d(p_{1}\circ
        \Upsilon_{\DC})(q_{k},q_{k+1})(\delta q_{k},\delta
        q_{k+1}),d\pi ^{Q,G}(q_{k+1})(\delta q_{k+1})) \\=& d\Upsilon
        _{\DC}(q_{k},q_{k+1})(\delta q_{k},\delta q_{k+1}),
      \end{split}
    \end{equation*}
    and we have verified that~\eqref{eq:Def_variation} is satisfied
    for all $k$. Hence,
    part~\ref{it::variations_fixed_endpoints-old_variations} is true.
  \end{enumerate}
\end{proof}

\begin{corollary}
  A discrete path $(v_{\cdot},r_{\cdot})$ is a trajectory of
  $\mathcal{M}$ if and only if it is a trajectory of the reduced
  system according to
  part~\ref{it:reduction_NH_discrete_mech_system-reduced} of
  Theorem~\ref{thm:reduction_NH_discrete_mech_system}.
\end{corollary}

\begin{proof}
  The equivalence between the two descriptions of a trajectory follows
  immediately by the correspondence of the infinitesimal variations
  established in Proposition~\ref{prop:variations_fixed_endpoints}.
\end{proof}

      
\section{Categorical formulation}
\label{sec:categorical_formulation}

\begin{definition}\label{def:category}
  We define the \textit{category of discrete
    Lagrange--D'Alembert--Poincar\'e systems} $\DLDPSC$ as the category
  whose objects are DLDPSs. Given
  $\mathcal{M},\mathcal{M}' \in \Ob_{\DLDPSC}$ with
  $\mathcal{M}=(E,L_d,\mathcal{D}_d,\mathcal{D},\mathcal{P})$ and
  $\mathcal{M}'=(E',L'_d,\mathcal{D}'_d,\mathcal{D}',\mathcal{P}')$ a
  map $\Upsilon : C'(E)\rightarrow C'(E')$ is a morphism in
  $\Mor_{\DLDPSC} (\mathcal{M},\mathcal{M}')$ if
  \begin{enumerate}
  \item \label{it:prop_morphism_1} $\Upsilon$ is a surjective submersion,
  \item \label{it:prop_morphism_2} $D_1(p_2 \circ \Upsilon) =0$,
  \item \label{it:prop_morphism_3} As maps from $C''(E)$ in $M'$
    \begin{equation}\label{eq:prop_morphism_3}
      p_{2}^{C'(E'),M'} \circ \Upsilon \circ p_{1}^{C''(E),C'(E)} = \phi' \circ p_{1}^{C'(E'),E'} \circ \Upsilon \circ p_{2}^{C''(E),C'(E)}
    \end{equation}
    where $p_{j}^{A,B} : A\rightarrow B$ are the maps induced by the
    canonical projections of the Cartesian product onto its factors,
  \item \label{it:prop_morphism_4} $L_d=L'_d \circ \Upsilon$,
  \item \label{it:prop_morphism_5} $\mathcal{D}'_d = \Upsilon(\mathcal{D}_d)$,
  \item \label{it:prop_morphism_6} $\mathcal{D}' = d\Upsilon(\mathcal{D})$,
  \item \label{it:prop_morphism_7} For all
    $(((\epsilon_0,m_1),(\epsilon_1,m_2)),(\delta \epsilon_1,0)) \in
    \ti{p_{34}}^*(\mathcal{D})$,
    \begin{equation}
      \label{eq:prop_7}
      \begin{split}
        \mathcal{P}'(\Upsilon^{(2)}&((\epsilon_0,m_1),
        (\epsilon_1,m_2)))(d\Upsilon(\epsilon_1,m_2)(\delta
        \epsilon_1,0))\\ =& d(p_1 \circ
        \Upsilon)(\epsilon_0,m_1)(\mathcal{P}((\epsilon_0,m_1),
        (\epsilon_1,m_2))(\delta
        \epsilon_1,0),d\phi(\epsilon_1)(\delta \epsilon_1))
      \end{split}
    \end{equation}
    where $\Upsilon \times \Upsilon$ defines a map
    $\Upsilon^{(2)}:C''(E)\rightarrow C''(E')$.
  \end{enumerate}
\end{definition}

We recall the statement of Lemma 4.3
of~\cite{ar:fernandez_tori_zuccalli-lagrangian_reduction_of_discrete_mechanical_systems_by_stages}.

\begin{lemma}\label{le:prop_upsilon}	
  Let
  $\Upsilon \in \Mor_{\DLDPSC}(\mathcal{M},\mathcal{M}'),
  ((\epsilon_0,m_1),(\epsilon_1,m_2))\in C''(E)$ and also
  $(\epsilon'_0,m'_1):=\Upsilon(\epsilon_0,m_1)$. Then, if
  $\delta \epsilon_1 \in T_{\epsilon_1} E$,
  \begin{equation*}
    D_2(p_2 \circ \Upsilon)(\epsilon_0,m_1)(d\phi(\epsilon_1)
    (\delta \epsilon_1)) = d\phi'(\epsilon'_1)(D_1(p_1 \circ \Upsilon)
    (\epsilon_1,m_2)(\delta\epsilon_1)).
  \end{equation*}
\end{lemma}

\begin{proposition}\label{prop:LPD_category}
  $\DLDPSC$ is a category with the standard composition of
  functions and identity mappings.
\end{proposition}

\begin{proof}
  This proof is analogous to the proof of Proposistion 4.4
  of~\cite{ar:fernandez_tori_zuccalli-lagrangian_reduction_of_discrete_mechanical_systems_by_stages}.
\end{proof}

\begin{remark}
  Any discrete Lagrange--Poincar\'e system~\cite[Definition
  3.4]{ar:fernandez_tori_zuccalli-lagrangian_reduction_of_discrete_mechanical_systems_by_stages}
  can be seen as a Lagrange--D'Alembert--Poincar\'e system ``without
  constraints'', that is, with $\mathcal{D}_d:=C'(E)$ and
  $\mathcal{D}:=p_1^*TE$. Also, with this interpretation any morphism
  of Lagrange--Poincar\'e systems is a morphism of
  Lagrange--D'Alembert--Poincar\'e systems. It is immediate to check
  that the category of Lagrange--Poincar\'e systems~\cite[Definition
  4.1]{ar:fernandez_tori_zuccalli-lagrangian_reduction_of_discrete_mechanical_systems_by_stages}
  is a full subcategory of $\DLDPSC$.
\end{remark}

\begin{lemma}\label{le:morphism}
  Let $\Upsilon'\in \Mor_{\DLDPSC}(\mathcal{M},\mathcal{M}')$ and
  $\Upsilon''\in \Mor_{\DLDPSC}(\mathcal{M},\mathcal{M}'')$ where
  $\mathcal{M} = (E,L_{d},\mathcal{D}_{d},\mathcal{D},\mathcal{P})$,
  $\mathcal{M}' =
  (E',L'_{d},\mathcal{D}'_{d},\mathcal{D}',\mathcal{P}')$ and
  $\mathcal{M}'' =
  (E'',L''_{d},\mathcal{D}''_{d},\mathcal{D}'',\mathcal{P}'')$. If
  $F:C'(E') \rightarrow C'(E'')$ is a smooth map such that the diagram
  \begin{equation*}
    \xymatrix{ {} & {C'(E)} \ar[dl]_{\Upsilon'} \ar[dr]^{\Upsilon''} & {}\\
      {C'(E')} \ar[rr]_{F} & {} & {C'(E'')}}
  \end{equation*}
  is commutative, then
  $F\in
  \Mor_{\DLDPSC}(\mathcal{M}',\mathcal{M}'')$. Also,
  if $F$ is a diffeomorphism, then $F$ is an isomorphism in
  $\DLDPSC$.
\end{lemma}

\begin{proof}
  The proof that $F$ satisfies the points \ref{it:prop_morphism_1} to
  \ref{it:prop_morphism_4} and \ref{it:prop_morphism_7} in Definition
  \ref{def:category} and the last assertion of the statement is the
  same as in the proof of Lemma 4.5
  in~\cite{ar:fernandez_tori_zuccalli-lagrangian_reduction_of_discrete_mechanical_systems_by_stages}.
  We want to prove that $F$ satisfies points \ref{it:prop_morphism_5}
  and \ref{it:prop_morphism_6} of Definition \ref{def:category}. Since
  $\Upsilon'\in \Mor_{\DLDPSC}(\mathcal{M},\mathcal{M}')$ and
  $\Upsilon''\in \Mor_{\DLDPSC}(\mathcal{M},\mathcal{M}'')$ and the
  previous diagram is commutative, we have that
  \begin{equation*}
    F(\mathcal{D}_{d}') = F(\Upsilon'(\mathcal{D}_{d})) =
    (F\circ \Upsilon')(\mathcal{D}_{d}) = \Upsilon''(\mathcal{D}_{d}) =
    \mathcal{D}''_{d},
  \end{equation*}
  and, then, point~\ref{it:prop_morphism_5} in
  Definition~\ref{def:category} is satisfied. Similarly, as $\Upsilon'$ and
  $\Upsilon''$ are morphisms in $\DLDPSC$, using the commutativity of
  the diagram, we have
  \begin{equation*}
    \mathcal{D}'' = d\Upsilon''(\mathcal{D}) = d(F\circ \Upsilon')(\mathcal{D})
    = dF(d\Upsilon'(\mathcal{D})) = dF(\mathcal{D}').
  \end{equation*}
  This proves that point~\ref{it:prop_morphism_6} in
  Definition~\ref{def:category} is satisfied.
\end{proof}

\begin{theorem}\label{thm:imp_traject_categ}
  Given $\Upsilon\in \Mor_{\DLDPSC}(\mathcal{M},\mathcal{M}')$ with
  $\mathcal{M}=(E,L_{d},\mathcal{D}_{d},\mathcal{D},\mathcal{P})$ and
  $\mathcal{M}'=(E',L'_{d},\mathcal{D}'_{d},\mathcal{D}',\mathcal{P}')$,
  let
  $(\epsilon_\cdot,m_\cdot) =
  ((\epsilon_0,m_1),\ldots,(\epsilon_{N-1},m_N))$ be a discrete path
  in $C'(E)$ and define
  $(\epsilon'_k,m'_{k+1}) := \Upsilon(\epsilon_k,m_{k+1})$ for
  $k=0,\ldots, N-1$.
  \begin{enumerate}
  \item \label{it:imp_traject_categ-M_to_M'} If
    $(\epsilon_\cdot, m_\cdot)$ is a trajectory of $\mathcal{M}$, then
    $(\epsilon'_\cdot, m'_\cdot)$ is a trajectory of $\mathcal{M}'$.
  \item \label{it:imp_traject_categ-M'_to_M} If
    $\mathcal{D}_d=\Upsilon^{-1}(\mathcal{D}_d')$ and
    $(\epsilon'_\cdot, m'_\cdot)$ is a trajectory of $\mathcal{M}'$,
    then $(\epsilon_\cdot, m_\cdot)$ is a trajectory of $\mathcal{M}$.
  \end{enumerate}
\end{theorem}

\begin{proof}
  By hypothesis, $(\epsilon_\cdot,m_\cdot)$ is a discrete path in
  $C'(E)$. It follows from its definition, the fact that $\Upsilon$ is
  a morphism and~\eqref{eq:prop_morphism_3} that
  $(\epsilon'_\cdot,m'_\cdot)$ is a discrete path in $C'(E)$.

  Assume that $(\delta \epsilon_\cdot,\delta m_\cdot)$ is an
  infinitesimal variation in $\mathcal{M}$ over
  $(\epsilon_\cdot,m_\cdot)$ and that
  $(\delta \epsilon'_\cdot,\delta m'_\cdot)$ is an infinitesimal
  variation in $\mathcal{M}'$ over $(\epsilon'_\cdot,m'_\cdot)$
  satisfying
  \begin{equation}\label{eq:image_of_variation_under_morphism}
    d\Upsilon(\epsilon_k,m_{k+1})(\delta \epsilon_k,\delta m_{k+1}) = 
    (\delta \epsilon'_k,\delta m'_{k+1}) \stext{ for } k=0,\ldots,N-1.
  \end{equation}
  Then, using the chain rule and that $L_d=L_d'\circ \Upsilon$, we see
  that
  \begin{equation}\label{eq:comparison_actions_Upsilon_related}
    dS_d(\epsilon_\cdot,m_\cdot)(\delta \epsilon_\cdot,\delta
    m_\cdot) = dS_d'(\epsilon'_\cdot,m'_\cdot)(\delta
    \epsilon'_\cdot, \delta m'_\cdot).
  \end{equation}

  In order to prove point~\ref{it:imp_traject_categ-M_to_M'}, we
  assume that $(\epsilon_\cdot,m_\cdot)$ is a trajectory of
  $\mathcal{M}$. Then $(\epsilon_k,m_{k+1}) \in \mathcal{D}_d$ for
  $k=0,\ldots,N-1$. Then, 
  \begin{equation}\label{eq:image_of_trajectory_in_kinematic_constraint}
    (\epsilon'_k,m'_{k+1}) = \Upsilon(\epsilon_k,m_{k+1}) \in
    \Upsilon(\mathcal{D}_d) = \mathcal{D}'_d \stext{ for } k=0,\ldots,N-1.
  \end{equation}

  Let $(\delta \epsilon'_\cdot, \delta m'_\cdot)$ be an infinitesimal
  variation with fixed endpoints in $\mathcal{M}'$ over the discrete
  path $(\epsilon'_\cdot,m'_\cdot)$. That is, there are
  $(\ti{\delta \epsilon'_k},0)\in
  \mathcal{D}'_{(\epsilon'_k,m'_{k+1})}$ for $k=1,\ldots,N-1$ such
  that~\eqref{eq:infinitesimal_varation_1}
  and~\eqref{eq:infinitesimal_varation_2} hold with
  $\delta \epsilon_k'$ and $\ti{\delta \epsilon_k'}$ instead of
  $\delta \epsilon_k$ and $\ti{\delta \epsilon_k}$.

  By morphism's property~\ref{it:prop_morphism_6} applied to
  $\Upsilon$, there exist
  $(\ti{\delta \epsilon_k},0)\in \mathcal{D}_{(\epsilon_k,m_{k+1})}$
  such that
  $d\Upsilon(\epsilon_k,m_{k+1})(\ti{\delta \epsilon_k},0) =
  (\ti{\delta \epsilon'_k},0)$ for $k=1,\ldots,N-1$; we fix one such
  vector for each $k$. Next apply~\eqref{eq:infinitesimal_varation_1}
  and~\eqref{eq:infinitesimal_varation_2} to define an infinitesimal
  variation $(\delta \epsilon_\cdot, \delta m_\cdot)$ on
  $(\epsilon_\cdot, m_\cdot)$ with fixed endpoints based on the
  $\ti{\delta \epsilon_\cdot}$ constructed above.

  Direct computations using the morphism properties of $\Upsilon$ show
  that condition~\eqref{eq:image_of_variation_under_morphism} holds
  for the $(\delta \epsilon_\cdot,\delta m_\cdot)$ and
  $(\delta \epsilon'_\cdot, \delta m'_\cdot)$ variations. Then,
  using~\eqref{eq:comparison_actions_Upsilon_related},
  \begin{equation*}
    dS_d'(\epsilon'_\cdot,m'_\cdot)(\delta \epsilon'_\cdot, \delta
    m'_\cdot) = dS_d(\epsilon_\cdot,m_\cdot)(\delta \epsilon_\cdot,\delta
    m_\cdot) = 0,
  \end{equation*}
  where the last equality holds because
  $(\delta \epsilon_\cdot,\delta m_\cdot)$ is an infinitesimal
  variation with fixed endpoints in $\mathcal{M}$ over
  $(\epsilon_\cdot,m_\cdot)$, that is a trajectory of
  $\mathcal{M}$. Finally, as
  $(\delta \epsilon'_\cdot, \delta m'_\cdot)$ was an arbitrary
  infinitesimal variation with fixed endpoints in $\mathcal{M}'$ over
  the path $(\epsilon'_\cdot,m'_\cdot)$, and we
  have~\eqref{eq:image_of_trajectory_in_kinematic_constraint}, we
  conclude that $(\epsilon'_\cdot,m'_\cdot)$ is a trajectory of
  $\mathcal{M}'$. This proves
  point~\ref{it:imp_traject_categ-M_to_M'}.

  In order to prove point~\ref{it:imp_traject_categ-M'_to_M}, assume
  that $(\epsilon'_\cdot, m'_\cdot)$ is a trajectory of
  $\mathcal{M}'$. Then, as
  $\mathcal{D}_d=\Upsilon^{-1}(\mathcal{D}'_d)$ and
  $(\epsilon'_k,m'_{k+1})\in \mathcal{D}_d'$ for $k=0,\ldots,N-1$, we
  have that
  $\Upsilon(\epsilon_k,m_{k+1}) = (\epsilon'_k,m'_{k+1}) \in
  \mathcal{D}_d'$, so that
  $(\epsilon_k,m_{k+1})\in
  \Upsilon^{-1}(\mathcal{D}_d')=\mathcal{D}_d$ for $k=0,\ldots,N-1$.

  An argument similar to the one used in the proof of
  point~\ref{it:imp_traject_categ-M_to_M'} shows that
  $(\epsilon_\cdot, m_\cdot)$ satisfies the criticality condition in
  $\mathcal{M}$, so that it is a trajectory of $\mathcal{M}$, thus
  proving point~\ref{it:imp_traject_categ-M'_to_M}.
\end{proof}

The following result, whose proof is immediate, is useful when working
with concrete DLDPSs.

\begin{lemma}\label{le:DLDPS_induced_by_isomorphism_of_FB}
  Let $\phi:E\rightarrow M$ and $\phi':E'\rightarrow M'$ be two fiber
  bundles and $(F,f)$ be a fiber bundle isomorphism from $E$ to
  $E'$. For any
  ${\mathcal M} = (E,L_d,{\mathcal D}_d,{\mathcal D},\mathcal{P}) \in
  \Ob_{\DLDPSC}$, let $L_d' := L_d\circ (F\times f)^{-1}$,
  ${\mathcal D}_d' := (F\times f)({\mathcal D}_d)$,
  ${\mathcal D}' := d(F\times f)({\mathcal D})$ and ${\mathcal P}'$ so
  that, for all
  $(((\epsilon_0',m_1'),(\epsilon_1',m_2')),(\delta \epsilon_1',0))\in
  \ti{p_{34}}^*({\mathcal D}')$, we have
  \begin{equation*}
    \begin{split}
      &{\mathcal P}'((\epsilon_0',m_1'),(\epsilon_1',m_2'))(\delta
      \epsilon_1',0) \\&= dF(F^{-1}(\epsilon_0'))(({\mathcal
        P((F^{-1}(\epsilon_0'),f^{-1}(m_1')),(F^{-1}(\epsilon_1'),f^{-1}(m_2')))
        (dF^{-1}(\epsilon_1')(\delta \epsilon_1'),0)}).
    \end{split}
  \end{equation*}
  Then,
  ${\mathcal M}' := (E',L_d',{\mathcal D}_d',{\mathcal D}',{\mathcal
    P'}) \in \Ob_{\DLDPSC}$ and
  $F\times f\in \hom_{\DLDPSC}({\mathcal M},{\mathcal M}')$ is an
  isomorphism in $\DLDPSC$. In particular, the corresponding sections
  $\nu_d$ and $\nu_d'$ defined by~\eqref{eq:section_of_motion-def}
  satisfy $(\widetilde{F\times f}^{(2)})^*(\nu_d') = \nu_d$ or, explicitly,
  \begin{gather*}
    \nu_d((\epsilon_0,m_1),(\epsilon_1,m_2))(\delta \epsilon_1,0) =
    \nu_d'((F(\epsilon_0),f(m_1)),(F(\epsilon_1),f(m_2)))(dF(\epsilon_1)(\delta
    \epsilon_1),0).
  \end{gather*}
\end{lemma}


\section{Reduction of discrete Lagrange--D'Alembert--Poincar\'e
  systems}
\label{sec:Reduction_discrete_LDP_systems}

In this section we introduce the notion of symmetry group of a DLDPS
and a reduction procedure to associate a ``reduced'' DLDPS system to a
symmetric one. In addition, we prove that the reduction procedure is a
morphism in $\DLDPSC$ and compare the dynamics of the reduced system
to that of the original one.

\subsection{Discrete Lagrange--D'Alembert--Poincar\'e systems with
  symmetry}
\label{subsec:DLP_systems_with_symmetry}

Let $G$ be a Lie group that acts on the fiber bunble $(E,M,\phi,F)$ as
in Definition~\ref{def:G_action_on_fiber_bundle}. We consider the
$G$-actions on $C'(E)$ and $C''(E)$ given by
\begin{equation}\label{eq:G-action_C'}
  l_{g}^{C'(E)}(\epsilon_0,m_1) := (l_{g}^{E}(\epsilon_0),l_{g}^{M}(m_1)),
\end{equation}
\begin{equation}\label{eq:G-action_C''}
  l_{g}^{C''(E)}((\epsilon_0,m_1),(\epsilon_1,m_2)) :=
  (l_g^{C'(E)}(\epsilon_0,m_1),l_g^{C'(E)}(\epsilon_1,m_2)).
\end{equation}
Also, we consider the $G$-actions on $\ker(d\phi)\subset TE$ and on
$\ti{p_{34}}^*T(C'(E))\subset TC''(E)$ given by
\begin{equation}\label{eq:G-action_TE}
  l_{g}^{TE}(\epsilon_0,\delta \epsilon_0) := 
  dl_{g}^{E}(\epsilon_0)(\delta \epsilon_0)
\end{equation}
\begin{equation}\label{eq:G-action_T(C'(E))}
  l_{g}^{T(C'(E))}(\epsilon_0,m_1)(\delta \epsilon_0,\delta m_1) :=
  (dl_{g}^{E}(\epsilon_0)(\delta \epsilon_0),dl_{g}^{M}(m_1)(\delta m_1))
\end{equation}
\begin{equation}\label{eq:G-action_p3(TE)}
  \begin{split}
    l_{g}^{\ti{p_{34}}^{*}T(C'(E))}((\epsilon_0,m_1),&(\epsilon_1,m_2),(\delta
    \epsilon_1,\delta m_2)) :=
    \\&(l_{g}^{C''(E)}((\epsilon_0,m_1),(\epsilon_1,m_2)),
    dl_{g}^{C'(E)}(\epsilon_1,m_2)(\delta \epsilon_1,\delta m_2)).
  \end{split}
\end{equation}

\begin{lemma}\label{le:prop_Upsilon}
  Let $G$ be a Lie group acting on the fiber bundle
  $\phi:E\rightarrow M$ and $\DC$ a discrete connection on the
  principal $G$-bundle $\pi^{M,G}:M\rightarrow M/G$. We define
  $\Upsilon_{\DC}^{(2)}:C''(E)\rightarrow C''(\ti{G}_E)$ as the
  restriction of
  $\Upsilon_{\DC}\times\Upsilon_{\DC}:C'(E)\times C'(E)\rightarrow
  C'(\ti{G}_E)\times C'(\ti{G}_E)$ to the corresponding spaces, where
  $\Upsilon_{\DC}$ is defined by \eqref{eq:Upsilon_DC-def}. Then,
  \begin{enumerate}
  \item \label{it:prop_Upsilon-well_def}
    $\Upsilon_{\DC}^{(2)}$ is well defined.
  \item \label{it:prop_Upsilon-iso}
    $d\Upsilon_{\DC}(\epsilon_0,m_1)|_{(p_{1}^*TE)_{(\epsilon_0,m_1)}}:
    (p_{1}^*TE)_{(\epsilon_0,m_1)}\rightarrow
    (p_{1}^*T(\ti{G}_E))_{\Upsilon_{\DC}(\epsilon_0,m_1)}$ is an
    isomorphism of vector spaces for every
    $(\epsilon_0,m_1)\in C'(E)$.
  \item \label{it:prop_Upsilon-ppal_bundle}
    $\Upsilon_{\DC}^{(2)}:C''(E)\rightarrow C''(\ti{G}_E)$ is a
    principal $G$-bundle with structure group $G$. In particular,
    $C''(E)/G\simeq C''(\ti{G}_E)$.
  \item \label{it:prop_Upsilon-lifting} For each
    $((v_0,r_1),(v_1,r_2))\in C''(\ti{G}_E)$ and
    $(\epsilon_0,m_1)\in C'(E)$ such that
    $\Upsilon_{\DC}(\epsilon_0,m_1)=(v_0,r_1)$, there is a unique pair
    $(\epsilon_1,m_2)\in C'(E)$ such that
    $((\epsilon_0,m_1),(\epsilon_1,m_2))\in C''(E)$ and
    $\Upsilon_{\DC}^{(2)}((\epsilon_0,m_1),(\epsilon_1,m_2)) =
    ((v_0,r_1),(v_1,r_2))$.
  \end{enumerate}
\end{lemma}

\begin{proof}
  This result is almost identical to Lemma 5.1
  in~\cite{ar:fernandez_tori_zuccalli-lagrangian_reduction_of_discrete_mechanical_systems_by_stages},
  the only difference being that, here, we are using affine discrete
  connections instead of discrete connections. It is easy to see that
  the proof of Lemma 5.1 remains valid for affine discrete
  connections.
\end{proof}

\begin{proposition}\label{prop:path_relation}
  Let $G$ be a Lie group acting on the fiber bundle
  $\phi:E\rightarrow M$ and $\DC$ be a discrete connection on the
  principal $G$-bundle $\pi^{M,G}:M\rightarrow M/G$. Given a discrete
  path $(v_{\cdot},r_{\cdot})=((v_0,r_1),\ldots,(v_{N-1},r_N))$ in
  $C'(\ti{G}_E)$ and $(\ti{\epsilon}_0,\ti{m}_1)\in C'(E)$ such that
  $\Upsilon_{\DC}(\ti{\epsilon}_0,\ti{m}_1)=(v_0,r_1)$, there is a
  unique discrete path $(\epsilon_{\cdot},m_{\cdot})\in C'(E)$ such
  that $(\epsilon_0,m_1)=(\ti{\epsilon_0},\ti{m_1})$ and
  $\Upsilon_{\DC}(\epsilon_k,m_{k+1})=(v_k,r_{k+1})$ for all $k$.
\end{proposition}

\begin{proof}
  This is Proposition 5.2
  in~\cite{ar:fernandez_tori_zuccalli-lagrangian_reduction_of_discrete_mechanical_systems_by_stages}
  except for using affine discrete connections instead of discrete
  connections, which doesn't alter the proof.
\end{proof}

\begin{definition}\label{def:G_symmetry_DLDPS}
  Let
  $\mathcal{M}=(E,L_d,\mathcal{D}_d,\mathcal{D},\mathcal{P})\in\Ob_{\DLDPSC}$.
  A Lie group $G$ is a \jdef{symmetry group} of $\mathcal{M}$ if
  \begin{enumerate}
  \item \label{it:G_symmetry_DLDPS-act_on_bundle} $G$ acts on
    $\phi:E\rightarrow M$
    (Definition~\ref{def:G_action_on_fiber_bundle}),
  \item $L_{d}$ is $G$-invariant by the action
    $l^{C'(E)}$~\eqref{eq:G-action_C'},
  \item $\mathcal{D}_d$ is $G$-invariant by the action
    $l^{C'(E)}$~\eqref{eq:G-action_C'},
  \item $\mathcal{D}$ is $G$-invariant by the lifted action
    $l^{TE}$~\eqref{eq:G-action_TE},
  \item \label{it:G_symmetry_DLDPS-P_equiv} $\mathcal{P}$ is $G$-equivariant for the
    actions $l^{\ti{p_{34}}^{*}T(C'(E))}$~\eqref{eq:G-action_p3(TE)} and
    $l^{TE}$~\eqref{eq:G-action_TE}.
  \end{enumerate}
\end{definition}

\begin{remark}\label{remark:G_symmetry_DNHMS}
  In the context of Example~\ref{ex:DNHMS_as_DLDPS}, if $G$ is a
  symmetry group of the nonholonomic discrete mechanical system
  $(Q,L_d,\mathcal{D}_d,\mathcal{D}^{nh})$ in the sense
  of~\cite{ar:fernandez_tori_zuccalli-lagrangian_reduction_of_discrete_mechanical_systems},
  then it is a symmetry group of
  $(Q,L_d,\mathcal{D}_d,\mathcal{D},0) \in \Ob_{\DLDPSC}$ in the sense
  of Definition~\ref{def:G_symmetry_DLDPS}.
\end{remark}

\begin{lemma}\label{le:equivalences}
  Let
  $\mathcal{M} = (E,L_d,\mathcal{D}_d,\mathcal{D},\mathcal{P})\in
  \Ob_{\DLDPSC}$ and $G$ be a Lie group. Then, for $g\in G$, if
  $\Upsilon := l_{g}^{C'(E)}$ and $\mathcal{M}' = \mathcal{M}$,
  \begin{enumerate}
  \item \label{it:equivalences-Pt_1} $\mathcal{D}_d$ is $G$-invariant
    by the diagonal action \eqref{eq:G-action_C'} if and only if
    point~\ref{it:prop_morphism_5} in Definition~\ref{def:category} is
    satisfied for $\Upsilon$,
		
  \item \label{it:equivalences-Pt_2} $\mathcal{D}$ is $G$-invariant by
    the lifted action~\eqref{eq:G-action_T(C'(E))} if and only if
    point~\ref{it:prop_morphism_6} in Definition~\ref{def:category} is
    satisfied for $\Upsilon$,
		
  \item \label{it:equivalences-Pt_3}
    Point~\ref{it:G_symmetry_DLDPS-P_equiv} in
    Definition~\ref{def:G_symmetry_DLDPS} is equivalent to
    point~\ref{it:prop_morphism_7} in Definition~\ref{def:category}
    for $\Upsilon$.
  \end{enumerate}
\end{lemma}

\begin{proof}
  Points~\ref{it:equivalences-Pt_1} and~\ref{it:equivalences-Pt_2} are
  directly satisfied by the definitions of $l_{g}^{C'(E)}$ and
  $l_{g}^{T(C'(E))}$.

  To prove point~\ref{it:equivalences-Pt_3} we start by noting that
  \begin{equation*}
    \begin{split}
      \mathcal{P}(l_{g}^{C''(E)}((\epsilon_0,m_1),
      (\epsilon_1,m_2)))&(dl_{g}^{C'(E)}(\epsilon_1,m_2)(\delta\epsilon_1,0))
      \\=& \mathcal{P}(l_{g}^{\ti{p_{34}}^{*}T(C'(E))}((\epsilon_0,m_1),
      (\epsilon_1,m_2)),(\delta\epsilon_1,0))
    \end{split}
  \end{equation*}
  and
  \begin{equation*}
    \begin{split}
      d(p_{1}\circ
      l_{g}^{C'(E)})(\epsilon_0,m_1)(\mathcal{P}((\epsilon_0,m_1),&
      (\epsilon_1,m_2))(\delta\epsilon_1,0),d\phi(\epsilon_1)(\delta\epsilon_1))
      \\=& dl_{g}^{E}(\epsilon_0)(\mathcal{P}((\epsilon_0,m_1),
      (\epsilon_1,m_2))(\delta\epsilon_1,0)).
    \end{split}
  \end{equation*}
  Then, by point~\ref{it:prop_morphism_7} in
  Definition~\ref{def:category} for $\Upsilon:=l_{g}^{C'(E)}$ and
  $\mathcal{M}'=\mathcal{M}$ we have that the first members of the
  previous identities are the same, and by
  point~\ref{it:G_symmetry_DLDPS-P_equiv} of
  Definition~\ref{def:G_symmetry_DLDPS} the last members of the
  previous identities are the same, proving the equivalence of the
  conditions.
\end{proof}

\begin{proposition}
  Let
  $\mathcal{M} =
  (E,L_d,\mathcal{D}_d,\mathcal{D},\mathcal{P})\in\Ob_{\DLDPSC}$ and
  $G$ be a Lie group. Then, $G$ is a symmetry group of $\mathcal{M}$
  if and only if $G$ acts on the fiber bundle $\phi:E\rightarrow M$
  and $l_{g}^{C'(E)}\in \Mor_{\DLDPSC}(\mathcal{M},\mathcal{M})$ for
  all $g\in G$.
\end{proposition}

\begin{proof}
  Assume that $G$ is a symmetry group of $\mathcal{M}$. Then, by
  definition, $G$ acts on the fiber bundle $\phi:E\rightarrow M$. We
  have to prove that
  $l_{g}^{C'(E)}\in \Mor_{\DLDPSC}(\mathcal{M},\mathcal{M})$ for all
  $g\in G$.
	
  Proving that $l_{g}^{C'(E)}$ satisfies
  conditions~\ref{it:prop_morphism_1} to~\ref{it:prop_morphism_4} of
  Definition~\ref{def:category} is analogous to what was done in the
  proof of the Proposition 5.6
  in~\cite{ar:fernandez_tori_zuccalli-lagrangian_reduction_of_discrete_mechanical_systems_by_stages}.
  Lemma~\ref{le:equivalences} proves that $l_{g}^{C'(E)}$ satisfies
  the remaining conditions of the Definition~\ref{def:category}. Thus,
  $l_{g}^{C'(E)}\in \Mor_{\DLDPSC}(\mathcal{M},\mathcal{M})$.
	
  Conversely, if $G$ acts on the fiber bundle $\phi:E\rightarrow M$
  and $l_{g}^{C'(E)}\in \Mor_{\DLDPSC}(\mathcal{M},\mathcal{M})$, the
  first condition of Definition~\ref{def:G_symmetry_DLDPS} is
  satisfied and the remaining conditions follow from morphism's
  properties and Lemma~\ref{le:equivalences}.
\end{proof}


\subsection{Reduced discrete Lagrange--D'Alembert--Poincar\'e system}
\label{subsec:Reduced_DLP_system}

Let $G$ be a symmetry group of
$\mathcal{M} = (E,L_d,\mathcal{D}_d,\mathcal{D},\mathcal{P}) \in
\Ob_{\DLDPSC}$. Since $G$ acts on $(E,M,\phi,F)$ the conjugate bundle
$(\ti{G}_E,M/G,p^{M/G},F\times G)$ is a fiber bundle (see Section 9
in~\cite{ar:fernandez_tori_zuccalli-lagrangian_reduction_of_discrete_mechanical_systems_by_stages}).

Let $\DC$ be a discrete connection on the principal $G$-bundle
$\pi^{M,G}:M\rightarrow M/G$ and
$\Upsilon _{\DC}:E\times M\rightarrow\ti{G}_{E}\times (M/G)$ be the
map defined by~\eqref{eq:Upsilon_DC-def} that is a principal bundle
with structure group $G$ by Lemma~\ref{le:Upsilon_fiber_bundle}.

We define $\check{L}_d:\ti{G}_{E}\times(M/G) \rightarrow \R$ by
$\check{L}_d(v_0,r_1):=L_d(\epsilon_0,m_1)$ for any
$(\epsilon_0,m_1) \in\Upsilon_{\DC}^{-1}(v_0,r_1)$ that, by the
$G$-invariance of $L_d$, is well defined. Hence,
$\check{L}_d\circ \Upsilon_{\DC} = L_d$.

\begin{lemma}\label{le:check_D_submanifold}
  The space $\check{\mathcal{D}}_d := \Upsilon_{\DC}(\mathcal{D}_d)$
  is a regular submanifold of $\ti{G}_E \times(M/G)$. Also,
  $\mathcal{D}_d = \Upsilon_{\DC}^{-1}(\check{\mathcal{D}}_d)$.
\end{lemma}

\begin{proof}
  Since $\Upsilon_{\DC}: E \times M \rightarrow \ti{G}_E \times(M/G)$
  is a principal $G$-bundle and $\mathcal{D}_d$ is a $G$-invariant
  regular submanifold of $E \times M$, by
  Proposition~\ref{prop:bundle_image_submanifold},
  $\check{\mathcal{D}}_d$ is a regular submanifold of
  $\ti{G}_E \times(M/G)$. Also, as $\mathcal{D}_d$ is $G$-invariant
  and $\Upsilon_{\DC}$ is a principal $G$-bundle, we have
  $\mathcal{D}_d = \Upsilon_{\DC}^{-1}(\Upsilon_{\DC}(\mathcal{D}_d))
  = \Upsilon_{\DC}^{-1}(\check{\mathcal{D}}_d)$.
\end{proof}

\begin{lemma}\label{le:check_D_subbundle}
  The space $\check{\mathcal{D}}:=d\Upsilon_{\DC}(\mathcal{D})$ is a
  subbunble of $p_{1}^*(T(\ti{G}_E))$ where
  $p_{1}:C'(\ti{G}_E)\rightarrow \ti{G}_E$ is the projection onto the
  first factor. In addition,
  $\rank(\check{\mathcal{D}}) = \rank(\mathcal{D})$.
\end{lemma}

\begin{proof}	
  Given $(v,r) \in C'(\ti{G}_E)$ and $(\epsilon,m) \in C'(E)$ such
  that $(v,r)=\Upsilon_{\DC}(\epsilon,m)$, we want to prove
  that the subspace
  $d\Upsilon_{\DC}(v,r)(\mathcal{D}_{(\epsilon,m)})\subset
  T_{(v,r)}C'(\ti{G}_E)$ is independent of the particular choice of
  $(\epsilon,m)$ in $\Upsilon_{\DC}^{-1}(v,r)$.
	
  If $(v,r) = \Upsilon_{\DC}(\bar{\epsilon},\bar{m})$ then, since
  $\Upsilon_{\DC}:C'(E)\rightarrow C'(\ti{G}_E)$ is a $G$-principal
  bundle, there is $g \in G$ such that
  $(\bar{\epsilon},\bar{m}) = l_{g}^{C'(E)}(\epsilon,m)$. Since
  $\mathcal{D}$ is $G$-invariant we have that
  \begin{equation*}
    \mathcal{D}_{(\bar{\epsilon},\bar{m})} = \mathcal{D}_{l_{g}^{C'(E)}(\epsilon,m)}
    = dl_{g}^{C'(E)}(\epsilon,m)(\mathcal{D}_{(\epsilon,m)}),
  \end{equation*}
  and, as $\Upsilon_{\DC}$ is $G$-invariant,
  \begin{equation*}
    \begin{split}
      d\Upsilon_{\DC}(\bar{\epsilon},\bar{m})
      (\mathcal{D}_{(\bar{\epsilon},\bar{m})}) =&
      d\Upsilon_{\DC}(l_{g}^{C'(E)}(\epsilon,m))
      (dl_{g}^{C'(E)}(\mathcal{D}_{(\epsilon,m)})) \\=&
      d(\Upsilon_{\DC} \circ
      l_{g}^{C'(E)})(\epsilon,m)(\mathcal{D}_{(\epsilon,m)}) \\=&
      d\Upsilon_{\DC}(\epsilon,m)(\mathcal{D}_{(\epsilon,m)}).
    \end{split}
  \end{equation*}
  Then $d\Upsilon_{\DC}(\epsilon,m)(\mathcal{D}_{(\epsilon,m)})$ is a
  vector subspace of $T_{(v,r)}(C'(\ti{G}_E))$ for each
  $(v,r)\in C'(\ti{G}_E)$ and is independent of the particular
  $(\epsilon,m)$ chosen in $\Upsilon_{\DC}^{-1}(v,r)$; we call it
  $\check{\mathcal{D}}_{(v,r)}$. This construction gives a fiberwise
  vector structure to $\check{\mathcal{D}}$. That
  $\rank(\check{\mathcal{D}}) = \rank(\mathcal{D})$ follows
  immediately from point~\ref{it:prop_Upsilon-iso} in
  Lemma~\ref{le:prop_Upsilon}.
	
  Now we need to check that for every $(v,r)\in C'(\ti{G}_E)$ there
  exist smooth sections defined in an open neighborhood of $(v,r)$
  that generate $\check{\mathcal{D}}_{(v',r')}$ for all $(v',r')$ in
  that neighborhood. To do this, notice that given
  $(v,r)\in C'(\ti{G}_E)$, for any
  $(\epsilon,m)\in \Upsilon_{\DC}^{-1}(v,r)$, as $\mathcal{D}$ is a
  subbundle of $T(C'(E))$, there is an open neighborhood
  $U\subset C'(E)$ of $(\epsilon,m)$ and
  $d = \dim(\mathcal{D}_{(\epsilon,m)})$ smooth local sections
  $\sigma_1,\ldots,\sigma_d :U\rightarrow T(C'(E))$ such that
  $\{\sigma_1(\epsilon',m'),\ldots,\sigma_d(\epsilon',m')\}$ is a
  basis of $\mathcal{D}_{(\epsilon',m')}$ for each
  $(\epsilon',m')\in U$ (Lemma 10.32
  in~\cite{bo:lee-introduction_to_smooth_manifolds}).  Then,
  $\Upsilon_{\DC}(U)$ is an open neighborhood of $(v,r)$ (because
  $\Upsilon_{\DC}$, being a principal bundle map, is an open map; see
  Lemma 21.1~\cite{bo:lee-introduction_to_smooth_manifolds}). In
  addition, as $\Upsilon_{\DC}$ is a principal bundle, there is an
  open neighborhood $V\subset \Upsilon_{\DC}(U)$ of $(v,r)$ and a
  smooth section $\Sigma:V\rightarrow U$ of $\Upsilon_{\DC}$. Define
  $\eta_j := d\Upsilon_{\DC} \circ \sigma_j \circ \Sigma$,
  $j=1,\ldots,d$, which are smooth sections over $V$ of
  $TC'(\ti{G}_E)$ such that, for each $(v',r')\in V$,
  $\{\eta_1(v',r'),\ldots, \eta_d(v',r')\}$ generates
  $\check{\mathcal{D}}_{(v',r')}$.
\end{proof}

\begin{remark}\label{rem_isomorphism}
  As, by point~\ref{it:prop_Upsilon-iso} of
  Lemma~\ref{le:prop_Upsilon},
  $d\Upsilon_{\DC}(\epsilon_0,m_1):(p_{1}^*TE)_{(\epsilon_0,m_1)}\rightarrow
  (p_{1}^*T(\ti{G}_E))_{\Upsilon_{\DC}(\epsilon_0,m_1)}$ is an
  isomorphism,
  $d\Upsilon_{\DC}(\epsilon_0,m_1)|_{\mathcal{D}_{(\epsilon_0,m_1)}}$
  is an isomorphism from $\mathcal{D}_{(\epsilon_0,m_1)}$ onto
  $\check{\mathcal{D}}_{\Upsilon_{\DC}(\epsilon_0,m_1)}$.
\end{remark}

As, by point~\ref{it:prop_Upsilon-ppal_bundle} of
Lemma~\ref{le:prop_Upsilon} $\Upsilon_{\DC}^{(2)}$ is a principal
$G$-bundle, given $((v_0,r_1),(v_1,r_2)) \in C''(\ti{G}_E)$, there
are $((\epsilon_0,m_1),(\epsilon_1,m_2))\in C''(E)$ such that
$\Upsilon_{\DC}^{(2)}((\epsilon_0,m_1),(\epsilon_1,m_2)) =
((v_0,r_1),(v_1,r_2))$. We fix one element in the $G$-orbit formed by
those elements. Using Remark~\ref{rem_isomorphism}, given
$(((v_0,r_1),(v_1,r_2)),(\delta v_1,0)) \in
\ti{p_{34}}^{*}(\check{\mathcal{D}})$ there is a unique
$(((\epsilon_0,m_1),(\epsilon_1,m_2)),(\delta\epsilon_1,0))\in
\ti{p_{34}}^{*}(\mathcal{D})$ such that
$(\delta v_1,0) =
d\Upsilon_{\DC}(\epsilon_1,m_2)(\delta\epsilon_1,0)$. For the previous
choices, let
\begin{equation}\label{eq:P_check-def}
  \begin{split}
    \check{\mathcal{P}}((v_0,r_1),(v_1,r_2))&(\delta v_1,0) :=
    D_1(p_{1}\circ \Upsilon_{\DC})(\epsilon_0,m_1)
    (\mathcal{P}((\epsilon_0,m_1), (\epsilon_1,m_2))(\delta \epsilon_1,0)) \\
    & \phantom{(\delta v_1,0) :=} +D_2(p_{1}\circ
    \Upsilon_{\DC})(\epsilon_0,m_1)(d\phi(\epsilon_1)(\delta\epsilon_1))
    \\=& d(p_1 \circ \Upsilon_{\DC})(\epsilon_0,m_1)
    (\mathcal{P}((\epsilon_0,m_1),(\epsilon_1,m_2))(\delta
    \epsilon_1,0),d\phi(\epsilon_1)(\delta \epsilon_1)).
  \end{split}
\end{equation}

\begin{lemma}\label{le:P_reduced_is_well_defined}
  Under the previous conditions, the map $\check{\mathcal{P}}$ defined
  by~\eqref{eq:P_check-def} is a well defined element of
  $\hom(\ti{p_{34}}^*(\check{\mathcal{D}}),\ker(d p^{M/G}))$.
\end{lemma}

\begin{proof}
  The proof is similar to the proof of the Lemma 5.10
  of~\cite{ar:fernandez_tori_zuccalli-lagrangian_reduction_of_discrete_mechanical_systems_by_stages}
  with $\ti{p_{34}}^*(\check{\mathcal{D}})$ instead of
  $p_{3}^*(T\ti{G}_E)$ and taking into account the $G$-invariance of
  $\mathcal{D}$.
\end{proof}

\begin{definition}
  Let $G$ be a symmetry group of
  $\mathcal{M} = (E,L_d,\mathcal{D}_d,\mathcal{D},\mathcal{P})
  \in\Ob_{\DLDPSC}$ and $\DC$ be a discrete connection on the
  principal bundle $\pi^{M,G}:M\rightarrow M/G$. The system
  $(\ti{G}_E, \check{L}_d, \check{\mathcal{D}}_d, \check{\mathcal{D}},
  \check{\mathcal{P}}) \in \Ob_{\DLDPSC}$ defined previously is called
  the \jdef{reduced discrete Lagrange--D'Alembert--Poincar\'e system}
  obtained as the reduction of $\mathcal{M}$ by the symmetry group $G$
  using the discrete connection $\DC$. We denote this system by
  $\mathcal{M}/G$ or $\mathcal{M}/(G,\DC)$.
\end{definition}

\begin{example}
  Given a discrete nonholonomic mechanical system
  $(Q,L_d,\mathcal{D}_d,\mathcal{D}^{nh})$ let
  $\mathcal{M} := (Q,L_d,\mathcal{D}_d,\mathcal{D},0)$ be the discrete
  Lagrange--D'Alembert--Poincar\'e system constructed in
  Example~\ref{ex:DNHMS_as_DLDPS}. Let $G$ be a symmetry group of
  $(Q,L_d,\mathcal{D}_d,\mathcal{D}^{nh})$. As noted in
  Remark~\ref{remark:G_symmetry_DNHMS}, $G$ is a symmetry group of
  $\mathcal{M}$. Let $\DC$ be a discrete connection on the principal
  $G$-bundle $\pi^{Q,G}:Q\rightarrow Q/G$. The system
  $\mathcal{M}/(G,\DC)$ is
  $(\ti{G}_E, \check{L}_d, \check{\mathcal{D}}_d, \check{\mathcal{D}},
  \check{\mathcal{P}})$ where the fiber bundle
  $\phi:\ti{G}_E \rightarrow M/G$ is $p^{Q/G}:\ti{G}\rightarrow Q/G$,
  the lagrangian is determined by
  $\check{L}_d \circ \Upsilon_{\DC}=L_d$ and,
  by~\eqref{eq:P_check-def},
  \begin{equation*}
    \check{\mathcal{P}}((v_{k-1},r_k),(v_k,r_{k+1}))(\delta v_k,0) =
    D_2(p_1 \circ \Upsilon_{\DC})(q_{k-1},q_k)(\delta q_k),
  \end{equation*}
  where we have $(v_{k-1},r_k)=\Upsilon_{\DC}(q_{k-1},q_k)$,
  $(v_k,r_{k+1})=\Upsilon_{\DC}(q_k,q_{k+1})$ and
  $(\delta v_k,0)=d \Upsilon_{\DC}(q_k,q_{k+1})(\delta q_k,0)$, with
  $\delta q_k \in \mathcal{D}^{nh}_{q_k}$. This DLDPS coincides with
  the one obtained in
  Section~\ref{sec:nonholonomic_discrete_mechanical_systems_with_symmetry}
  as associated to the reduction of
  $(Q,L_d,\mathcal{D}_d,\mathcal{D}^{nh})$ modulo $G$ in the sense
  of~\cite{ar:fernandez_tori_zuccalli-lagrangian_reduction_of_discrete_mechanical_systems}.
  Thus, the reduction process of DLDPSs extends the reduction
  construction of discrete nonholonomic mechanical systems introduced
  in~\cite{ar:fernandez_tori_zuccalli-lagrangian_reduction_of_discrete_mechanical_systems}.
\end{example}

\begin{proposition}\label{prop:Upsilon_morphism}
  Let $G$ be a symmetry group of
  $\mathcal{M} = (E,L_d,\mathcal{D}_d,\mathcal{D},\mathcal{P})\in
  \Ob_{\DLDPSC}$ and $\DC$ a discrete connection on the principal
  $G$-bundle $\pi ^{M,G}:M\rightarrow M/G$. Then,
  $\Upsilon_{\DC}\in \Mor_{\DLDPSC}(\mathcal{M},\mathcal{M}/(G,\DC))$,
  where $\Upsilon_{\DC}$ is the map defined by
  \eqref{eq:Upsilon_DC-def}.
\end{proposition}

\begin{proof}
  The proof that $\Upsilon_{\DC}$ satisfies the
  conditions~\ref{it:prop_morphism_1} to~\ref{it:prop_morphism_4}
  and~\ref{it:prop_morphism_7} of Definition~\ref{def:category} is
  analogous to the proof of Proposition 5.13
  in~\cite{ar:fernandez_tori_zuccalli-lagrangian_reduction_of_discrete_mechanical_systems_by_stages},
  with $\ti{p_{34}}^*(\mathcal{D})$ instead of $p_{3}^{*}TE$.  Since
  the constraint spaces of the system $\mathcal{M}/(G,\DC)$ have been
  defined as $\check{\mathcal{D}}_d := \Upsilon_{\DC}(\mathcal{D}_d)$
  and $\check{\mathcal{D}} := d\Upsilon_{\DC}(\mathcal{D})$
  conditions~\ref{it:prop_morphism_5} and~\ref{it:prop_morphism_6} of
  Definition~\ref{def:category} hold.
\end{proof}

The following result proves that given a discrete
Lagrange--D'Alembert--Poincar\'e system with symmetry, the reduced
systems obtained using different discrete connections are all
isomorphic in $\DLDPSC$.

\begin{proposition}\label{prop:reduction_with_different_DC_are_isomorphic}
  Let $G$ be a symmetry group of
  $\mathcal{M} = (E,L_d,\mathcal{D}_d,\mathcal{D},\mathcal{P})
  \in\Ob_{\DLDPSC}$ and $\DCp{1}$, $\DCp{2}$ be two discrete connections
  on the principal $G$-bundle $\pi^{M,G}:M\rightarrow M/G$. Then, the
  reduced systems $\mathcal{M}/(G,\DCp{1})$ and
  $\mathcal{M}/(G,\DCp{2})$ are isomorphic in $\DLDPSC$.
\end{proposition}

\begin{proof}
  The proof is analogous to the proof of Proposition 5.14
  in~\cite{ar:fernandez_tori_zuccalli-lagrangian_reduction_of_discrete_mechanical_systems_by_stages},
  using Lemmas~\ref{le:Upsilon_fiber_bundle} and~\ref{le:morphism} and
  Proposition~\ref{prop:Upsilon_morphism}.
\end{proof}


\subsection{Dynamics of the reduced discrete
  Lagrange--D'Alembert--Poincar\'e system}
\label{sec:dynamics_of_the_reduced_DLDPS}

In this section we consider the dynamics of the reduced system defined
in Section~\ref{subsec:Reduced_DLP_system}.

\begin{theorem}\label{thm:equivalent_trajectories}
  Let $G$ be a symmetry group of
  $\mathcal{M} = (E,L_d,\mathcal{D}_d,\mathcal{D},\mathcal{P})$, $\DC$
  be a discrete connection on the principal $G$-bundle
  $\pi^{M,G}:M\rightarrow M/G$ and
  $\mathcal{M}/(G,\DC) = (\ti{G}_E,\check{L}_d, \check{\mathcal{D}}_d,
  \check{\mathcal{D}}, \check{\mathcal{P}})\in\Ob_{\DLDPSC}$ be the
  corresponding reduced DLDPS. Assume that
  $(\epsilon_{\cdot},m_{\cdot}) =
  ((\epsilon_0,m_1),\ldots,(\epsilon_{N-1},m_N))$ is a discrete path
  in $C'(E)$, and
  $(v_\cdot, r_\cdot) = ((v_0,r_1),\ldots,(v_{N-1},r_N))$ is a
  discrete path in $C'(\ti{G}_E)$ such that
  $\Upsilon_{\DC}(\epsilon_k,m_{k+1}) = (v_k,r_{k+1})$ for
  $k=0,\ldots,N-1$. Then, $(\epsilon_{\cdot},m_{\cdot})$ is a
  trajectory of $\mathcal{M}$ if and only if $(v_{\cdot},r_{\cdot})$
  is a trajectory of $\mathcal{M}/(G,\DC)$.
\end{theorem}

\begin{proof}
  By Proposition~\ref{prop:Upsilon_morphism},
  $\Upsilon_{\DC}\in\Mor_{\DLDPSC}(\mathcal{M},\mathcal{M}/(G,\DC))$
  and, by Lemma~\ref{le:check_D_submanifold},
  $\mathcal{D}_d = \Upsilon_{\DC}^{-1}(\check{\mathcal{D}}_d)$.  Then
  the result follows from Theorem~\ref{thm:imp_traject_categ}.
\end{proof}

\begin{corollary}\label{cor:4_pts}
  Let $G$ be a symmetry group of
  $\mathcal{M}=(E,L_d,\mathcal{D}_d,\mathcal{D},\mathcal{P})
  \in\Ob_{\DLDPSC}$, and $\DC$ be a discrete connection on the
  principal $G$-bundle $\pi^{M,G}:M\rightarrow M/G$. For the discrete
  path $(\epsilon_{\cdot},m_{\cdot})$ in $C'(E)$ we define a discrete
  path $(v_{\cdot},r_{\cdot})$ in $C'(\ti{G}_E)$ as
  $(v_k,r_{k+1}):=\Upsilon_{\DC}(\epsilon_k,m_{k+1})$ for
  $k=0,\ldots,N-1$. Then, the following statements are equivalent.
  \begin{enumerate}
  \item \label{it:4_pts-var} $(\epsilon_{\cdot},m_{\cdot})$ is a
    traejctory of the system $\mathcal{M}$.
  \item \label{it:4_pts-eq} Condition~\eqref{eq:evolution} is
    satisfied for $\nu_d:=\nu_d^\mathcal{M}$ defined
    by~\eqref{eq:section_of_motion-def} for $\mathcal{M}$.
  \item \label{it:4_pts-var_red} $(v_{\cdot},r_{\cdot})$ is a
    trajectory of the system $\mathcal{M}/(G,\DC)$.
  \item \label{it:4_pts-eq_red} Condition~\eqref{eq:evolution} is
    satisfied for $\nu_d:=\nu_d^{\mathcal{M}/(G,\DC)}$ and
    $\mathcal{D}_d:=\check{\mathcal{D}}_d$ being those of
      $\mathcal{M}/(G,\DC)$.
  \end{enumerate}
\end{corollary}

\begin{proof}
  The equivalence~\ref{it:4_pts-var} \iff~\ref{it:4_pts-eq} was
  demonstrated in Proposition~\ref{prop:dynamics_DLDPS}.  The
  equivalence ~\ref{it:4_pts-var_red} \iff~\ref{it:4_pts-eq_red}
  follows from Proposition~\ref{prop:dynamics_DLDPS} applied to the
  system $\mathcal{M}/(G,\DC)$. The equivalence~\ref{it:4_pts-var}
  \iff~\ref{it:4_pts-var_red} follows from
  Theorem~\ref{thm:equivalent_trajectories}.
\end{proof}

\begin{theorem}\label{thm:reconstruction}
  Let $G$ be a symmetry group of
  $\mathcal{M} =
  (E,L_d,\mathcal{D}_d,\mathcal{D},\mathcal{P})\in\Ob_{\DLDPSC}$ and
  $\DC$ be a discrete connection on the principal $G$-bundle
  $\pi^{M,G}:M\rightarrow M/G$. Let $(v_{\cdot},r_{\cdot})$ be a
  trajectory of the system $\mathcal{M}/(G,\DC)$ and
  $(\ti{\epsilon}_0,\ti{m}_1)\in \mathcal{D}_d$ such that
  $\Upsilon_{\DC}(\ti{\epsilon}_0,\ti{m}_1)=(v_0,r_1)$. Then, there
  exists a unique trajectory $(\epsilon_{\cdot},m_{\cdot})$ of
  $\mathcal{M}$ such that
  $(\epsilon_0,m_1) = (\ti{\epsilon}_0,\ti{m}_1)$ and
  $\Upsilon_{\DC}(\epsilon_k,m_{k+1})=(v_k,r_{k+1})$ for all $k$.
\end{theorem}

\begin{proof}
  By Proposition \ref{prop:path_relation}, the discrete path
  $(v_{\cdot},r_{\cdot})$ lifts to a unique discrete path
  $(\epsilon_{\cdot},m_{\cdot})$ in $C'(E)$ starting at
  $(\ti{\epsilon}_0,\ti{m}_1)$. Then,
  $(\epsilon_0,m_1) = (\ti{\epsilon}_0,\ti{m}_1)$ and
  $(v_k,r_{k+1})=\Upsilon_{\DC}(\epsilon_k,m_{k+1})$ for all $k$. As
  $(v_{\cdot},r_{\cdot})$ is a trajectory of $\mathcal{M}/(G,\DC)$, by
  Theorem~\ref{thm:equivalent_trajectories},
  $(\epsilon_{\cdot},m_{\cdot})$ is a trajectory of $\mathcal{M}$.
\end{proof}

\begin{remark}
  Theorem~\ref{thm:reconstruction} states that all trajectories of a
  reduced discrete Lagrange--D'Alembert--Poincar\'e system
  $\mathcal{M}/(G,\DC)$ come from trajectories of the original system
  $\mathcal{M}$. A direct description of the reconstruction process in
  terms of the lifting of discrete paths is given in Remark 5.18
  of~\cite{ar:fernandez_tori_zuccalli-lagrangian_reduction_of_discrete_mechanical_systems_by_stages}. In
  that respect, it should be kept in mind that, as
  $\mathcal{D}_d = \Upsilon_{\DC}^{-1}(\check{\mathcal{D}}_d)$ and the
  trajectories of $\mathcal{M}/(G,\DC)$ are in
  $\check{\mathcal{D}}_d$, the lifted discrete paths are in
  $\mathcal{D}_d$ automatically.
\end{remark}

Next we study the relationship between the equation of motion of a
symmetric DLDPS and that of its reduction. Before we can state the
result, we recall the pullback construction for sections of vector
bundles. If $\rho_j:\mathcal{V}_j\rightarrow X_j$ for $j=1,2$ are
smooth vector bundles and $F:\mathcal{V}_1\rightarrow \mathcal{V}_2$
is a morphism of vector bundles over $f:X_1\rightarrow X_2$, there is
a pullback map
$F^*:\Gamma(X_2,\mathcal{V}_2^*)\rightarrow
\Gamma(X_1,\mathcal{V}_1^*)$ determined by
$F^*(\alpha_2)(x_1)(v_1) := \alpha_2(f(x_1))(F(v_1))$, for
$\alpha_2\in \Gamma(X_2,\mathcal{V}_2^*)$, $x_1\in X_1$ and
$v_1\in (\mathcal{V}_1)_{x_1}$.

Let $G$ be a symmetry group of
$\mathcal{M} = (E,L_d,\mathcal{D}_d,\mathcal{D},\mathcal{P}) \in
\Ob_{\DLDPSC}$ and $\DC$ be a discrete connection on the principal
$G$-bundle $\pi^{M,G}:M\rightarrow M/G$. Let $\mathcal{M}/(G,\DC)$ be
the corresponding reduced system. We have the vector bundles
$\ti{p_{34}}^*(\mathcal{D})\rightarrow C''(E)$ and
$\ti{p_{34}}^*(\check{\mathcal{D}})\rightarrow C''(\ti{G}_E)$. Then,
as $d\Upsilon_{\DC}^{(2)}:TC''(E)\rightarrow TC''(\ti{G}_E)$ is a
morphism of vector bundles (over
$\Upsilon_{\DC}^{(2)}:C''(E)\rightarrow C''(\ti{G}_E)$) that restricts
to a morphism
$d\Upsilon_{\DC}^{(2)}|_{\ti{p_{34}}^*(\mathcal{D})}:
\ti{p_{34}}^*(\mathcal{D})\rightarrow
\ti{p_{34}}^*(\check{\mathcal{D}})$, we have the pullback map
$(d\Upsilon_{\DC}^{(2)}|_{\ti{p_{34}}^*(\mathcal{D})})^* :
\Gamma(C''(\ti{G}_E),\ti{p_{34}}^*(\check{\mathcal{D}})^*)\rightarrow
\Gamma(C''(E),\ti{p_{34}}^*(\mathcal{D})^*)$. The following result
relates the equation of motion of $\mathcal{M}$,
$\nu_d^{\mathcal{M}}\in \Gamma(C''(E),\ti{p_{34}}^*(\mathcal{D})^*)$,
to the one of $\mathcal{M}/(G,\DC)$,
$\nu_d^{\mathcal{M}/G}\in
\Gamma(C''(\ti{G}_E),\ti{p_{34}}^*(\check{\mathcal{D}})^*)$.

\begin{lemma}
  With the previous notation,
  $\nu_d^{\mathcal{M}} =
  (d\Upsilon_{\DC}^{(2)}|_{\ti{p_{34}}^*(\mathcal{D})})^*(\nu_d^{\mathcal{M}/G})$.
\end{lemma}

\begin{proof}
  For any $\mu := ((\epsilon_0,m_1),(\epsilon_1,m_2))\in C''(E)$ and
  $(\delta \epsilon_1,0) \in (\ti{p_{34}}^*(\mathcal{D}))_\mu$,
  we have
  \begin{equation*}
    \begin{split}
      (d\Upsilon_{\DC}^{(2)}|_{\ti{p_{34}}^*(\mathcal{D})})^*
      (\nu_d^{\mathcal{M}/G})(\mu)(\delta \epsilon_1,0) =
      \nu_d^{\mathcal{M}/G}(\Upsilon_{\DC}^{(2)}(\mu))
      (d\Upsilon_{\DC}(\epsilon_1,m_2)(\delta \epsilon_1,0)).
    \end{split}
  \end{equation*}
  Then, a direct computation shows that
  \begin{equation*}
    \nu_d^{\mathcal{M}/G}(\Upsilon_{\DC}^{(2)}(\mu))
    (d\Upsilon_{\DC}(\epsilon_1,m_2)(\delta \epsilon_1,0)) =
    \nu_d^{\mathcal{M}}(\mu)(\delta \epsilon_1,0),
  \end{equation*}
  proving the statement.
\end{proof}

Let
$\mathcal{M} = (E,L_d,\mathcal{D}_d,\mathcal{D},\mathcal{P}) \in
\Ob_{\DLDPSC}$ and $G$ be a symmetry group of $\mathcal{M}$.  Consider
the vertical bundle $\ker(T\pi^{E,G})\subset TE$ over $E$ and let
$\mathcal{V}:=p_1^*\ker(T\pi^{E,G})$, where $p_1:C'(E)\rightarrow E$
is the projection map. Assume that
$\mathcal{S}:=\mathcal{D}\cap \mathcal{V}$ is a vector bundle over
$C'(E)$ and that there is another vector bundle
$\mathcal{H}\rightarrow C'(E)$ such that
$\mathcal{D} = \mathcal{S}\oplus \mathcal{H}$; the vector bundle
$\mathcal{H}$ may be constructed using a (continuous) connection on
the principal bundle $\pi^{E,G}:E\rightarrow E/G$. Then, as the
section of motion $\nu_d^\mathcal{M}$ takes values in
$\ti{p_{34}}^*(\mathcal{D})^* \simeq
\ti{p_{34}}^*(\mathcal{S}^*)\oplus \ti{p_{34}}^*(\mathcal{H}^*)$ we
can decompose it as
$\nu_d^\mathcal{M} = ((\nu_d^\mathcal{M})^\mathcal{S},
(\nu_d^\mathcal{M})^\mathcal{H})$. Clearly, for any $\mu\in C''(E)$,
the condition $\nu_d^\mathcal{M}(\mu) = 0$ is equivalent to
\begin{equation*}
  (\nu_d^\mathcal{M})^\mathcal{S}(\mu) = 0 \stext{ and }
  (\nu_d^\mathcal{M})^\mathcal{H}(\mu)=0.
\end{equation*}
The last two conditions are usually called the vertical and horizontal
equations of motion.

\begin{definition}\label{def:J_d}
  Let
  $\mathcal{M} = (E,L_d,\mathcal{D}_d,\mathcal{D},\mathcal{P}) \in
  \Ob_{\DLDPSC}$ and $G$ be a symmetry group of $\mathcal{M}$.  We
  define the \jdef{discrete nonholonomic momentum map}
  \begin{equation*}
    J_d:C'(E)\rightarrow (\jgg^\mathcal{D})^* \stext{ by }
    J_d(\epsilon_0,m_1)(\xi) :=
    -D_1L_d(\epsilon_0,m_1)(\xi_E(\epsilon_0))
  \end{equation*}
  for $(\epsilon_0,m_1)\in C'(E)$ and
  $((\epsilon_0,m_1),\xi)\in \jgg^\mathcal{D} :=
  \{((\epsilon_0,m_1),\xi)\in C'(E)\times \jgg :
  (\xi_E(\epsilon_0),0)\in \mathcal{D}_{(\epsilon_0,m_1)}\}$, where
  $\jgg:=\lie{G}$; we assume that $\jgg^\mathcal{D}\rightarrow C'(E)$
  is a smooth vector bundle. We also define, for each
  $\conj{\xi}\in \Gamma(\jgg^\mathcal{D})$,
  $(J_d)_{\conj{\xi}}:C'(E)\rightarrow \R$ by
  $(J_d)_{\conj{\xi}}(\epsilon_0,m_1):=
  J_d(\epsilon_0,m_1)(\conj{\xi}(\epsilon_0,m_1))$.
\end{definition}

In the context of Definition~\ref{def:J_d} we have that
$L_d(l^E_g(\epsilon_0),l^M_g(m_1)) = L_d(\epsilon_0,m_1)$ for all
$g\in G$ and $(\epsilon_0,m_1)\in C'(E)$. Then, for each
$((\epsilon_0,m_1),\xi)\in C'(E)\times \jgg$,
\begin{equation*}
  \begin{split}
    0 =& \frac{d}{dt}\bigg|_{t=0} L_d(l^E_{\exp(t\xi)}(\epsilon_0),
    l^M_{\exp(t\xi)}(m_1)) \\=&
    D_1L_d(\epsilon_0,m_1)(\xi_E(\epsilon_0)) +
    D_2L_d(\epsilon_0,m_1)(\xi_M(m_1)),
  \end{split}
\end{equation*}
thus, for $((\epsilon_0,m_1),\xi)\in C'(E)\times \jgg^\mathcal{D}$,
\begin{equation}\label{eq:J_d_in_terms_of_D_2}
  J_d(\epsilon_0,m_1)(\xi) = -D_1L_d(\epsilon_0,m_1)(\xi_E(\epsilon_0)) =
  D_2L_d(\epsilon_0,m_1)(\xi_M(m_1)).
\end{equation}

By Proposition~\ref{prop:dynamics_DLDPS}, if
$(\epsilon_\cdot,m_\cdot)$ is a trajectory of $\mathcal{M}$
then~\eqref{eq:evolution} holds. Let
$\conj{\xi} \in \Gamma(\jgg^\mathcal{D})$, so that, for each
$k=1,\ldots,N-1$,
$\conj{\xi}(\epsilon_k,m_{k+1}) \in
\jgg^\mathcal{D}_{(\epsilon_k,m_{k+1})}$; then, for each such $k$, we
can evaluate the second condition in~\eqref{eq:evolution} at
$(\conj{\xi}(\epsilon_k,m_{k+1}))_E(\epsilon_k),0)\in
\mathcal{S}_{(\epsilon_k,m_{k+1})} \subset
\mathcal{D}_{(\epsilon_k,m_{k+1})}$ to obtain
\begin{equation*}
  \begin{split}
    -D_1L_d(\epsilon_k,&m_{k+1})((\conj{\xi}(\epsilon_k,m_{k+1}))_E(\epsilon_k))
    \\=& D_2L_d(\epsilon_{k-1},m_k) \circ
    \underbrace{d\phi(\epsilon_k)(\conj{\xi}(\epsilon_k,m_{k+1}))_E(\epsilon_k))}_{=(\conj{\xi}(\epsilon_k,m_{k+1}))_M(m_k)}
    \\&+ D_1L_d(\epsilon_{k-1},m_k) \circ
    \mathcal{P}((\epsilon_{k-1},m_k),(\epsilon_k,m_{k+1}))(\conj{\xi}(\epsilon_k,m_{k+1})_E(\epsilon_k),0).
  \end{split}
\end{equation*}
Plugging this last result into~\eqref{eq:J_d_in_terms_of_D_2}, we see
that
\begin{equation}\label{eq:J_d_evolution}
  \begin{split}
    (J_d)_{\conj{\xi}}(&\epsilon_k,m_{k+1}) =
    (J_d)_{\conj{\xi}}(\epsilon_{k-1},m_{k}) \\&+
    D_1L_d(\epsilon_{k-1},m_k) \circ
    \mathcal{P}((\epsilon_{k-1},m_k),(\epsilon_k,m_{k+1}))(\conj{\xi}(\epsilon_k,m_{k+1})_E(\epsilon_k),0)
  \end{split}
\end{equation}
for all $k=1,\ldots,N-1$. The following result completes the previous
discussion.

\begin{proposition}\label{prop:J_d_evolution}
  Let
  $\mathcal{M} = (E,L_d,\mathcal{D}_d,\mathcal{D},\mathcal{P}) \in
  \Ob_{\DLDPSC}$ and $G$ be a symmetry group of $\mathcal{M}$. Also,
  let $(\epsilon_\cdot, m_\cdot)$ be a discrete path in $C'(E)$. The
  following statements are equivalent.
  \begin{enumerate}
  \item \label{it:J_d_evolution-vert_eq} $(\epsilon_\cdot,m_\cdot)$
    satisfies
    $\nu_d^\mathcal{S}((\epsilon_{k-1},m_k),(\epsilon_k,m_{k+1})) = 0$
    for all $k=1,\ldots N-1$.
  \item \label{it:J_d_evolution-J_d} For all sections
    $\conj{\xi}\in \Gamma(\jgg^\mathcal{D})$ and $k=1,\ldots N-1$,
    equation~\eqref{eq:J_d_evolution} is satisfied.
  \end{enumerate}
\end{proposition}

\begin{proof}
  The previous argument leading to~\eqref{eq:J_d_evolution} shows
  that~\ref{it:J_d_evolution-vert_eq} \imp~\ref{it:J_d_evolution-J_d}
  is true. Conversely, assume that point~\ref{it:J_d_evolution-J_d} is
  valid. For $k=1,\ldots, N-1$ and any
  $((\epsilon_k,m_{k+1}),s)\in \mathcal{S}_{(\epsilon_{k},m_{k+1})}$,
  there is $\xi\in\jgg$ such that $s=\xi_E(\epsilon_k)$. Then, as we
  are assuming that $\jgg^\mathcal{D}$ is a smooth vector bundle,
  there is $\conj{\xi}\in \Gamma(\jgg^\mathcal{D})$ such that
  $\conj{\xi}(\epsilon_k,m_{k+1}) = \xi$. Using this particular
  $\conj{\xi}$, equation~\eqref{eq:J_d_evolution} leads to
  $\nu_d^\mathcal{S}((\epsilon_{k-1},m_k),(\epsilon_k,m_{k+1}))(s) =
  0$ and, eventually, to the validity of
  point~\ref{it:J_d_evolution-vert_eq}.
\end{proof}

The equation~\eqref{eq:J_d_evolution} is known as the \jdef{discrete
  nonholonomic momentum evolution equation} and has been considered,
for instance, in~\cite{ar:cortes_martinez-non_holonomic_integrators}
and~\cite{ar:fernandez_tori_zuccalli-lagrangian_reduction_of_discrete_mechanical_systems}.


\subsection{Discrete LL systems on Lie groups}
\label{sec:discrete_LL_systems_on_lie_groups}

In this section we review the notions of discrete and continuous LL
system on a Lie group $G$ and then show how, in the discrete case, LL
systems are examples of DLDPSs. We also show that their reduced and
``momentum description'' on $\lie{G}^*$ are also examples of DLDPSs in
a natural way and find their equations of motion, which agree with the
ones that appear in the literature. As a concrete case, we explore the
discrete Suslov system.

An LL system on the Lie group $G$ is a nonholonomic mechanical system
$(G,L,\mathcal{D})$ for whom $G$, acting on itself by left
multiplication, is a symmetry group. Such a system can be described
alternatively as a reduced system on $\jgg:=\lie{G}$ with reduced
lagrangian $\ell$ and constraint subspace $\jgd\subset\jgg$. Yet
another description, using a (reduced) Legendre transform, is as a
dynamical system on $\jgg^*$ satisfying the Euler--Poincar\'e--Suslov
equations (see, for
instance~\cite{bo:bloch-nonholonomic_mechanics_and_control}).

\begin{example}\label{ex:suslov_system-continuous}
  A well known example of this type of system, due to
  G. Suslov~\cite{bo:suslov-theoretical_mechanics}, is a model for a
  rigid body, with a fixed point and constrained so that one of the
  components of its angular velocity relative to the body frame
  vanishes. Explicitly, the configuration space is the Lie group
  $G:=SO(3)$, with Lagrangian
  $L(g,\dot{g}) := \frac{1}{2} \langle \I dL_{g^{-1}}(g)(\dot{g}),
  dL_{g^{-1}}(g)(\dot{g})\rangle$, where $L_g$ is the left
  multiplication by $g$ map in $G$, so that
  $dL_{g^{-1}}(g)(\dot{g})\in T_eSO(3) = \jgso(3) \simeq \R^3$ (with
  the Lie algebra operation given by the vector product $\times$),
  $\I$ is the inertia tensor of the body and
  $\langle \cdot,\cdot\rangle$ is the canonical inner product of
  $\R^3$. The nonholonomic constraint ${\mathcal D}$ is determined by
  the subspace $\jgd:=\{\omega\in\R^3:\omega_3=0\}$, requiring that
  ${\mathcal D}_g := dL_g(e)(\jgd)\subset T_gG$ for all $g\in G$. The
  dynamics of this nonholonomic system is completely determined by the
  Euler--Poincar\'e--Suslov equations in $\jgso(3)^*$ that, in terms
  of the angular momentum $M:=\I\omega$ are
  \begin{equation*}
    \begin{cases}
      \dot{M} = M\times (\I^{-1}M) +\lambda e_3,\\
      M\in \jgd^*:=\I(\jgd).
    \end{cases}
  \end{equation*}
\end{example}

A discrete analogue of the LL systems has been considered by
Yu. Fedorov and D. Zenkov
in~\cite{ar:fedorov_zenkov-discrete_nonholonomic_ll_systems_on_lie_groups}
and, also, by R. McLachlan and M. Perlmutter
in~\cite{ar:mclachlan_perlmutter-integrators_for_nonholonomic_mechanical_systems};
the purpose of this section is to show that all the discrete systems
that have been considered (reduced, non-reduced and on $\jgg^*$) can
be seen as DLDPS. A discrete LL system on the Lie group $G$ is a
discrete nonholonomic system $(G,L_d,\mathcal{D}_d,\mathcal{D}^{nh})$
for whom $G$, acting on itself by left multiplication, is a symmetry
group. As seen in Example~\ref{ex:DNHMS_as_DLDPS}, such a discrete
nonholonomic system can naturally be seen as a DLDPS
${\mathcal M}^{LL} = (E^{LL}, L_d^{LL}, {\mathcal D}_d^{LL}, {\mathcal
  D}^{LL}, {\mathcal P}^{LL})$ where the fiber bundle
$E^{LL}\rightarrow M^{LL}$ is $id_G:G\rightarrow G$ (with $G$ acting
by left multiplication on both $G$s), so that $C'(id_G) =G\times G$
while $L_d^{LL}=L_d$, ${\mathcal D}_d^{LL} = {\mathcal D}_d$,
${\mathcal P}^{LL}=0$ and
${\mathcal D}^{LL} = p_1^*({\mathcal D}^{nh})$ for
$p_1:G\times G\rightarrow G$ the projection onto the first factor.

\begin{example}\label{ex:suslov_system-M^LL}
  The discrete version of the Suslov system has been extensively
  studied
  in~\cite{ar:fedorov_zenkov-discrete_nonholonomic_ll_systems_on_lie_groups}
  and~\cite{ar:garciaNaranjo_jimenez-the_geometric_discretisation_of_the_suslov_problem}
  as a reduced discrete mechanical system on $SO(3)$ and as a discrete
  dynamical system obeying the discrete Euler--Lagrange--Suslov
  equations. Here we follow the notation
  of~\cite{ar:garciaNaranjo_jimenez-the_geometric_discretisation_of_the_suslov_problem}\footnote{We
    consider here only the system whose discrete Lagrangian originates
    in $\ell_d^{(1,\epsilon)}$ with $\epsilon=1$. Still, the analysis
    remains valid for arbitrary $\epsilon$ and, also, for
    $\ell_d^{(\infty,\epsilon)}$.}. This nonholonomic discrete
  mechanical system is defined on the space $G:=SO(3)$, with discrete
  Lagrangian $L_d(g_0,g_1):=-\tr(g_1 \J g_0^t)$, where
  \begin{equation*}
    \J:=\left(
      \begin{array}{ccc}
        \frac{1}{2} (I_{22}+I_{33}-I_{11}) & 0 & -I_{13}\\
        0 & \frac{1}{2}(I_{11}+I_{33}-I_{22}) & -I_{23}\\
        -I_{13} & -I_{23} & \frac{1}{2}(I_{11}+I_{22}-I_{33})
      \end{array}
    \right) 
  \end{equation*}
  is the \jdef{mass tensor} associated to the rigid body's inertia
  tensor $\I$ (which can be assumed to have component $I_{12}=0$). The
  infinitesimal variations are determined by the subspace
  \begin{equation*}
    \jgd :=\left\{\left(
        \begin{array}{ccc}
          0&-\omega_3&\omega_2\\
          \omega_3&0&-\omega_1\\
          -\omega_2&\omega_1&0
        \end{array}
      \right) \in\R^{3\times 3} : \omega_3=0\right\} \subset \jgso(3)
  \end{equation*}
  ---that corresponds to the subspace $\jgd\subset\R^3$ defined in
  Example~\ref{ex:suslov_system-continuous} under the isomorphism
  $\R^3\simeq \jgso(3)$--- through
  ${\mathcal D}^{nh}_g := dL_g(1)(\jgd) = \{g\omega\in\R^{3\times 3} :
  \omega \in\jgd\}$ for any $g\in G$. In order to define the discrete
  dynamical constraints, we recall the \jdef{Cayley transform}
  $\cay:\jgso(3)\rightarrow G$ defined by
  $\cay(\omega) := (1+\frac{\omega}{2})(1-\frac{\omega}{2})^{-1}$;
  then we define ${\mathcal S}_d := \cay(\jgd)$. An explicit
  parametrization of ${\mathcal S}_d$ is given in $(25)$
  of~\cite{ar:garciaNaranjo_jimenez-the_geometric_discretisation_of_the_suslov_problem},
  but it won't be necessary for our purposes. It suffices to say that
  $W\in G$ is in ${\mathcal S}_d$ if and only if its rotation axis is
  orthogonal to the $e_3$ axis in $\R^3$ and the angle of rotation is
  in $(-\pi,\pi)$. Then, the \jdef{discrete dynamical constraint} is
  ${\mathcal D}_d :=\{(g_0,g_1)\in G\times G:g_0^{-1}g_1 \in {\mathcal
    S}_d\}$. All together, $(G,L_d,{\mathcal D}_d,{\mathcal D}^{nh})$
  is a discrete nonholonomic mechanical system in the sense of
  Example~\ref{ex:DNHMS_as_DLDPS}. Hence, as seen in the same Example,
  ${\mathcal M}^{LL}:=(E^{LL},L_d^{LL},{\mathcal D}_d^{LL},{\mathcal
    D}^{LL},{\mathcal P}^{LL})$ is a DLDPS, where the fiber
  bundle $E^{LL}\rightarrow M^{LL}$ is $id_G:G\rightarrow G$,
  $L_d^{LL}:= L_d$, ${\mathcal D}_d^{LL}:= {\mathcal D}_d$,
  ${\mathcal D}^{LL} := p_1^*({\mathcal D}^{nh})$ and
  ${\mathcal P}^{LL}=0$. We conclude that the discrete Suslov system
  can be seen as a discrete Lagrange--D'Alembert--Poincar\'e system in
  a natural way.
\end{example}

Back in the setting of an arbitrary discrete $LL$ system, we are given
that $G$ is a symmetry group of the nonholonomic discrete mechanical
system $(G,L_d,\mathcal{D}_d,\mathcal{D}^{nh})$ thus, as noted in
Remark~\ref{remark:G_symmetry_DNHMS}, $G$ is also a symmetry group of
${\mathcal M}^{LL}$ in $\DLDPSC$. In order to reduce
${\mathcal M}^{LL}$ by $G$ we need a discrete affine connection on the
trivial principal bundle $G\rightarrow G/G =\{[e]\}$. It is easy to
check that $\DC(g_0,g_1):=g_1g_0^{-1}$ is such a connection and, for
completeness sake, we recall that, by
Proposition~\ref{prop:reduction_with_different_DC_are_isomorphic}, all
the reduced systems ${\mathcal M}^{LL}/G$ obtained using different
discrete connections\footnote{In this case it is easy to describe all
  possible affine discrete connections: their domain can be extended
  to $G\times G$ and have discrete connection form
  $\DCp{h}(g_0,g_1) = g_1h^{-1} g_0^{-1}$ for a fixed $h\in G$.} are
isomorphic in $\DLDPSC$. Thus, we have the reduced space
${\mathcal M}^r:={\mathcal M}^{LL}/(G,\DC)$. The total space of the
fiber bundle underlying ${\mathcal M}^r$ is $\ti{G}_G = (G\times G)/G$
where the $G$-action is $l^{G\times
  G}_g(g_0,g_1):=(gg_0,gg_1g^{-1})$. The corresponding reduction
morphism $\Upsilon_\DC:G\times G\rightarrow \ti{G}_G\times \{[e]\}$ is
given by
$\Upsilon_\DC(g_0,g_1) = (\pi^{G\times G,G}(g_0,g_1g_0^{-1}),[e])$.

It is easier to work with an isomorphic model of ${\mathcal M}^r$. Let
$\ti{\eta}:G\times G\rightarrow G$ be defined by
$\ti{\eta}(g_0,g_1):=g_0^{-1}g_1g_0$; as $\ti{\eta}$ is $G$-invariant
for the action $l_g^{G\times G}$ defined above it induces a smooth map
$\eta:(G\times G)/G\rightarrow G$ that turns out to be a
diffeomorphism. If we view $\eta$ as an isomorphism of the fiber
bundle $(G\times G)/G\rightarrow \{[e]\}$ onto $G\rightarrow\{e\}$ we
can use Lemma~\ref{le:DLDPS_induced_by_isomorphism_of_FB} to obtain
${\mathcal M}^\eta \in \Ob_{\DLDPSC}$ that is isomorphic to
${\mathcal M}^r$ in $\DLDPSC$. Explicitly,
${\mathcal M}^\eta = (E^{\eta},L_d^\eta,{\mathcal D}_d^\eta,{\mathcal
  D}^\eta,{\mathcal P}^\eta)$, where $E^\eta\rightarrow M^\eta$ is the
fiber bundle $G\rightarrow\{e\}$, $L_d^\eta = \ell_d$ ---where
$\ell_d(W) = L_d(e,W)$---,
${\mathcal D}_d^\eta={\mathcal S}_d\times\{e\}$ ---where
$W\in {\mathcal S}_d$ if and only if $(e,W)\in {\mathcal D}_d$---,
${\mathcal D}^\eta_W = dR_W(e)(\jgd)$ and
\begin{equation*}
  {\mathcal P}^\eta((W_0,e),(W_1,e))(\delta W_1,0) =
  dL_{W_0}(\delta h)
\end{equation*}
where
$(\delta W_1,0) = (-dR_{W_1}(e)(\delta h),0) \in \ti{p}_{34}({\mathcal
  D}^\eta)$. With this information, the evolution in
${\mathcal M}^\eta$ can be determined using the section $\nu_d^\eta$
defined in~\eqref{eq:section_of_motion-def}.

\begin{proposition}\label{prop:eta_evolution_of_LL_system}
  $W_k\in G$ for $k=0,\ldots,N-1$ is a trajectory of ${\mathcal M}^\eta$
  if and only if
  \begin{equation*}
    \begin{cases}
      W_k \in {\mathcal S}_d \stext{ for } k=0,\ldots,N-1,\\
      R_{W_{k+1}}^*(T_{W_{k+1}}\ell_d) -
      L_{W_{k}}^*(T_{W_{k-1}}\ell_d) \in \jgd^{\circ} \stext{ for }
      k=0,\ldots, N-2.
    \end{cases}
  \end{equation*}
\end{proposition}

\begin{proof}
  Indeed, by Proposition~\ref{prop:dynamics_DLDPS}, $W_k$ is a
  trajectory of ${\mathcal M}^\eta$ if and only if
  \begin{equation*}
    \begin{cases}
      W_k \in {\mathcal S}_d \stext{ for } k=0,\ldots,N-1,\\
      \nu_d^\eta((W_{k},e),(W_{k+1},e))({\mathcal
        D}^\eta_{(W_{k+1},e)}) =0 \stext{ for } k=0,\ldots, N-2.
    \end{cases}
  \end{equation*}
  As, in this case,
  \begin{equation*}
    \begin{split}
      \nu_d^\eta((W_{k},e),(W_{k+1},e))({\mathcal
        D}^\eta_{(W_{k+1},e)}) =&
      \nu_d^\eta((W_{k},e),(W_{k+1},e))(-dR_{W_{k+1}}(e)(\jgd)) \\=&
      (d\ell_d(W_{k+1})\circ dR_{\ti{g}_k}(e) - d\ell_d(W_k) \circ
      dL_{W_{k}}(e))(\jgd) \\=& (R_{W_{k+1}}^*(d\ell_d(W_{k+1})) -
      L_{W_{k}}^*(d\ell_d(W_k)))(\jgd),
    \end{split}
  \end{equation*}
  the statement follows.
\end{proof}

\begin{remark}
  When the discrete path $W_\cdot$ is the reduction of the discrete
  path $g_\cdot$, that is, when
  $W_k = p_1\circ \eta \circ \Upsilon_{\DC}(g_k,g_{k+1}) = g_k^{-1}
  g_{k+1}$, by Corollary~\ref{cor:4_pts}, $W_\cdot$ is a trajectory of
  ${\mathcal M}^\eta$ if and only if $g_\cdot$ is a trajectory of
  ${\mathcal M}^{LL}$ and, by Example~\ref{ex:DNHMS_as_DLDPS}, if and
  only if $g_\cdot$ is a trajectory of the discrete nonholonomic
  system $(G,L_d,{\mathcal D}_d,{\mathcal D}^{nh})$.
\end{remark}

\begin{remark}
  The result of Proposition~\ref{prop:eta_evolution_of_LL_system} is
  part (iv) of Theorem 3.2
  in~\cite{ar:fedorov_zenkov-discrete_nonholonomic_ll_systems_on_lie_groups}. The
  complete result can be read off Corollary~\ref{cor:4_pts} applied to
  ${\mathcal M}^{LL}$ and ${\mathcal M}^\eta$.
\end{remark}

\begin{example}\label{ex:suslov_system-M^eta}
  It is immediate that the discrete Suslov system described in
  Example~\ref{ex:suslov_system-M^LL} is an $LL$-system, so that
  ${\mathcal M}^{LL}\in\Ob_{\DLDPSC}$ can be reduced by $G=SO(3)$ and
  the connection $\DC(g_0,g_1):=g_1g_0^{-1}$, defining a reduced
  system ${\mathcal M}^r:={\mathcal M}^{LL}/(G,\DC)$. Just as it was
  described above, the diffeomorphism
  $\eta:(G\times G)/G\rightarrow G$ defined by
  $\eta(\pi^{G\times G,G}(g_0,g_1)):=g_0^{-1}g_1 g_0$ can be used to
  define ${\mathcal M}^\eta\in \Ob_{\DLDPSC}$ with the property that
  $\eta$ turns out to be an isomorphism between ${\mathcal M}^r$ and
  ${\mathcal M}^\eta$. Explicitly, the fiber bundle
  $E^{\eta}\rightarrow M^{\eta}$ is $G\rightarrow\{1\}$,
  $L_d^\eta(W)=\ell_d(W) = -\tr(\J W)$,
  ${\mathcal D}_d^\eta = {\mathcal S}_d\times\{1\} =
  \cay(\jgd)\times\{1\}$,
  ${\mathcal D}^\eta_{W,1}=dR_W(1)(\jgd) = \{\delta h W\in\R^{3\times
    3}:\delta h\in\jgd\}$ and, for any
  $\delta W_1\in {\mathcal D}^\eta_{W_1,1}$,
  \begin{equation*}
    {\mathcal P}^\eta((W_0,1),(W_0,1))(\delta W_1,0) =
    -W_0\delta W_1 W_1^{-1}.
  \end{equation*}
  This discrete system ${\mathcal M}^\eta$ provides the usual
  (reduced) description of the discrete Suslov system as a dynamical
  system on $SO(3)$. Specializing
  Proposition~\ref{prop:eta_evolution_of_LL_system} to the current
  setting, we have that a discrete path $W_\cdot$ in $G$ is a discrete
  trajectory of ${\mathcal M}^\eta$ if and only if
  \begin{equation*}
    \begin{cases}
      W_k \in {\mathcal S}_d = \cay(\jgd) \stext{ for all } k=0,\ldots\\
      -(W_{k+1}\J-\J W_{k+1}^t)+(\J W_{k}-W_{k}^t\J) \in \jgd^\perp
      \stext{ for all } k=0,\ldots,
    \end{cases}
  \end{equation*}
  where we identify $\jgso(3)^*$ with $\jgso(3)$ using the inner
  product $\langle A_1,A_2\rangle:=\frac{1}{2}\tr(A_1A_2^t)$; in
  particular, $\jgd^\circ$ corresponds to $\jgd^\perp$.
\end{example}

Just as in the continuous case, it is possible to give an alternative
model for ${\mathcal M}^\eta$ as a dynamical system in (a submanifold of)
$\jgg^*$. Recall that the ($-$) discrete Legendre transform of $L_d$
is the map $\F^-L_d:G\times G\rightarrow T^*G$ defined by
$\F^-L_d(g_0,g_1) := -D_1L_d(g_0,g_1)$. Using the trivialization
$\lambda:T^*G\rightarrow G\times \jgg^*$ defined by
$\lambda(\alpha_g):=(g,L_g^*(\alpha_g))$ we see that
\begin{equation*}
  \lambda(\F^-L_d)(g_0,g_1) = \lambda(-D_1L_d(g_0,g_1)) =
  (g_0,L_{g_0}^*(-D_1L_d(g_0,g_1))).
\end{equation*}
If we define $p:G\times G\rightarrow \jgg^*$ by
$p(g_0,g_1) := L_{g_0}^*(-D_1L_d(g_0,g_1))$, for any $\xi\in\jgg$ we
have
\begin{equation*}
  \begin{split}
    p(\xi) =& L_{g_0}^*(-D_1L_d(g_0,g_1))(\xi) =
    -D_1L_d(g_0,g_1)(dL_{g_0}(e)(\xi)) \\=& -\frac{d}{ds}\bigg|_{s=0}
    L_d(g_0\exp(s\xi),g_1) = -\frac{d}{ds}\bigg|_{s=0}
    \ell_d(\exp(-s\xi)g_0^{-1}g_1) \\=& d\ell_d(W_0)(T_eR_{W_0}(\xi))
    = R_{W_0}^*(d\ell_d(W_0))(\xi)
  \end{split}
\end{equation*}
for $W_0:=g_0^{-1}g_1$. Thus, we define the \jdef{reduced Legendre
  transform} ${\mathcal L}:G\rightarrow \jgg^*$ by
\begin{equation}
  \label{eq:reduced_discrete_legendre_transform-def}
  {\mathcal L}(W_0) := p =
  R_{W_0}^*(d\ell_d(W_0)).
\end{equation}

In what follows we assume that ${\mathcal L}$ is a
diffeomorphism\footnote{In fact, under the usual regularity conditions
  on $\ell_d$, ${\mathcal L}$ is a \emph{local} diffeomorphism and
  care must be taken, restricting the constructions to appropriate
  open subsets, where ${\mathcal L}$ is diffeomorphism.}. Then,
applying Lemma~\ref{le:DLDPS_induced_by_isomorphism_of_FB} to
${\mathcal L}$ and ${\mathcal M}^\eta$, we see that there is
${\mathcal M}^S\in \Ob_{\DLDPSC}$ such that ${\mathcal L}$ is an
isomorphism from ${\mathcal M}^\eta$ into ${\mathcal M}^S$. Explicitly,
$E^S\rightarrow M^S$ is $\jgg^*\rightarrow \{0\}$,
$L_d^S= \ell_d \circ {\mathcal L}^{-1}$,
${\mathcal D}_d^S = {\mathcal L}({\mathcal S}_d)$,
${\mathcal D}^S_p = d({\mathcal L},0)({\mathcal
  L}^{-1}(p),e)({\mathcal D}^\eta_{({\mathcal L}^{-1}(p),e)})
=R_{{\mathcal L}^{-1}(p)}^*(d{\mathcal L}({\mathcal
  L}^{-1}(p)))(\jgd)$ and
\begin{equation*}
  {\mathcal P}^S((p_{k-1},0)(p_k,0))(\delta p_k,0) =
  R_{{\mathcal L}^{-1}(p_{k-1})}^*(d{\mathcal L}({\mathcal L}^{-1}(p_{k-1})))
  (\delta h)
\end{equation*}
if
$\delta p_k = -R_{{\mathcal L}^{-1}(p_k)}^*(d{\mathcal L}({\mathcal
  L}^{-1}(p_k)))(\delta h)$.

A direct application of Proposition~\ref{prop:dynamics_DLDPS} to
${\mathcal M}^S\in\Ob_{\DLDPSC}$ leads to the following result.

\begin{proposition}\label{prop:EPS_equations}
  A discrete path $p_\cdot$ in $\jgg^*$ is a discrete trajectory of 
  ${\mathcal M}^S$ if and only if
  \begin{equation}\label{eq:euler_poincare_suslov}
    \begin{cases}
      p_k \in {\mathcal L}({\mathcal S}_d), \stext{ for } k=0,\ldots\\
      (p_{k+1} - \Ad_{{\mathcal L}^{-1}(p_{k})}^*(p_{k})) \in
      \jgd^\circ \stext{ for } k=0,\ldots,
    \end{cases}
  \end{equation}
  where $\Ad_g^* := L_g^*\circ R_{g^{-1}}^*$.
\end{proposition}

The system~\eqref{eq:euler_poincare_suslov} is known as the
\jdef{discrete Euler--Poincar\'e--Suslov equations} (see Theorem 3.3
in~\cite{ar:fedorov_zenkov-discrete_nonholonomic_ll_systems_on_lie_groups}).

\begin{proof}
  As ${\mathcal M}^S$ is constructed out of ${\mathcal M}^\eta$ using the
  diffeomorphism ${\mathcal L}$ and
  Lemma~\ref{le:DLDPS_induced_by_isomorphism_of_FB}, we know that
  \begin{equation*}
    \begin{split}
      \nu_d^S(({\mathcal L}(W_{k}),0),({\mathcal L}(W_{k+1}),0))
      (d{\mathcal L}(W_{k+1})&(\delta W_{k+1}),0) \\=&
      \nu_d^\eta((W_{k},e),(W_{k+1},e))(\delta W_{k+1}k,0).
    \end{split}
  \end{equation*}
  Using the computation of $\nu_d^\eta$ developed in the proof of
  Proposition~\ref{prop:eta_evolution_of_LL_system},
  \begin{equation*}
    \begin{split}
      \nu_d^S((p_{k},0),(p_{k+1},0))(&{\mathcal D}^S_{(p_{k+1},0)}) \\=&
      \nu_d^S(({\mathcal L}(W_{k}),0),({\mathcal
        L}(W_{k+1}),0))(d{\mathcal L}(W_{k+1},e)({\mathcal
        D}^\eta_{(W_{k+1},e)})) \\=&
      \nu_d^\eta((W_{k},e),(W_{k+1},e))({\mathcal
        D}^\eta_{(W_{k+1},e)}) \\=& (R_{W_{k+1}}^*(d\ell_d(W_{k+1})) -
      L_{W_{k}}^*(d\ell_d(W_k)))(\jgd) \\=& (p_{k+1} -
      \Ad_{W_{k}}^*(p_{k}))(\jgd) = (p_{k+1} - \Ad_{{\mathcal
          L}^{-1}(p_{k})}^*(p_{k}))(\jgd)
    \end{split}
  \end{equation*}
  As ${\mathcal M}^S\in\Ob_{\DLDPSC}$, the result follows from
  Proposition~\ref{prop:dynamics_DLDPS}.
\end{proof}

\begin{example}\label{ex:suslov_system-M^S}
  The discrete Legendre transform
  ${\mathcal L}:G\rightarrow\jgso(3)^*$ associated to the (reduced)
  discrete Suslov system ${\mathcal M}^\eta$ described in
  Example~\ref{ex:suslov_system-M^eta} can be easily computed
  using~\eqref{eq:reduced_discrete_legendre_transform-def}: for
  $W\in G$ and $\xi\in\jgso(3)$,
  \begin{equation*}
    \begin{split}
      {\mathcal L}(W)(\xi) =& \frac{d}{ds}\bigg|_{s=0} \ell_d(\exp(s\xi)W)
      = \frac{d}{ds}\bigg|_{s=0} -\tr(\J \exp(s\xi)W) = -\tr(\J \xi W)
      \\=& -\frac{1}{2} (\tr(\J \xi W) + \tr((\J \xi W)^t))
      = -\frac{1}{2} (\tr(W \J \xi) - \tr(\J W^t \xi)) \\=&
      \frac{1}{2} \tr((W \J -\J W^t) \xi^t) = \langle W \J -\J W^t,
      \xi\rangle,
    \end{split}
  \end{equation*}
  and we conclude that ${\mathcal L}(W) = W \J -\J W^t$. It is easy to
  check that ${\mathcal L}$ is a local diffeomorphism, but it is not
  globally injective (nor onto). Still, the previous arguments can be
  applied locally to a pair of domains where ${\mathcal L}$ restricts
  to a diffeomorphism. Hence, ${\mathcal L}$ can be used together with
  Lemma~\ref{le:DLDPS_induced_by_isomorphism_of_FB} to construct
  ${\mathcal M}^S\in \Ob_{\DLDPSC}$ that is isomorphic to
  ${\mathcal M}^\eta$. Proposition~\ref{prop:EPS_equations} provides
  the equations of motion for ${\mathcal M}^S$. Let $p_\cdot$ be a
  discrete path in $\jgso(3)^*\simeq \jgso(3)$, and define
  $W_k:={\mathcal L}^{-1}(p_k)\in G$ for all $k$. Then, by
  Proposition~\ref{prop:EPS_equations}, $p_\cdot$ is a trajectory of
  ${\mathcal M}^S$ if and only if
  \begin{equation*}
    \begin{cases}
      p_k \in {\mathcal L}(S_d) \stext{ for } k=0,\ldots,
      \stext{(equivalently)} W_k\in
      {\mathcal S}_d,\\
      p_{k+1} - \Ad_{W_{k}}(p_{k}) \in \jgd^\perp \stext{ for }
      k=0,\ldots
    \end{cases}
  \end{equation*}
  where, as before, $\jgd^\circ$ is identified with $\jgd^\perp$. As,
  for any $\xi\in\jgso(3)$,
  \begin{equation*}
    \begin{split}
      \Ad_{W_{k}}^*(p_{k})(\xi) =&
      p_{k}(\Ad_{W_{k}}(\xi)) = p_{k}(W_{k}\xi W_{k}^{-1})) \\=&
      \frac{1}{2}\tr(p_{k}(W_{k}\xi W_{k}^{-1})^t) =
      \frac{1}{2}\tr(
      W_{k}^tp_{k}W_{k} \xi^t) \\=& (W_{k}^tp_{k}W_{k})(\xi)
    \end{split}
  \end{equation*}
  we see that
  $p_{k+1}-\Ad_{W_{k}}^*(p_{k}) = p_{k+1}-W_{k}^tp_{k}W_{k}$. Then,
  the discrete Euler--Poincar\'e--Suslov
  equations~\eqref{eq:euler_poincare_suslov} for ${\mathcal M}^S$ are
  \begin{equation*}
    \begin{cases}
      p_k \in {\mathcal L}(S_d) \stext{ for } k=0,\ldots,\\
      p_{k+1}-W_{k}^tp_{k}W_{k} \in \jgd^\perp \stext{ for }
      k=0,\ldots.
    \end{cases}
  \end{equation*}
  This expression matches the ones given in (6.11)
  of~\cite{ar:fedorov_zenkov-discrete_nonholonomic_ll_systems_on_lie_groups}
  and (29)
  of~\cite{ar:garciaNaranjo_jimenez-the_geometric_discretisation_of_the_suslov_problem}.
\end{example}


\section{Reduction by two stages}
\label{sec:reduction_by_two_stages}

Having introduced a category of DLDPSs, a notion of symmetry group for
a DLDPS and a process of reduction for such symmetric objects that is
closed in the category, we study the problem of reduction by stages in
this section. In other words, when $G$ is a symmetry group of
$\mathcal{M}$ and $H$ is a subgroup of $G$ we want to compare the
result of the reduction $\mathcal{M}/G$ with that of the iterated
reduction $(\mathcal{M}/H)/(G/H)$ whenever possible.


\subsection{Residual symmetry}
\label{sec:residual_symmetry}

Let $G$ be a symmetry group of
$\mathcal{M} = (E,L_d,\mathcal{D}_d,\mathcal{D},\mathcal{P})\in
\Ob_{\DLDPSC}$ and $H\subset G$ a closed normal subgroup.  The
following result proves that $H$ is a symmetry group of $\mathcal{M}$.

\begin{proposition}\label{prop:H_symmetry_subgroup}
  Let $G$ be a symmetry group of $\mathcal{M}\in\Ob_{\DLDPSC}$. If
  $H\subset G$ is a closed Lie subgroup, then $H$ is a symmetry group
  of $\mathcal{M}$.
\end{proposition}

\begin{proof}
  By Lemma 5.7
  of~\cite{ar:fernandez_tori_zuccalli-lagrangian_reduction_of_discrete_mechanical_systems_by_stages}
  we have that if $G$ is a Lie group that acts on the fiber bundle
  $(E,M,\phi,F)$ and $H\subset G$ is a closed Lie subgroup, then $H$
  acts on the fiber bundle $(E,M,\phi,F)$ so that
  condition~\ref{it:G_symmetry_DLDPS-act_on_bundle} in
  Definition~\ref{def:G_symmetry_DLDPS} is valid. The remaining
  conditions follow from the fact that $G$ satisfies them and that $H$
  acts by the restriction of the corresponding actions of $G$.
\end{proof}

In what follows we consider the action of the group $G/H$ on the
system obtained after having reduced the symmetry of $H$. As a first
step we recall the statement of Lemma 7.1 in
\cite{ar:fernandez_tori_zuccalli-lagrangian_reduction_of_discrete_mechanical_systems_by_stages},
that establishes that $G/H$ acts on the fiber bundle obtained after
the first reduction stage.

\begin{lemma}\label{le:residual_action}
  Let $G$ be a Lie group that acts on the fiber bundle $(E,M,\phi,F)$
  and $H\subset G$ be a closed normal subgroup. Define the maps
  \begin{gather*}
    l_{\pi^{G,H}(g)}^{\ti{H}_E}(\pi^{E\times H,H}(\epsilon,w)) :=
    \pi^{E\times H,H}(l_{g}^{E}(\epsilon),l_{g}^{G}(w)),
    \\
    l_{\pi^{G,H}(g)}^{M/H}(\pi^{M,H}(m)) := \pi^{M,H}(l_{g}^{M}(m)).
  \end{gather*}
  Then $l^{\ti{H}_E}$, $l^{M/H}$ and the trivial right action on
  $F\times H$ define an action of $G/H$ on the fiber bundle
  $(\ti{H}_E,M/H,p^{M/H},F\times H)$.
\end{lemma}

As in Section~\ref{sec:Reduction_discrete_LDP_systems}, these actions
induce "diagonal" actions on $C'(\ti{H}_E)$, $C''(\ti{H}_E)$ through
definitions~\eqref{eq:G-action_C'} and~\eqref{eq:G-action_C''} and
"lifted" actions on the spaces $T\ti{H}_E$ and
$\ti{p_{34}}^{*}(TC'(\ti{H}_E))$ through
definitions~\eqref{eq:G-action_TE} and~\eqref{eq:G-action_p3(TE)}.

The reduction of $\mathcal{M}$ by $H$ requires the choice of a
discrete connection $\DCp{H}$ on the principal $H$-bundle
$\pi^{M,H}:M\rightarrow M/H$. It turns out that, under some conditions
on $\DCp{H}$ that we explore next, $G/H$ is a symmetry group of
$\check{\mathcal{M}} := \mathcal{M}/(H,\DCp{H})$.

\begin{lemma}\label{le:G-invariant_horiz_space_equivalence}
  Let $G$ be a Lie group that acts on $M$ by the action $l^{M}$ in
  such a way that $\pi^{M,G}:M\rightarrow M/G$ is a principal
  $G$-bundle. Assume that $H\subset G$ is a closed normal subgroup and
  that $\DCp{H}$ is a discrete connection on the principal $H$-bundle
  $\pi ^{M,H}:M\rightarrow M/H$, whose domain $\jgU$ is $G$-invariant
  by the diagonal action $l^{M\times M}$. Then, the following
  statements are equivalent.
  \begin{enumerate}
  \item For each $g\in G$ and $(m_0,m_1) \in\jgU$,
    \begin{equation*}
      \DCp{H}(l_{g}^{M}(m_0),l_{g}^{M}(m_1)) = g\DCp{H}(m_0,m_1)g^{-1}.
    \end{equation*}		
  \item The submanifold $Hor_{\DCp{H}}\subset M\times M$ is
    $G$-invariant by the action $l^{M\times M}$.
  \end{enumerate}
\end{lemma}

\begin{proof}
  The proof is analogous to the proof of Lemma 7.2 in
  \cite{ar:fernandez_tori_zuccalli-lagrangian_reduction_of_discrete_mechanical_systems_by_stages}.
\end{proof}

\begin{proposition}\label{prop:G/H-symmetry}
  Let $G$ be a symmetry group of
  $\mathcal{M} = (E, L_d, \mathcal{D}_d, \mathcal{D}, \mathcal{P})
  \in\Ob_{\DLDPSC}$ and $H\subset G$ be a closed normal
  subgroup. Choose a discrete connection $\DCp{H}$ on the principal
  $H$-bundle $\pi ^{M,H}:M\rightarrow M/H$ so that one of the
  conditions of Lemma \ref{le:G-invariant_horiz_space_equivalence}
  holds. Then, $G/H$ is a symmetry group of
  $\check{\mathcal{M}} := \mathcal{M}/(H,\DCp{H}) = (\ti{H}_E,
  \check{L}_d, \check{\mathcal{D}}_d, \check{\mathcal{D}},
  \check{\mathcal{P}})$.
\end{proposition}

\begin{proof}
  By Lemma~\ref{le:residual_action}, $G/H$ acts on the fiber bundle
  $(\ti{H}_E,M/H,p^{M/H},F\times H)$. Thus, we have the $G/H$-action
  $l^{C'(\ti{H}_E)}_{\pi^{G,H}(g)}(v_0,r_1) :=
  (l^{\ti{H}_E}_{\pi^{G,H}(g)}(v_0),l^{M/H}_{\pi^{G,H}(g)}(r_1))$. Unraveling
  the definitions, we have that, for $g\in G$,
  \begin{equation}
    \label{eq:Upsilon_equiv}
    \Upsilon_{\DCp{H}} \circ l^{C'(E)}_g =
    l^{C'(\ti{H}_E)}_{\pi ^{G,H}(g)}\circ \Upsilon_{\DCp{H}}.
  \end{equation}
  In addition, just as in the proof of Proposition 7.3
  in~\cite{ar:fernandez_tori_zuccalli-lagrangian_reduction_of_discrete_mechanical_systems_by_stages},
  $\check{L}_d : \ti{H}_E \times (M/H) \rightarrow \R$ is
  $l^{C'(\ti{H}_E)}$-invariant.
	
  As $\check{\mathcal{D}}_d :=\Upsilon_{\DCp{H}}(\mathcal{D}_d)$, for
  any $g\in G$, using~\eqref{eq:Upsilon_equiv} and the $G$-invariance of
  $\mathcal{D}_d$, we have that
  \begin{equation*}
    l^{C'(\ti{H}_E)}_{\pi^{G,H}(g)}(\check{\mathcal{D}}_d) =
    l^{C'(\ti{H}_E)}_{\pi^{G,H}(g)}(\Upsilon_{\DCp{H}}(\mathcal{D}_d)) =
    \Upsilon_{\DCp{H}}(l^{C'(E)}_g(\mathcal{D}_d)) =
    \Upsilon_{\DCp{H}}(\mathcal{D}_d) = \check{\mathcal{D}}_d,
  \end{equation*}
  so that $\check{\mathcal{D}}_d$ is $G/H$-invariant.
  
  Also, as $\check{\mathcal{D}}:=d\Upsilon_{\DCp{H}}(\mathcal{D})$,
  for any $g\in G$, using~\eqref{eq:Upsilon_equiv} and the
  $G$-invariance of $\mathcal{D}$, we have
  \begin{equation*}
    \begin{split}
      l^{TC'(\ti{H}_E)}_{\pi^{G,H}(g)}(\check{\mathcal{D}}) =&
      dl^{C'(\ti{H}_E)}_{\pi^{G,H}(g)}(d\Upsilon_{\DCp{H}}(\mathcal{D}))
      = d(l^{C'(\ti{H}_E)}_{\pi^{G,H}(g)}\circ
      \Upsilon_{\DCp{H}})(\mathcal{D}) \\=& d(\Upsilon_{\DCp{H}} \circ
      l^{C'(E)}_g)(\mathcal{D}) =
      d\Upsilon_{\DCp{H}}(dl^{C'(E)}_g(\mathcal{D})) \\=&
      d\Upsilon_{\DCp{H}}(l^{TC'(E)}_g(\mathcal{D})) =
      d\Upsilon_{\DCp{H}}(\mathcal{D}) = \check{\mathcal{D}},
    \end{split}
  \end{equation*}
  so that $\check{\mathcal{D}}$ is $G/H$-invariant.

  The proof of the $G/H$-invariance of $\check{\mathcal{P}}$ mimics
  the one given in Proposition 7.3
  of~\cite{ar:fernandez_tori_zuccalli-lagrangian_reduction_of_discrete_mechanical_systems_by_stages},
  adapted to the current context.
\end{proof}


\subsection{Comparison with reduction by the full symmetry group}
\label{sec:comparison_with_reduction_by_the_full_symmetry_group}

Let $G$ be a symmetry group of
$\mathcal{M} = (E,L_d,\mathcal{D}_d,\mathcal{D},\mathcal{P}) \in
\Ob_{\DLDPSC}$. Then, if we choose a discrete connection $\DCp{G}$ on
the principal $G$-bundle $\pi^{M,G}:M\rightarrow M/G$ we have the
reduced system
$\mathcal{M}^G := \mathcal{M}/(G,\DCp{G}) \in \Ob_{\DLDPSC}$. If
$H\subset G$ is a closed normal Lie subgroup, then $H$ is a symmetry
group of $\mathcal{M}$ (Proposition~\ref{prop:H_symmetry_subgroup})
and choosing a discrete connection $\DCp{H}$ on the principal $H$
bundle $\pi^{M,H}\rightarrow M/H$ we have the reduced system
$\mathcal{M}^H := \mathcal{M}/(H,\DCp{H})\in \Ob_{\DLDPSC}$. Last, if
$\DCp{H}$ satisfies either one of the conditions that appear in
Lemma~\ref{le:G-invariant_horiz_space_equivalence}, then $G/H$ is a
symmetry group of $\mathcal{M}^H$
(Proposition~\ref{prop:G/H-symmetry}). Then, choosing a discrete
connection $\DCp{G/H}$ on the principal $G/H$-bundle
$\pi^{M/H,G/G}:M/H\rightarrow (M/H)/(G/H)$, we have the discrete
system
$\mathcal{M}^{G/H} := \mathcal{M}^H/(G/H,\DCp{G/H})\in \Ob_{\DLDPSC}$.

The goal of this section is to prove that $\mathcal{M}^G$ and
$\mathcal{M}^{G/H}$ are isomorphic in $\DLDPSC$. The following diagram
depicts the relation between the different discrete
Lagrange--D'Alembert--Poincar\'e systems and morphisms.
\begin{equation*}
  \xymatrix{
    {} & {} & {\mathcal{M}} \ar[ddll]_{\Upsilon_{\DCp{G}}} 
    \ar[dr]^{\Upsilon_{\DCp{H}}} & {} & {}\\
    {} & {} & {} & {\mathcal{M}^H} \ar[dr]^{\Upsilon_{\DCp{G/H}}} & {}\\
    {\mathcal{M}^{G}} & {} & {} & {} & {\mathcal{M}^{G/H}}
  }
\end{equation*}
At the "geometric level" the corresponding manifolds and smooth maps
are
\begin{equation*}
  \xymatrix{
    {} & {} & {C'(E)} \ar[ddll]_{\Upsilon_{\DCp{G}}} 
    \ar[dr]^{\Upsilon_{\DCp{H}}} & {} & {}\\
    {} & {} & {} & {C'(\ti{H}_E)} \ar[dr]^{\Upsilon_{\DCp{G/H}}} & {}\\
    {C'(\ti{G}_E)} & {} & {} & {} & 
    {C'\big(\ti{G/H}_{\ti{H}_E}\big)}
  }
\end{equation*}

We can enlarge the previous diagram by adding the various
diffeomorphisms associated to a discrete connection and by taking into
account the diagram~\eqref{eq:diagram_ExM_to_reduced}.
\begin{equation}\label{eq:two_step_reduction-geometric_diagram}
  \xymatrix{
    {} & {} & {C'(E)} \ar@/_/[ddll]_{\Upsilon_{\DCp{G}}} 
    \ar[ddl]_(.7){\pi^{C'(E),G}} 
    \ar[d]^{\pi^{C'(E),H}} \ar@/^/[dr]^{\Upsilon_{\DCp{H}}} & {} & {}\\
    {} & {} & {\frac{C'(E)}{H}} \ar[d]^{\pi^{\frac{C'(E)}{H},G/H}} 
    \ar[dl]^(.45){F_1} \ar[r]^(.4){\Phi_{\DCp{H}}}_(.4){\sim} & 
    {C'(\ti{H}_E)} \ar[d]_(.25){\pi^{C'(\ti{H}_E),G/H}} 
    \ar[dr]^{\Upsilon_{\DCp{G/H}}} & {}\\
    {C'(\ti{G}_E)} & 
    {\frac{C'(E)}{G}} \ar[l]^(.4){\Phi_{\DCp{G}}}_(.4)\sim & 
    {\frac{\frac{C'(E)}{H}}{G/H}} \ar[l]^{F_2}_{\sim} 
    \ar[r]_{\widecheck{\Phi_{\DCp{H}}}}^(.45){\sim} & 
    {\frac{C'(\ti{H}_E)}{G/H}} \ar[r]_(.45){\Phi_{\DCp{G/H}}}^(.4)\sim & 
    {C'\big(\ti{G/H}_{\ti{H}_E}\big)}
  }
\end{equation}

The following result introduces the new functions that appear in
diagram \eqref{eq:two_step_reduction-geometric_diagram} and explores
their basic properties.

\begin{lemma}\label{le:two_step_diagram-properties}
  Under the previous conditions,
  \begin{enumerate}
  \item $\Phi_{\DCp{H}}:\frac{C'(E)}{H}\rightarrow C'(\ti{H}_{E})$
    (see Proposition~\ref{prop:equivariant_diffeomorphisms}) is a
    $G/H$-equivariant diffeomorphism. Then, it induces a smooth
    diffeomorphism
    $\widecheck{\Phi_{\DCp{H}}}:\frac{\frac{C'(E)}{H}}{G/H}\rightarrow
    \frac{C'(\ti{H}_{E})}{G/H}$.
  \item $\pi^{C'(E),G}:C'(E)\rightarrow \frac{C'(E)}{G}$ is a smooth
    $H$-invariant map. Then, it induces a smooth map
    $F_{1}:\frac{C'(E)}{H}\rightarrow \frac{C'(E)}{G}$.
  \item $F_{1}:\frac{C'(E)}{H}\rightarrow \frac{C'(E)}{G}$ is a smooth
    $G/H$-invariant map. Then, it induces a smooth map
    $F_{2}:\frac{\frac{C'(E)}{H}}{G/H}\rightarrow
    \frac{C'(E)}{G}$. Also, $F_{2}$ is a diffeomorphism.
  \item Diagram \eqref{eq:two_step_reduction-geometric_diagram} is
    commutative.
  \end{enumerate}
\end{lemma}

\begin{proof}
  The proof is the same as that of Lemma 7.5 in
  \cite{ar:fernandez_tori_zuccalli-lagrangian_reduction_of_discrete_mechanical_systems_by_stages}.
\end{proof}

\begin{theorem} \label{thm:F_diffeomorphism}
  Consider the description given at the beginning of this section. Let
  $F:C'(\ti{G/H}_{\ti{H}_E}) \rightarrow C'(\ti{G}_E)$ be definided by
  $F:=\Phi_{\DCp{G}}\circ F_{2} \circ(\widecheck{\Phi_{\DCp{H}}})^{-1}
  \circ (\Phi_{\DCp{G/H}})^{-1}$ (see diagram
  \eqref{eq:two_step_reduction-geometric_diagram}). Then, $F$ is an
  isomorphism in $\DLDPSC$.
\end{theorem}

\begin{proof}
  From Proposition~\ref{prop:equivariant_diffeomorphisms} both
  $\Phi_{\DCp{G}}$ and $\Phi_{\DCp{G/H}}$ are diffeomorphisms and, by
  Lemma~\ref{le:two_step_diagram-properties}, both
  $\widecheck{\Phi_{\DCp{H}}}$ and $F_2$ are diffeomorphisms. Thus,
  $F$ is a diffeomorphism. As $\Upsilon_{\DCp{H}}$ and
  $\Upsilon_{\DCp{G/H}}$ are morphisms in $\DLDPSC$, the same is true
  for $\Upsilon_{\DCp{G/H}} \circ \Upsilon_{\DCp{H}}$. Then, by
  Lemma~\ref{le:morphism} applied to $\Upsilon_{\DCp{G}}$,
  $\Upsilon_{\DCp{G/H}} \circ \Upsilon_{\DCp{H}}$ and $F$, we conclude
  that $F$ is an isomorphism in $\DLDPSC$
\end{proof}

\begin{theorem}\label{thm:4_pts_in_stages}
  Consider the description given at the beginning of this section.
  \begin{enumerate}
  \item \label{it:4_pts_in_stages-trajectories} Let
    $(\epsilon_{\cdot},m_{\cdot}) =
    ((\epsilon_0,m_1),\ldots,(\epsilon_{N-1},m_N))$ be a discrete path
    in $C'(E)$. For $k=0,\ldots,N-1$ we define the discrete paths
    $(v_{k}^{H},r_{k+1}^{H}):=\Upsilon_{\DCp{H}}(\epsilon_k,m_{k+1})$,
    $(v_{k}^{G},r_{k+1}^{G}):=\Upsilon_{\DCp{G}}(\epsilon_k,m_{k+1})$
    and
    $(v_{k}^{G/H},r_{k+1}^{G/H}):=\Upsilon_{\DCp{G/H}}(v_{k}^{H},r_{k+1}^{H})$
    in $C'(\ti{H}_{E})$, $C'(\ti{G}_{E})$ and
    $C'(\ti{G/H}_{\ti{H}_{E}})$ respectively.  Then, the following
    conditions are equivalent.
    \begin{enumerate}
    \item \label{it:4_pts_in_stages-1} $(\epsilon_{\cdot},m_{\cdot})$
      is a trajectory of $\mathcal{M}$.
    \item \label{it:4_pts_in_stages-G} $(v_{\cdot}^{G},r_{\cdot}^{G})$
      is a trajectory of $\mathcal{M}^{G}$.
    \item \label{it:4_pts_in_stages-H} $(v_{\cdot}^{H},r_{\cdot}^{H})$
      is a trajectory of $\mathcal{M}^{H}$.
    \item \label{it:4_pts_in_stages-G/H}
      $(v_{\cdot}^{G/H},r_{\cdot}^{G/H})$ is a trajectory of
      $\mathcal{M}^{G/H}$.
    \end{enumerate}
		
  \item \label{it:4_pts_in_stages-trajectories_F} Let
    $F:C'(\ti{G/H}_{\ti{H}_{E}}) \rightarrow C'(\ti{G}_{E})$ be the
    diffeomorphism defined in Theorem
    \ref{thm:F_diffeomorphism}. Then,
    $F(v_{k}^{G/H},r_{k+1}^{G/H})=(v_{k}^{G},r_{k+1}^{G})$ for all
    $k$.
		
  \item \label{it:4_pts_in_stages-isomorphism} The systems
    $\mathcal{M}^{G}$ and $\mathcal{M}^{G/H}$ are isomorphic in
    $\DLDPSC$.
  \end{enumerate}
\end{theorem}

\begin{proof}
  Point~\ref{it:4_pts_in_stages-trajectories} is verified by
  Theorem~\ref{thm:equivalent_trajectories}, while
  point~\ref{it:4_pts_in_stages-isomorphism} follows from
  Theorem~\ref{thm:F_diffeomorphism}. The following computation proves
  point~\ref{it:4_pts_in_stages-trajectories_F}.
  \begin{equation*}
    \begin{split}
      (v_{k}^{G},r_{k+1}^{G}) =&\Upsilon_{\DCp{G}}(\epsilon_k,m_{k+1})
      = (F\circ \Upsilon_{\DCp{G/H}}\circ
      \Upsilon_{\DCp{H}})(\epsilon_k,m_{k+1}) \\=& (F\circ
      \Upsilon_{\DCp{G/H}})(v_{k}^{H},r_{k+1}^{H}) =
      F(v_{k}^{G/H},r_{k+1}^{G/H}).
    \end{split}
  \end{equation*}
\end{proof}



\begin{thebibliography}{10}

\bibitem{bo:AM-mechanics}
  R. Abraham and J.~E. Marsden, \emph{Foundations of mechanics},
  Benjamin/Cummings Publishing Co. Inc. Advanced Book Program,
  Reading, Mass., 1978, Second edition, revised and enlarged, With the
  assistance of Tudor Ra{\c{t}}iu and Richard Cushman. \MR{MR515141
    (81e:58025)}

\bibitem{ar:arnlod-sur_la_geometrie_differentielle_des_groupes_de_lie_de_dimension_infinie_et_ses_applications_a_l'hydrodynamique_des_fluides_parfaits}
  V.~Arnold, \emph{Sur la g\'eom\'etrie diff\'erentielle des groupes
    de {L}ie de dimension infinie et ses applications \`a
    l'hydrodynamique des fluides parfaits}, Ann. Inst. Fourier
  (Grenoble) \textbf{16} (1966), no.~fasc. 1, 319--361. \MR{0202082
    (34 \#1956)}

\bibitem{bo:bloch-nonholonomic_mechanics_and_control}
  A.~M. Bloch, \emph{Nonholonomic mechanics and control},
  Interdisciplinary Applied Mathematics, vol.~24, Springer-Verlag, New
  York, 2003, With the collaboration of J. Baillieul, P. Crouch and
  J. Marsden, With scientific input from P. S. Krishnaprasad,
  R. M. Murray and D. Zenkov, Systems and Control. \MR{1978379
    (2004e:37099)}

\bibitem{ar:castrillonlopez_ratiu-reduction_of_principal_bundles_covariant_lagrange_poincare_equations}
  M.~Castrill{\'o}n~L{\'o}pez and T.~S. Ratiu, \emph{Reduction in
    principal bundles: covariant {L}agrange-{P}oincar\'e equations},
  Comm. Math. Phys.  \textbf{236} (2003), no.~2, 223--250. \MR{1981991
    (2004e:58024)}

\bibitem{ar:cendra_diaz-lagrange_dalembert_poincare_equations_by_several_stages}
  H. Cendra and V.~A. D\'{i}az,
  \emph{Lagrange-d'{A}lembert-{P}oincar\'{e} equations by several
    stages}, J.  Geom. Mech. \textbf{10} (2018), no.~1, 1--41, Also,
  \href{http://arXiv.org/abs/1406.7271}{{\tt
      arXiv:1406.7271}}. \MR{3808241}

\bibitem{ar:cendra_marsden_ratiu-geometric_mechanics_lagrangian_reduction_and_nonholonomic_systems}
  H. Cendra, J.~E. Marsden, and T.~S. Ratiu, \emph{Geometric
    mechanics, {L}agrangian reduction, and nonholonomic systems},
  Mathematics unlimited---2001 and beyond, Springer, Berlin, 2001,
  pp.~221--273.  \MR{MR1852159 (2002g:37067)}

\bibitem{bo:cendra_marsden_ratiu-lagrangian_reduction_by_stages}
  \bysame, \emph{Lagrangian reduction by stages}, Mem. Amer. Math. Soc.
  \textbf{152} (2001), no.~722, x+108. \MR{MR1840979 (2002c:37081)}

\bibitem{ar:cortes_martinez-non_holonomic_integrators}
  J.~Cort{\'e}s and S.~Mart{\'{\i}}nez, \emph{Non-holonomic integrators},
  Nonlinearity \textbf{14} (2001), no.~5, 1365--1392. \MR{MR1862825
  (2002h:37165)}

\bibitem{bo:cortes-non_holonomic}
  J. Cort{\'e}s~Monforte, \emph{Geometric, control and numerical aspects of
  nonholonomic systems}, Lecture Notes in Mathematics, vol. 1793,
  Springer-Verlag, Berlin, 2002. \MR{MR1942617 (2003k:70013)}

\bibitem{ar:deDiego_deAlmagro-variational_order_for_forced_lagrangian_systems}
  D.~Mart\'{\i}n de~Diego and R.~Sato~Mart\'{\i}n de~Almagro, \emph{Variational
  order for forced {L}agrangian systems}, Nonlinearity \textbf{31} (2018),
  no.~8, 3814--3846. \MR{3826116}

\bibitem{ar:fedorov_zenkov-discrete_nonholonomic_ll_systems_on_lie_groups}
  Yu.~N. Fedorov and D.~V. Zenkov, \emph{Discrete nonholonomic {LL}
    systems on {L}ie groups}, Nonlinearity \textbf{18} (2005), no.~5,
  2211--2241.  \MR{MR2164739 (2006d:37118)}

\bibitem{ar:fernandez_graiffZurita_grillo-error_analysis_of_forced_discrete_mechanical_systems}
  J.~Fern{\'a}ndez, S.~Graiff~Zurita, and S.~Grillo, \emph{Error
    analysis of forced discrete mechanical systems}, in preparation.

\bibitem{ar:fernandez_tori_zuccalli-lagrangian_reduction_of_discrete_mechanical_systems}
  J. Fern{\'a}ndez, C. Tori, and M. Zuccalli, \emph{Lagrangian
    reduction of nonholonomic discrete mechanical systems},
  J. Geom. Mech.  \textbf{2} (2010), no.~1, 69--111, Also,
  \href{http://arXiv.org/abs/1004.4288}{{\tt
      arXiv:1004.4288}}. \MR{2646536}

\bibitem{ar:fernandez_tori_zuccalli-lagrangian_reduction_of_discrete_mechanical_systems_by_stages}
  \bysame, \emph{Lagrangian reduction of discrete mechanical systems
    by stages}, J. Geom. Mech. \textbf{8} (2016), no.~1, 35--70, Also,
  \href{http://arXiv.org/abs/1511.06682}{{\tt arXiv:1511.06682
      [math.DG]}}.  \MR{3485921}

\bibitem{ar:fernandez_zuccalli-a_geometric_approach_to_discrete_connections_on_principal_bundles}
  J. Fern{\'a}ndez and M. Zuccalli, \emph{A geometric approach to
    discrete connections on principal bundles},
  J. Geom. Mech. \textbf{5} (2013), no.~4, 433--444, Also,
  \href{http://arXiv.org/abs/1311.0260}{{\tt arXiv:1311.0260
      [math.DG]}}. \MR{3180706}

\bibitem{ar:garciaNaranjo_jimenez-the_geometric_discretisation_of_the_suslov_problem}
  L.~C. Garc\'{\i}a-Naranjo and F. Jim\'{e}nez, \emph{The geometric
    discretisation of the {S}uslov problem: a case study of
    consistency for nonholonomic integrators}, Discrete
  Contin. Dyn. Syst. \textbf{37} (2017), no.~8,
  4249--4275. \MR{3642264}

\bibitem{bo:goldstein-classical_mechanics}
  H. Goldstein, \emph{Classical mechanics}, second ed., Addison-Wesley
  Publishing Co., Reading, Mass., 1980, Addison-Wesley Series in Physics.
  \MR{575343 (81j:70001)}

\bibitem{bo:hairer_lubich_wanner-geometric_numerical_integration}
  E. Hairer, C. Lubich, and G. Wanner, \emph{Geometric numerical
    integration}, second ed., Springer Series in Computational
  Mathematics, vol.~31, Springer-Verlag, Berlin, 2006,
  Structure-preserving algorithms for ordinary differential
  equations. \MR{MR2221614 (2006m:65006)}

\bibitem{ar:iglesias_marrero_deDiego_martinez-discrete_nonholonomic_lagrangian_systems_on_lie_groupoids}
  D. Iglesias, J.~C. Marrero, D. Mart\'{\i}n de~Diego, and E.
  Mart\'{\i}nez, \emph{Discrete nonholonomic {L}agrangian systems on
    {L}ie groupoids}, J. Nonlinear Sci. \textbf{18} (2008), no.~3,
  221--276.  \MR{2411379}

\bibitem{ar:iglesias_marrero_martin_martinez_padron-reduction_of_symplectic_lie_algebroids_by_a_lie_subalgebroid_and_a_symmetry_lie_group}
  D. Iglesias, J.~C. Marrero, D. Mart{\'{\i}}n~de Diego, E.
  Mart{\'{\i}}nez, and E. Padr{\'o}n, \emph{Reduction of symplectic
    {L}ie algebroids by a {L}ie subalgebroid and a symmetry {L}ie
    group}, SIGMA Symmetry Integrability Geom. Methods
  Appl. \textbf{3} (2007), Paper 049, 28.  \MR{2299850 (2008g:53103)}

\bibitem{ar:jalnapurkar_leok_marsden_west-discrete_routh_reducion}
  S.~M. Jalnapurkar, M. Leok, J.~E. Marsden, and M. West,
  \emph{Discrete {R}outh reduction}, J. Phys. A \textbf{39} (2006),
  no.~19, 5521--5544. \MR{MR2220774 (2007g:37038)}

\bibitem{bo:lee-introduction_to_smooth_manifolds}
  J.~M. Lee, \emph{Introduction to smooth manifolds}, Graduate Texts in
  Mathematics, vol. 218, Springer-Verlag, New York, 2003. \MR{1930091
  (2003k:58001)}

\bibitem{ar:marrero_martin_martinez-discrete_lagrangian_and_hamiltonian_mechanics_on_lie_groupoids}
  J.~C. Marrero, D. Mart{\'{\i}}n~de Diego, and E.  Mart{\'{\i}}nez,
  \emph{Discrete {L}agrangian and {H}amiltonian mechanics on {L}ie
    groupoids}, Nonlinearity \textbf{19} (2006), no.~6,
  1313--1348. \MR{2230001 (2007c:37068)}

\bibitem{ar:marsden_west-discrete_mechanics_and_variational_integrators}
  J.~E. Marsden and M.~West, \emph{Discrete mechanics and variational
  integrators}, Acta Numer. \textbf{10} (2001), 357--514. \MR{MR2009697
  (2004h:37130)}

\bibitem{ar:marsden_weinstein-reduction_of_symplectic_manifolds_with_symmetry}
  J.~E. Marsden and A. Weinstein, \emph{Reduction of symplectic
    manifolds with symmetry}, Rep. Mathematical Phys. \textbf{5}
  (1974), no.~1, 121--130.  \MR{0402819 (53 \#6633)}

\bibitem{bo:marsden_misiolek_ortega_perlmutter_ratiu-hamiltonian_reduction_by_stages}
  J.~E. Marsden, G. Misio{\l}ek, J.-P. Ortega, M. Perlmutter, and
  T.~S. Ratiu, \emph{Hamiltonian reduction by stages}, Lecture Notes
  in Mathematics, vol. 1913, Springer, Berlin, 2007. \MR{MR2337886
    (2008i:37112)}

\bibitem{ar:mclachlan_perlmutter-integrators_for_nonholonomic_mechanical_systems}
  R.~McLachlan and M.~Perlmutter, \emph{Integrators for nonholonomic mechanical
  systems}, J. Nonlinear Sci. \textbf{16} (2006), no.~4, 283--328.
  \MR{MR2254707 (2008d:37154)}

\bibitem{ar:meyer-symmetries_and_integrals_in_mechanics}
  K.~R. Meyer, \emph{Symmetries and integrals in mechanics}, Dynamical
  systems ({P}roc. {S}ympos., {U}niv. {B}ahia, {S}alvador, 1971),
  Academic Press, New York, 1973, pp.~259--272. \MR{0331427 (48
    \#9760)}

\bibitem{ar:patrick_cuell-error_analysis_of_variational_integrators_of_unconstrained_lagrangian_systems}
  G.~W. Patrick and C. Cuell, \emph{Error analysis of variational
    integrators of unconstrained {L}agrangian systems},
  Numer. Math. \textbf{113} (2009), no.~2, 243--264. \MR{2529508
    (2010f:37146)}

\bibitem{ar:smale-topology_and_mechanics_1}
  S.~Smale, \emph{Topology and mechanics. {I}},
  Invent. Math. \textbf{10} (1970), 305--331. \MR{0309153 (46 \#8263)}

\bibitem{ar:smale-topology_and_mechanics_2}
  \bysame, \emph{Topology and mechanics. {II}. {T}he planar {$n$}-body problem},
  Invent. Math. \textbf{11} (1970), 45--64. \MR{0321138 (47 \#9671)}

\bibitem{bo:suslov-theoretical_mechanics}
  G.~K. Suslov, \emph{Theoretical mechanics}, Gostekhizdat, Moscow, 1946.

\end{thebibliography}


\def\cprime{$'$} \def\polhk#1{\setbox0=\hbox{#1}{\ooalign{\hidewidth
  \lower1.5ex\hbox{`}\hidewidth\crcr\unhbox0}}} \def\cprime{$'$}
  \def\cprime{$'$}
\providecommand{\bysame}{\leavevmode\hbox to3em{\hrulefill}\thinspace}
\providecommand{\MR}{\relax\ifhmode\unskip\space\fi MR }
\providecommand{\MRhref}[2]{%
  \href{http://www.ams.org/mathscinet-getitem?mr=#1}{#2}
}
\providecommand{\href}[2]{#2}


\end{document}